\documentclass[10pt,reqno]{amsart}
\usepackage[margin=1in]{geometry}
\usepackage{bbm}
\usepackage{amsfonts}
\usepackage{latexsym, amssymb, amsmath, amscd, amsthm, amsxtra}
\usepackage{mathtools}
\usepackage{enumerate}
\usepackage[all]{xy}
\usepackage{mathrsfs}
\usepackage{fancyhdr}
\usepackage{listings}
\usepackage{hyperref}
\usepackage{cleveref}
\usepackage{soul}
\usepackage{xcolor}
\usepackage{dsfont}
\usepackage{stmaryrd}
\usepackage{bm}
\usepackage{comment}
\usepackage{diagbox}
\usepackage{verbatim}
\usepackage{stmaryrd}

\allowdisplaybreaks[4]
\hypersetup{colorlinks=true}
\usepackage{graphicx, color}
\usepackage{cite}

\numberwithin{equation}{section}
\numberwithin{figure}{section}

\makeatletter
\newcommand{\rmnum}[1]{\uppercase{{\expandafter{\romannumeral #1}}}}
\makeatother

\theoremstyle{plain}
\newtheorem{theorem}{Theorem}[section]
\newtheorem*{theorem*}{Theorem}
\newtheorem{lemma}[theorem]{Lemma}
\newtheorem*{lemma*}{Lemma}

\newtheorem*{corollary*}{Corollary}

\newtheorem*{proposition*}{Proposition}
\newtheorem{definition}[theorem]{Definition}
\newtheorem*{definition*}{Definition}

\newtheorem*{conjecture*}{Conjecture}

\theoremstyle{definition}
\newtheorem{example}[theorem]{Example}
\newtheorem*{example*}{Example}

\newtheorem*{remark*}{Remark}
\newtheorem{claim}{Claim}

\newcommand{\cal}{\mathcal}
\newcommand{\scr}{\mathscr}

\renewcommand{\ul}[1]{\underline{#1} \!\,}
\newcommand{\ol}[1]{\overline{#1} \!\,}
\newcommand{\wh}{\widehat}
\newcommand{\wt}{\widetilde}
\renewcommand{\txt}[1]{\text{\rm{#1}}}

\definecolor{darkred}{rgb}{0.9,0,0.3}
\definecolor{darkblue}{rgb}{0,0.3,0.9}

\usepackage{ifthen}
\def\comment#1{\ifthenelse{\isodd{\value{page}}}{\marginpar{\raggedright\scriptsize{\textcolor{darkred}{#1}}}}{\marginpar{\raggedleft\scriptsize{\textcolor{darkred}{#1}}}}}

\renewcommand{\P}{\mathbb{P}}
\newcommand{\E}{\mathbb{E}}
\newcommand{\R}{\mathbb{R}}
\newcommand{\C}{\mathbb{C}}
\newcommand{\N}{\mathbb{N}}
\newcommand{\Z}{\mathbb{Z}}

\newcommand{\cW}{{\mathcal W}}
\newcommand{\cT}{\mathcal T}
\newcommand{\cI}{\mathcal{I}}
\newcommand{\cC}{\mathcal{C}}
\newcommand{\cR}{{\mathcal R}}

\newcommand{\cS}{{\mathcal{S}}}
\newcommand{\cB}{{\mathcal{B}}}

\newcommand{\cL}{{\mathcal{L}}}
\newcommand{\cK}{{\mathcal K}}
\newcommand{\cQ}{{\mathcal{Q}}}
\newcommand{\cJ}{{\mathcal{J}}}

\newcommand{\fc}{\mathfrak c}
\newcommand{\fd}{\mathfrak d}

\newcommand{\fn}{\mathfrak n}

\newcommand{\fS}{\mathfrak S}

\newcommand{\fB}{{\mathfrak{B}}}

\newcommand{\sE}{\mathsf E}

\newcommand{\bxi}{\boldsymbol{\xi}}

\newcommand{\ii}{\mathrm{i}}
\newcommand{\dd}{\mathrm{d}}

\renewcommand{\leq}{\leqslant}
\renewcommand{\geq}{\geqslant}
\renewcommand{\epsilon}{\varepsilon}

\newcommand{\qq}[1]{[\![{#1}]\!]}

\newcommand{\p}[1]{({#1})}

\newcommand{\pa}[1]{\left({#1}\right)}

\newcommand{\q}[1]{[{#1}]}

\newcommand{\qa}[1]{\left[{#1}\right]}

\newcommand{\h}[1]{\{{#1}\}}

\newcommand{\ha}[1]{\left\{{#1}\right\}}

\newcommand{\abs}[1]{\lvert #1 \rvert}

\newcommand{\absa}[1]{\left\lvert #1 \right\rvert}

\newcommand{\norm}[1]{\lVert #1 \rVert}

\newcommand{\norma}[1]{\left\lVert #1 \right\rVert}

\newcommand{\avg}[1]{\langle #1 \rangle}

\newcommand{\avga}[1]{\left\langle #1 \right\rangle}

\DeclareMathOperator{\tr}{Tr}

\DeclareMathOperator{\re}{Re}
\DeclareMathOperator{\im}{Im}

\newcommand{\ba}{{\bf{a}}}
\newcommand{\bx}{{\bf{x}}}

\newcommand{\bu}{{\bf{u}}}
\newcommand{\bv}{{\bf{v}}}

\newcommand{\be}{\begin{equation}}
\newcommand{\ee}{\end{equation}}

\newcommand{\e}{{\varepsilon}}

\newcommand{\rd}{\mathrm{d}}

\newcommand{\oo}{\mathrm{o}}
\newcommand{\OO}{\mathrm{O}}

\newcommand{\bsigma}{\bm{\sigma}}

\newcommand{\cut}{\mathrm{Cut}}
\newcommand{\cutL}{(\mathrm{Cut}_L)}
\newcommand{\cutR}{(\mathrm{Cut}_R)}
\newcommand{\Zn}{\widetilde{\mathbb Z}_n^d}
\newcommand{\ZL}{{\mathbb Z}_L^d}
\newcommand{\Gc}{{\mathring G}}
\newcommand{\bsig}{\bsigma}
\newcommand{\Kgen}{\wh{\cK}}
\newcommand{\OK}{\mathcal O_{\mathcal K}}
\newcommand{\lenk}{l_{\mathcal K}}

\newcommand{\SE}{{S_{\txt{B}}}}
\newcommand{\SRBM}{{S^{\txt{RBM}}}}

\newcommand{\ppp}[1]{{\llparenthesis #1\rrparenthesis}}
\newcommand{\ti}{t_{\txt{i}}}
\newcommand{\tf}{t_{\txt{f}}}

\renewcommand{\leq}{\le}
\renewcommand{\geq}{\ge}

\newcommand{\opr}[1]{\mathrm{O}_\prec\left({#1}\right)}
\allowdisplaybreaks

\numberwithin{equation}{section}
\usepackage{environ}

\begin{document}

\title{A Block Reduction Method for Random Band Matrices with General Variance Profiles}

\author{Jiaqi Fan$^\star$}
\thanks{$^\star$Qiuzhen College, Tsinghua University, Beijing, China, \href{mailto:fanjq24@mails.tsinghua.edu.cn}{fanjq24@mails.tsinghua.edu.cn}}

\author{Fan Yang$^\ddagger$}
\thanks{$^\ddagger$Yau Mathematical Sciences Center, Tsinghua University, and Beijing Institute of Mathematical Sciences and Applications, Beijing, China,
\href{mailto:fyangmath@mail.tsinghua.edu.cn}{fyangmath@mail.tsinghua.edu.cn}}

\author{Jun Yin$^\S$}
\thanks{$^\S$Department of Mathematics, University of California, Los Angeles, Los Angeles, CA, USA,
\href{mailto:jyin@math.ucla.edu}{jyin@math.ucla.edu}}

\begin{abstract}
We present a novel block reduction method for the study of a general class of random band matrices (RBM) defined on the $d$-dimensional lattice $\ZL:=\{1,2,\ldots,L\}^{d}$ for $d\in \{1,2\}$,  with band width $W$ and almost arbitrary variance profiles subject to a core condition. We prove the delocalization of bulk eigenvectors for such RBMs under the assumptions $W\geq L^{1/2+\varepsilon}$ in one dimension and $W\geq L^{\varepsilon}$ in two dimensions, where $\varepsilon>0$ is an arbitrarily small constant. This result extends the findings of \cite{Band1D,Band2D} on block RBMs to models with general variance profiles.
Furthermore, we generalize our results to Wegner orbital models with small interaction strength $\lambda\ll 1$. Under the sharp condition $\lambda\gg W^{-d/2}$, we establish lower bounds for the localization lengths of bulk eigenvectors, thereby extending the results of \cite{truong2025localizationlengthfinitevolumerandom} to settings with nearly arbitrary potential and hopping terms.
Our block reduction method provides a powerful and flexible framework that reduces both the dynamical analysis of the loop hierarchy and the derivation of deterministic estimates for general RBMs to the corresponding analysis of block RBMs, as developed in \cite{Band1D,Band2D,truong2025localizationlengthfinitevolumerandom}.

\end{abstract}

\maketitle

\section{Introduction}
Since the seminal work by Anderson \cite{PhysRev.109.1492}, the study of the Anderson localization-delocalization transition has become a central topic in probability theory and mathematical physics \cite{Lagendijk09Fifty,Abrahams201050years,Sheng2006Intro,lee1985disordered,thouless1974electrons,Anderson1978Local}. Mathematically, Anderson \cite{PhysRev.109.1492} introduced a class of random Schr\"{o}dinger operators---now known as \emph{Anderson's tight-binding model}—on discrete lattices to investigate the transition between localized and delocalized electron wave functions.
The conjecture for the one-dimensional (1D) Anderson model, which posits that localization occurs across the entire spectrum for any non-zero disorder, was rigorously confirmed decades ago \cite{Carmona1982Exp,David2002Local,JIMSPL1977FAIA,cmp/1103908590,ishii1973localization}.
In higher dimensions, however, rigorous mathematical understanding of Anderson localization/delocalization remains limited. The two-dimensional (2D) Anderson model is expected to exhibit localization throughout the entire spectrum for any non-zero disorder \cite{PRL_Anderson}.
In dimensions three and higher, localization is believed to occur throughout the spectrum for sufficiently large disorder. Conversely, for weak disorder, localization is anticipated only near the spectral edges, with delocalization conjectured to occur within the bulk of the spectrum. Despite numerous significant results concerning Anderson localization (see e.g., \cite{FroSpen_1983,Aizenman1993,bourgain2005localization,klein2012comprehensive,carmona1987anderson,ding2020localization,li2022anderson,frohlich1985constructive,von1989new,spencer1988localization}), these conjectures remain unresolved in dimensions two and higher. Notably, rigorous delocalization results have so far been established only for the Anderson model on the Bethe lattice (see e.g., \cite{Bethe_PRL, Bethe_JEMS, aggarwal2025mobilityedgeandersonmodel}), which can be viewed as an infinite-tree-like graph representing an \(\infty\)-dimensional structure, but such results have not yet been demonstrated for any random Schr\"{o}dinger operators in finite dimensions.

Another well-known model in probability theory and mathematical physics that exhibits the Anderson localization–delocalization transition is the \emph{random band matrix} (RBM) model \cite{scalingabndCGMLIF1990PRL,ConJ-Ref2,ScalingPropertyBandMatrixFYMA1991PRL}, also referred to as the \emph{Wegner orbital model} \cite{Wegner1,Wegner2,Wegner3}.
In this paper, we study an RBM $H = (H_{xy})$ defined on a large $d$-dimensional discrete torus $\ZL := \{1, 2, \ldots, L\}^d$, consisting of $N = L^d$ lattice sites. The entries $H_{xy}$ are independent (subject to the Hermitian symmetry $H_{xy} = \overline{H}_{yx}$), centered complex Gaussian random variables with a banded variance profile $S_{xy}:=\mathbb{E}|H_{xy}|^2$. A defining characteristic of the RBM is the locality of hopping---that is, the variances $S_{xy}$ decay rapidly when the graph distance $|x-y|$ on $\ZL$ significantly exceeds the band width parameter $W\ll L$. We normalize the row sums of the variance matrix $S = (S_{xy})$ to be 1, under which the global eigenvalue distribution of $H$ converges weakly to the Wigner semicircle law, supported on the interval $[-2, 2]$ \cite{Wigner}.

The RBM can be viewed as a natural interpolation between the Anderson model and the celebrated Wigner ensemble \cite{Wigner}, with the band width $W$ serving as the interpolation parameter. Like the Anderson model, the RBM exhibits a sharp Anderson metal–insulator transition, depending on both the band width $W$ and the spatial dimension $d$.
More precisely, numerical simulations \cite{ConJ-Ref2, scalingabndCGMLIF1990PRL, ConJ-Ref4, ConJ-Ref6}, as well as non-rigorous arguments based on supersymmetric techniques \cite{ScalingPropertyBandMatrixFYMA1991PRL}, suggest that within the bulk of the spectrum $(-2,2)$, the localization length for 1D RBMs scales as $W^2$. In 2D, it is conjectured to grow exponentially as $\exp(\Omega(W^2))$.
In particular, as $W$ increases and the localization length exceeds the system size $L$, the RBM undergoes a transition from a localized to a delocalized phase.
In dimensions $d\ge 3$, the localization length of bulk eigenvectors becomes infinite once $W$ surpasses a sufficiently large constant, indicating complete delocalization in this regime.

In recent years, substantial progress has been made in understanding the above localization–delocalization conjecture for random band matrices (RBMs), driven by the development of powerful tools in random matrix theory. The localization of 1D RBMs has been rigorously established under the condition $W \ll L^{a}$ for various exponents $a < 1/2$ in \cite{Sch2009,Wegner,CS1_4,CPSS1_4}, and more recently under the sharp condition $W \ll \sqrt L$ in \cite{Localization1_2}.
On the delocalization side, delocalization, bulk universality, and quantum unique ergodicity for 1D RBMs within the bulk of the spectrum $(-2,2)$ were obtained under the condition $W\gg L^{3/4}$ in \cite{PartI,PartII,Band1D_III}. An improved result under $W\gg L^{8/11}$ was later achieved via a flow-based argument introduced in \cite{DY}. This approach, combined with a more refined analysis of the so-called loop hierarchy (see \Cref{lem:SE_basic} below), ultimately led to the resolution of the delocalization conjecture for 1D RBMs down to the critical band width $W\gg \sqrt{L}$ \cite{Band1D}.
The core of this approach is a \emph{tree approximation} to the loop hierarchy (as defined in \Cref{Def_Ktza}), also referred to as the \emph{primitive loop} in \cite{Band1D}.
This method was subsequently extended to the 2D case in \cite{Band2D}, where delocalization of bulk eigenvectors was established under the condition $W\gg L^{\varepsilon}$ for any small constant $\e>0$.
The results of \cite{Band1D,Band2D} were also generalized in \cite{truong2025localizationlengthfinitevolumerandom} to a broader class of block random Schr{\"o}dinger operators. As a special case, the models considered in \cite{truong2025localizationlengthfinitevolumerandom} include the Wegner orbital model with weak inter-block hoppings, which can be viewed as a RBM model with weak interactions.
More recently, the results of \cite{Band1D} were extended to 1D RBMs with non-Gaussian entries and general variance profiles in \cite{erdos2025zigzagstrategyrandomband}, by combining the flow-based analysis of the loop hierarchy with an improved Green's function comparison method. A brief comparison of our results with those of \cite{erdos2025zigzagstrategyrandomband} will be provided in \Cref{sec:comparison}.

With these developments, the delocalization of RBMs is now relatively well understood in one and two dimensions. However, the existing works \cite{Band1D,Band2D,truong2025localizationlengthfinitevolumerandom} focus primarily on RBMs with special block-structured variance profiles, while \cite{erdos2025zigzagstrategyrandomband} addresses only the 1D case under certain technical assumptions on the $\Theta$-propagators (defined in \Cref{def_Theta_propagators} below). These assumptions are often difficult to verify and may fail in some cases (e.g., when the interaction strength is weak, or when the random walk represented by the variance matrix degenerates).
The goal of this paper is to establish delocalization for a broad class of RBM models that includes the block-structured matrices studied in \cite{Band1D,Band2D}, the translationally invariant models in \cite{PartI,PartII,Band1D_III,yang2021delocalizationquantumdiffusionrandom,YYYTexpansion2022CMP,CFHJBulkBandAOP2024}, and the Wegner orbital models with weak inter-block hopping from \cite{truong2025localizationlengthfinitevolumerandom}. In fact, our assumptions on the variance profiles---presented in \Cref{def_considered_model}---extend significantly beyond these previous works, requiring only a certain block-translation symmetry and a core condition. These assumptions are designed specifically to ensure the desired properties of the $\Theta$-propagators, which were assumed without proof in \cite{erdos2025zigzagstrategyrandomband}.
Our approach relies on a new block reduction argument, inspired by techniques from \cite{PartI,PartII,Band1D_III}, to generalize the results of \cite{Band1D,Band2D,truong2025localizationlengthfinitevolumerandom}. We also believe that our results in the 1D setting can be extended to non-Gaussian ensembles by employing the refined Green's function comparison method developed in \cite{erdos2025zigzagstrategyrandomband}. However, extending these results to 2D appears to require new ideas beyond current methods, and we therefore do not pursue universality in the present work to simplify presentation.

The main focus of this paper is the bulk regime of RBMs in dimensions 1 and 2. For completeness, we also briefly summarize the known results near the spectral edges $\pm 2$ and for higher-dimensional RBMs.
At the spectral edges, a sharp phase transition in the edge eigenvalue statistics of 1D RBMs was rigorously established in the seminal work \cite{Sod2010}, occurring as the band width $W$ crosses the critical threshold $L^{5/6}$. This analysis was later extended to higher dimensions ($2 \le d \le 4$), revealing an analogous transition at $W = L^{1 - d/6}$ \cite{Band_Edge123}.
The corresponding delocalization of bulk eigenvectors for 1D and 2D RBMs was recently proved in \cite{Bandedge}, building on the techniques developed in \cite{Band1D,Band2D}. In \cite{Bandedge}, lower bounds on the localization lengths of all eigenvectors were established under the condition $W \gg L^{\varepsilon}$, covering eigenvectors associated with energies across the entire spectrum $[-2,2]$, including those near the spectral edges.
The delocalization of bulk eigenvectors for RBMs in dimensions $d\ge 7$ was proved under the condition $W\gg L^{\varepsilon}$, using a technically intricate diagrammatic expansion method \cite{yang2021delocalizationquantumdiffusionrandom,YYYTexpansion2022CMP,CFHJBulkBandAOP2024}.
By combining this diagrammatic expansion method with the tree approximation of the loop hierarchy, delocalization results were further extended to all dimensions $d \ge 3$ under the condition $W \gg L^{\varepsilon}$ \cite{DYYY25_d3}.

\subsection{The model}

We consider a general class of random band matrices with band width $W$ in dimensions $1$ and $2$. More precisely, let $H=(H_{xy})$ be a random band matrix whose indices $x,y$ belong to the $d$-dimensional lattice \(\Z_L^d=\qq{-\pa{L-1}/{2},\pa{L-1}/2}^d\).
For definiteness, we assume throughout the paper that $L=nW$ for some $n,W\in 2\N+1$; the case of even $L$ can be treated analogously by shifting the origin of the lattice. We also view $\Z_L^d$ as the set of representatives of elements in the discrete torus $\Z^d/L\Z^d$, and denote by $\p{x}_L$ the representative of $x$. For notational convenience, we impose the following block structure on $\Z_L^d$. We emphasize that this block structure is purely spatial; the variance profile of the model remains completely general (see \Cref{def_considered_model} below).

\begin{definition}[Block structure] \label{def: BM2}
The lattice $\mathbb Z_L^d$ is divided into $n^d$ blocks of linear size $W$, such that the central one is $\qq{ -(W-1)/2, (W-1)/2}^d$. Given any $x\in \Z_L^d$, denote the block containing $x$ by $[x]$. Denote the lattice of blocks $[x]$ by \smash{$\wt\Z_n^d$}. We will also view \smash{$\wt\Z_n^d$} as a torus and denote by $\pa{[x]}_n$ the representative of $[x]$ in \smash{$\wt\Z_n^d$}. For convenience, we will regard $[x]$ both as a vertex of the lattice \smash{$\wt\Z_n^d$} and a subset of vertices on the lattice $\Z_L^d$. Denote by $\{x\}$ the representative of $x$ in the block $[0]$ containing the origin, i.e.,
$$\{x\}:=(x+W\Z^d)\cap [0] = x - W[x].$$
Any $x\in \Z_L^d$ can be labeled as $([x],\{x\})$. For definiteness, we use the $L^1$-norm in this paper, i.e., $\|x-y\|_L:=\|\p{x-y}_L\|_1$, which is the (periodic) graph distance on $\Z_L^d$.
Similarly, we also define the periodic $L^1$-distance $\|\cdot\|_n$ on \smash{$\wt \Z_n^d$}. For simplicity of notations, throughout this paper, we will abbreviate
\begin{align}\label{Japanesebracket} |x-y|\equiv \|x-y\|_L,\quad &\langle x-y \rangle \equiv \|x-y\|_L + W,\quad \text{for} \ \ x,y \in \Z_L^d, \\
\label{Japanesebracket2} |[x]-[y]|\equiv \|[x]-[y]\|_n,\quad &\langle [x]-[y] \rangle \equiv \|[x]-[y]\|_n + 1, \quad \text{for} \ \  x,y \in \wt \Z_n^d.
\end{align}
\end{definition}
Using the block structure in \Cref{def: BM2}, for any $\fn$-tensor $\cal A$ with indices $\bx\equiv\pa{{x_1},\ldots,{x_\fn}}\in\pa{\Z_L^d}^\fn$, we can define the ``projected tensor'' with indices in \smash{$\ba\equiv\pa{\qa{a_1},\ldots,\qa{a_\fn}}\in\p{\wt \Z_n^d}^\fn$} as
\begin{equation}\label{def_average_tensor}
        \qq{\cal A}_{\ba}:=W^{-\fn d}\sum_{\bx\in\qa{a_1}\times\cdots\times\qa{a_\fn}}\cal{A}_{\bx}.
\end{equation}
Denoting $N:=L^d$ as the volume of the lattice, we define the model studied in this paper as follows.
For definiteness, we focus on the complex Hermitian model in this paper, while the real symmetric case can be treated in a similar manner, as illustrated in \cite{erdos2025zigzagstrategyrandomband}.

\begin{definition}[Random band matrix]\label{def_considered_model}
Fix $d\in\{1,2\}$. Assume that $H$ is an $N\times N$ complex Hermitian random matrix with independent Gaussian entries up to the Hermitian symmetry $H_{xy}=\overline H_{yx}$. Specifically, the diagonal entries of $H$ are real Gaussian random variables, and the off-diagonal entries are complex Gaussian random variables, distributed as follows:
    \begin{equation}\label{eq:distr_H}
        \begin{aligned}
            H_{xy}\sim \begin{cases}\mathcal{N}_{\R}(0, S_{xy}), \ & \text{if}\ {x=y} \\ \mathcal{N}_{\C}(0, S_{xy}) , \ & \text{if}\  {x\ne y}.\end{cases}
        \end{aligned}
    \end{equation}
Here, the variance matrix $S$ is a \emph{symmetric doubly stochastic} matrix satisfying the following conditions for some constants $\varepsilon_S,C_S>0$:
    \begin{enumerate}
        \item (Block translation invariance) For any $\qa{a},\qa{b},\qa{x}\in \Zn$, we have        \begin{equation}\label{def_block_translation_invarinace}
            \begin{aligned}                S|_{\qa{a}+\qa{x},\qa{b}+\qa{x}}=S|_{\qa{a}\qa{b}},
            \end{aligned}
        \end{equation}
where $S|_{[a][b]}$ denotes the $([a],[b])$-th block of $S$, which is a $W^d\times W^d$ matrix.

        \item (Local $\varepsilon_S$-fullness/core condition) For any $\qa{a}\in\wt\Z_n^d$ and $x,y\in\qa{a}$, we have        \begin{equation}\label{def_epsilon_full}
            \begin{aligned}                \pa{S|_{\qa{a}\qa{a}}}_{xy}\geq  \varepsilon_S W^{-d}.
            \end{aligned}
        \end{equation}
        \item ($C_S$-flat condition) For any $x,y\in\Z_L^d$, we have
        \begin{equation}\label{def_C_flat}
            \begin{aligned}
                S_{xy}\leq C_S W^{-d}\mathbf{1}_{\absa{x-y}\leq C_SW}.
            \end{aligned}
        \end{equation}
    \end{enumerate}
    Additionally, we assume the following parity symmetry: for any $\qa{a},\qa{b}\in \Zn$ and $x,y\in\qa{0}$,     \begin{equation}\label{parity_symmetry_block_2D}
        \begin{aligned}      \pa{S|_{\qa{a}\qa{b}}}_{x,y}=\pa{S|_{\qa{a}\qa{b}}}_{-y,-x}.
        \end{aligned}
    \end{equation}
    We define the ``interaction strength" as the average of the elements in the off-diagonal blocks of $S$:
    \begin{equation}\label{def_lambda}
        \begin{aligned}
            \lambda^2\equiv \lambda^2\pa{S}:=\frac{1}{L^d}\sum_{[a]}\sum_{[b]:[b]\ne[a]}\sum_{x\in\qa{a},y\in\qa{b}}S_{xy}=\frac{1}{W   ^d}\sum_{[b]:[b]\ne[a]}\sum_{x\in\qa{a},y\in\qa{b}}S_{xy},\quad \forall [a]\in \Zn.
        \end{aligned}
    \end{equation}
    Here, the equality follows from the block translation invariance \eqref{def_block_translation_invarinace}.
    Note that $\lambda^2\pa{S}$ corresponds (up to a constant factor) to the variance of a single step of the random walk on \smash{$\Zn$} whose transition probability matrix is given by $W^d\qq{S}$ (recall \eqref{def_average_tensor}). We now assume the following three conditions on this random walk:
    \begin{enumerate}
        \item[(i)] {\bf Interaction strength}: There exists a small constant $\varepsilon_{\txt{inter}}>0$ such that
    \begin{equation}\label{S_interaction_strength}
        \begin{aligned}
            \lambda^2\pa{S}\geq W^{-d+\varepsilon_{\txt{inter}}}.
        \end{aligned}
    \end{equation}
        \item[(ii)] {\bf Irreducibility}: There exists a constant \(c_{\mathrm{irre}} > 0\) such that, uniformly for all \(\mathbf{p} \in [-\pi, \pi]^d\),          \begin{equation}\label{S_irreducibility}
                1-\varphi\pa{\mathbf{p}}=W^d\sum_{\qa{x}\in \Zn\setminus\{0\}}\qq{S}_{\qa{0}\qa{x}}\pa{1-\cos\pa{\mathbf{p}\cdot\qa{x}}}\geq c_{\txt{irre}}\lambda^2\pa{S}|\mathbf{p}|^2,
        \end{equation}
        where \(\varphi(\mathbf{p}) := \E(e^{i\mathbf{p} \cdot X_1})\) is the characteristic function of the one-step displacement \(X_1\) of the random walk.
This irreducibility condition implies isotropy of the random walk by taking the Hessian of both sides of \eqref{S_irreducibility} and using that $X_1$ is symmetrically distributed.

        \item[(ii)'] {\bf Isotropy}: There exists a small constant $c_{\txt{iso}}>0$ such that
        \begin{equation}\label{S_isotropy}
            \begin{aligned}
            c_{\txt{iso}}\lambda^2(S)I_d\leq \Sigma\leq c_{\txt{iso}}^{-1}\lambda^2(S) I_d,
            \end{aligned}
        \end{equation}
        where $\Sigma:=\E\pa{X_1X_1^{\top}}$ is the covariance matrix of a single step \(X_1 \in \Zn\) of the random walk.

    \end{enumerate}
\end{definition}

The main assumptions for our model are given in conditions \eqref{def_block_translation_invarinace}--\eqref{def_C_flat}. We now briefly comment on the additional assumptions \eqref{parity_symmetry_block_2D}, \eqref{S_interaction_strength}, and \eqref{S_irreducibility}.
The assumption \eqref{S_interaction_strength} is in fact a sharp condition for delocalization, as demonstrated for Wegner orbital models in \cite{truong2025localizationlengthfinitevolumerandom,Wegner}. Specifically, when $\lambda^2(S)\ll W^{-d}$, bulk eigenvectors of $H$ are localized with localization length of order one, as shown in \cite{Wegner} for a special class of Wegner orbital models that fall within our framework.
The irreducibility assumption \eqref{S_irreducibility} imposes a qualitative requirement on the irreducibility of the associated random walk. It is designed to prevent the model from degenerating into a lower-dimensional system—in which case the wave function would be confined to a lower-dimensional sublattice and the localization–delocalization transition would no longer be properly captured. For example, if the isotropy condition \eqref{S_isotropy} fails, a 2D random band matrix may effectively reduce to a 1D model, whose behavior differs fundamentally from the 2D case.
From a technical standpoint, the assumption in \eqref{S_irreducibility} is only required to establish the properties of the $\Theta$-propagators in \Cref{lemma_properties_of_propagators}. Once these properties are assumed---as is done, for instance, in Definition 2.4 of \cite{erdos2025zigzagstrategyrandomband}---the irreducibility assumption is no longer needed.

Finally, the symmetry assumption \eqref{parity_symmetry_block_2D} is purely technical and should be regarded as a candidate for removal in future work. It requires each block to be symmetric across its anti-diagonal. This condition is automatically satisfied when the variance matrix exhibits translation symmetry---a property we do not assume in our model definition in \Cref{def_considered_model}, where only block-level translation symmetry is imposed.
We emphasize that, from a technical perspective, this assumption is only used in the proofs of the upper bound for $\cK$-loops in \Cref{bounds_on_primitive_loops} and the upper bound for the evolution kernel in \Cref{lem:sum_decay}. In the latter case, some form of symmetry is required for the 2D setting, as also observed in \cite{Band2D,truong2025localizationlengthfinitevolumerandom}.

We now present several examples of random band matrix models covered by \Cref{def_considered_model}, which also serve as the main motivating examples for this work.

\begin{example}[Translationally invariant random band matrices]\label{example_model}
The models described in \Cref{def_considered_model} include the most standard random band matrix models with a translationally invariant variance profile $S$. We can easily check that $S$ satisfies the assumptions in \Cref{def_considered_model} as long as the following conditions hold for some large constant $C>0$:
    \begin{equation*}
        \begin{aligned}
            C^{-1}\frac{\mathbf{1}_{|x-y|\leq W}}{W^d}\leq S_{xy}\leq C\frac{\mathbf{1}_{|x-y|\leq CW}}{W^d}, \quad \forall x,y\in\ZL.
        \end{aligned}
    \end{equation*}
    This class of models was studied in \cite{PartI,PartII,Band1D_III}.
\end{example}

\begin{example}[Wegner orbital model with general block variance profiles]\label{example_WO}
Our framework also covers a general class of Wegner orbital models. Specifically, let the potential block $V\in\R^{\qa{0}\times \qa{0}}$ be a symmetric matrix, and denote the interaction blocks by \smash{$\h{A^{\qa{x}}:\qa{x}\in\Zn}\subseteq \R^{\qa{0}\times \qa{0}}$}. Assume that all these matrices have nonnegative entries. We define the variance profile matrix $S$ by
    \begin{equation*}            S_{W\qa{a}+\ha{a},W\qa{b}+\ha{b}}:=\mathbf{1}_{\qa{a}=\qa{b}}V_{\ha{a}\ha{b}}+\mathbf{1}_{\qa{a}\neq\qa{b}}A^{\qa{b}-\qa{a}}_{\ha{a}\ha{b}}.
    \end{equation*}
In accordance with the assumptions in \Cref{def_considered_model}, we impose the following three conditions on the blocks $V$ and \smash{$\h{A^{\qa{x}}:\qa{x}\in\Zn\setminus\ha{0}}$}.
    \begin{enumerate}
        \item For any block $X\in\ha{V}\cup\h{A^{\qa{x}}:\qa{x}\in\Zn\setminus\ha{0}}$,
        \begin{equation*}
            \begin{aligned}
                X_{xy}=X_{-y,-x}\quad\forall x,y\in\qa{0}.
            \end{aligned}
        \end{equation*}
        Moreover, there exists a constant $\varepsilon_0>0$ such that \smash{$\max_{x,y\in\qa{0}}X_{xy}\le \varepsilon_0^{-1}W^{-d}$}, and, in addition, $V_{xy}\geq \varepsilon_0W^{-d}$ for all $x,y\in\qa{0}$.

        \item For any $\qa{x}\in\Zn\setminus\ha{0}$, we have $(A^{\qa{x}})^{\top}=A^{-\qa{x}}$. Moreover, there exists a constant $C>0$ such that
        \begin{equation*}
            \begin{aligned}
                A^{\qa{x}}=0,\quad \forall |\q{x}|> C.
            \end{aligned}
        \end{equation*}

        \item There exists a small constant $c>0$ such that the following set generates, in the sense of group operations, a $d$-dimensional lattice:
        \begin{equation}\label{eq:irred_A}
                \scr{C}\equiv\scr{C}\pa{c}:=\bigg\{\qa{x}\in\Zn\setminus\ha{[0]}: \sum_{a,b\in\qa{0}}A^{\qa{x}}_{ab}\geq c\sum_{\qa{y}\ne [0]}\sum_{a,b\in\qa{0}}A^{\qa{y}}_{ab}\bigg\}.
        \end{equation}
        Additionally, there exists a constant $\varepsilon_{\txt{inter}}>0$ such that
        \begin{equation*}
            \begin{aligned}
\sum_{\qa{y}\ne[0]}\sum_{a,b\in\qa{0}}A^{\qa{y}}_{ab}\geq W^{\varepsilon_{\txt{inter}}}.
            \end{aligned}
        \end{equation*}
    \end{enumerate}
    Under these assumptions, all conditions in \Cref{def_considered_model}---except for \eqref{S_irreducibility}---can be verified directly. The irreducibility condition \eqref{S_irreducibility} follows from the definition
    \begin{equation*}
        \begin{aligned}
            \lambda^2=W^{-d}\sum_{\qa{y}\ne [0]}\sum_{a,b\in\qa{0}}A^{\qa{y}}_{ab},
        \end{aligned}
    \end{equation*}
    along with \eqref{eq:irred_A} and the existence of a constant $c_0>0$ such that
    \begin{equation}
            \sum_{\qa{x}\in \scr{C}}\q{1-\cos\pa{\mathbf{p}\cdot\qa{x}}}\geq c_0|\mathbf{p}|^2,\quad \forall \mathbf{p}\in\qa{-\pi,\pi}^d.
    \end{equation}
    The Wegner orbital model considered in Definition 1.2 of \cite{truong2025localizationlengthfinitevolumerandom} is a special case of this framework.  Furthermore, the current setting allows an interaction block $A^{\qa{x}}$ to be concentrated within a small subset of the index set $[0]\times [0]$, involving only $W^{d+\txt{o}(1)}$ nonzero entries out of the total $W^{2d}$.
\end{example}

\subsection{Main results}\label{sec:comparison}

In this paper, we extend the results for 1D and 2D random band matrices (or Wegner orbital models) from \cite{Band1D,Band2D,truong2025localizationlengthfinitevolumerandom} to a broader class of random band matrices defined in \Cref{def_considered_model}. Specifically, our results cover the models in Examples \ref{example_model} and \ref{example_WO}.
For the matrix $H$ defined in \Cref{def_considered_model}, we denote its eigenvalues by \smash{$\ha{\lambda_k:k\in\ZL}$} and the corresponding eigenvectors by \smash{$\ha{\bu_{k}:k\in\ZL}$}. Our first result establishes the delocalization of bulk eigenvectors for $H$.

\begin{theorem}[Delocalization]\label{thm:supu}
    Consider the matrix $H$ defined in \Cref{def_considered_model} with $d\in\{1,2\}$. Let $\kappa,\delta>0$ be arbitrary small constants. If $W\geq L^{\delta}$, then for any constants $\tau,D>0$, the following estimate holds for sufficiently large $L$:
    \begin{equation}\label{eq:delocalmax}
        \begin{aligned}
            \P\bigg(\sup_{k: |\lambda_k | \leq 2 - \kappa} \|\bu_k\|_\infty^2 \leq W^\tau \eta_* \bigg) &\ge 1- L^{-D} ,
        \end{aligned}
    \end{equation}
    where $\eta_*$ is defined as
    \be\label{def_eta_star}
\eta_*:=\frac{1}{(W\lambda)^2}\mathbf 1_{d=1}+\frac{1}{N}.
\ee
\end{theorem}

The above delocalization of bulk eigenvectors is an immediate consequence of the following optimal local law for the Green's function (or resolvent) of $H$, defined as\[G\pa{z}:=\pa{H-z}^{-1},\quad z\in \C\setminus\R.\]

\begin{theorem}[Local law]\label{thm_locallaw}
In the setting of \Cref{thm:supu}. For any constants $\fc, \tau,D>0$, the following local law estimates hold for $z=E+\ii \eta$ and large enough $L$ with probability $\geq 1-L^{-D}$:
\begin{align}\label{locallaw}
&\bigcap_{|E|\le 2- \kappa}\bigcap_{W^{\fc}\eta_*\le \eta\le \fc^{-1}}\ha{ \|G(z) - m(z) \|_{\max}^2  \le \frac{W^\tau}{W^d\ell(\eta)^d \eta }} ,
\\
\label{locallaw_aver}
&\bigcap_{|E|\le 2- \kappa}\bigcap_{W^{\fc}\eta_*\le \eta\le \fc^{-1}}\ha{\max_{[a]\in\Zn} \Big|W^{-d}\sum_{x\in [a]}G_{xx}(z) - m(z) \Big| \le \frac{W^\tau}{W^d\ell(\eta)^d \eta }}.
\end{align}
Here, $m\p{z}$ is the Stieltjes transformation of the Wigner semicircle law:
\begin{equation}\label{def_m}
    \begin{aligned}
        m\p{z}:=\frac{1}{2\pi}\int_{\R}\frac{\sqrt{4-x^2}}{x-z}\rd x
        = \frac{-z+\sqrt{z^2-4}}{2},
    \end{aligned}
\end{equation}
and $\ell(\eta)$ is defined by
\be\label{eq:elleta}
\ell(\eta):= \min\left(\lambda\eta^{-1/2}+1, n\right).
\ee
\end{theorem}

In the delocalized regime, where $\eta_*$ is of order $N^{-1}$, we can further establish a stronger \emph{quantum unique ergodicity} (QUE) estimate for the bulk eigenvectors, albeit at the cost of a slightly weaker probability bound.

\begin{theorem}[Quantum unique ergodicity]\label{thm:QUE}
In the setting of \Cref{thm:supu}, assume in addition that $W\ge W^\fc N^{1/2}$ for a constant $\fc>0$ in dimension $d=1$. Given a constant $\e\in (0, \fc)$, define the subset $${\cal I}_E\equiv {\cal I}_E(\e):=\left\{x: |x-E|\le W^{-\e}\eta_{0}\right\},\quad\text{with}\quad \eta_{0}:=\begin{cases}
  W\lambda/N^{3/2}, & \text{if} \ d=1\\
{W\lambda}/{N}, & \text{if} \ d=2\\
\end{cases}.$$
Then, for each $d\in\{1,2\}$, there exists a small constant $c>0$ depending on $\e$, $\fc$, and the constant $\varepsilon_{\txt{inter}}$ in \eqref{S_interaction_strength}, such that the following estimate holds for large enough $L$:
\begin{equation}\label{Meq:QUE}
\sup_{E: |E|\le 2-\kappa}
\max_{ [a]\in \Zn} \P\bigg(\max_{i,j:\lambda_i, \lambda_j \in {\cal I}_E}
 \bigg| \sum_{x\in[a]}\overline \bu_i(x)\bu_j(x)-\frac{W^{d}}{N}\delta_{ij}  \bigg|^2 \ge \frac{W^{d-c}}{N} \bigg) \le  W^{-c},
\end{equation}
More generally, for any subset $A\subset \Zn$, we have
\begin{equation}\label{Meq:QUE2}
\sup_{E: |E|\le 2-\kappa}
\mathbb{P}\bigg(\max_{k: \lambda_k\in {\cal I}_E}
\bigg|\sum_{[a]\in A}\sum_{x\in [a]}\left|\bu_k(x)\right |^2 -\frac{W^d}{N}|A|\bigg| \ge  \frac{W^{d-c}|A|}{N}  \bigg) \le W^{-c}.
\end{equation}
\end{theorem}

The above QUE estimates can be derived from the following \emph{quantum diffusion conjecture} for random band matrices.

\begin{theorem}[Quantum diffusion]\label{thm_diffu}
In the setting of \Cref{thm:supu}, for any constants $\tau, D>0$, the following events hold for large enough $L$ with probability $\geq 1-L^{-D}$:
\begin{align}
&\bigcap_{|E|\le 2- \kappa}\bigcap_{W^{\fc}\eta_*\le \eta\le \fc^{-1}}\ha{\max_{[a],[b]} \bigg|\frac{1}{W^{2d}}\sum_{x\in[a],y\in[b]}\left(|G_{xy}|^2 -\wh \cK_{\pa{+,-},\pa{x,y}}\right)\bigg| \le\frac{W^\tau}{(W^d\ell(\eta)^d\eta)^2}},\label{eq:diffu1}\\
&\bigcap_{|E|\le 2- \kappa}\bigcap_{W^{\fc}\eta_*\le \eta\le \fc^{-1}}\ha{\max_{[a],[b]} \bigg|\frac{1}{W^{2d}}\sum_{x\in[a],y\in[b]}\left(G_{xy}G_{yx} -\wh \cK_{\pa{+,+},\pa{x,y}}\right)\bigg| \le\frac{W^\tau}{(W^d\ell(\eta)^d\eta)^2}},\label{eq:diffu2}
\end{align}
where $\wh \cK_{\pa{+,-},\pa{x,y}}$ and $\wh \cK_{\pa{+,-},\pa{x,y}}$ are defined as
\begin{equation}\nonumber
    \begin{aligned}
        \wh \cK_{\pa{+,-},\pa{x,y}}=\absa{m\p{z}}^2\p{1-\absa{m\p{z}}^2S}^{-1}_{xy},\quad \wh \cK_{\pa{+,+},\pa{x,y}}=m^2\p{z}\pa{1-m^2\p{z}S}^{-1}_{xy}.
    \end{aligned}
\end{equation}
Moreover, stronger bounds hold in the sense of expectation for each $z$ with $|E|\le 2- \kappa$ and $W^{\e}\eta_*\le \eta\le \fc^{-1}$:
\begin{align}
\max_{[a],[b]}\bigg|\frac{1}{W^{2d}}\sum_{x\in[a],y\in[b]}\E\left(|G_{xy}|^2 -\wh \cK_{\pa{+,-},\pa{x,y}}\right) \bigg| &\le\frac{W^\tau}{(W^d\ell(\eta)^d\eta)^3},\label{eq:diffuExp1}\\
\max_{[a],[b]}\bigg|\frac{1}{W^{2d}}\sum_{x\in[a],y\in[b]}\E\left(G_{xy}G_{yx} -\wh \cK_{\pa{+,-},\pa{x,y}}\right)\bigg|&\le\frac{W^\tau}{(W^d\ell(\eta)^d\eta)^3}.\label{eq:diffuExp2}
\end{align}
\end{theorem}

The derivations of \Cref{thm:supu} and \Cref{thm:QUE} from \Cref{thm_locallaw} and \Cref{thm_diffu} are standard and follow exactly the same arguments as those in previous works \cite{Band1D,Band2D,truong2025localizationlengthfinitevolumerandom} (see, for example, the proofs of Theorems 2.1 and 2.2 in \cite{truong2025localizationlengthfinitevolumerandom}). Hence, we omit their detailed proofs. The proofs of the local laws in \Cref{thm_locallaw} and the quantum diffusion estimates in \Cref{thm_diffu} extend the flow argument developed in \cite{Band1D,Band2D}, using a block reduction method that we will discuss below, and are presented in \Cref{Sec:Stoflo}.

\subsection{Key ideas}

The proofs of our main theorems are based on an extension of the methods developed in \cite{Band1D,Band2D,truong2025localizationlengthfinitevolumerandom}, which rely on a detailed analysis of the so-called loop hierarchy (defined in \Cref{lem:SE_basic} below) along a carefully constructed characteristic flow. The models considered in these previous works share a common block structure in the variance matrix $S$, where each $W^d\times W^d$ block has a completely flat variance profile. This flatness ensures that the loop hierarchy remains self-consistent under the matrix flow
\begin{equation}\label{matrix_flow_block}
    \begin{aligned}
        \rd H_t=S\odot\rd B_t,\quad \txt{for}\quad t\geq 0,\txt{ and } H_0=0,
    \end{aligned}
\end{equation}
where $\odot$ denotes the Hadamard product, and $B_t$ is a matrix Brownian motion. However, when the variance profile within each block is no longer flat, this self-consistency is lost, making the analysis of the loop hierarchy much more challenging.

A key observation in our proof is that the convenient structure of the loop hierarchy, as established in \cite{Band1D,Band2D,truong2025localizationlengthfinitevolumerandom}, can be recovered even in our more general setting---provided the core condition \eqref{def_epsilon_full} holds---through the use of a technique we call the \emph{block reduction method}.
Specifically, we define the block diagonal matrix \smash{$\pa{\SE}_{xy}=\sum_{\qa{a}}\mathbf{1}_{x,y\in\qa{a}}W^{-d},$}
and replace the matrix flow \eqref{matrix_flow_block} with the modified flow
\begin{equation}\label{matrix_flow_introduction}
    \begin{aligned}
        \rd H_t=\SE\odot\rd B_t,\quad t\geq \ti.
    \end{aligned}
\end{equation}
By carefully choosing the initial time $\ti$ and the initial condition $H_{\ti}$, we are able to recover the self-consistent form of the loop hierarchy (see \eqref{pro_dyncalK} below). Consequently, the loop hierarchy analysis developed in \cite{Band1D,Band2D,truong2025localizationlengthfinitevolumerandom} can be applied to our setting with minimal modification (see \Cref{Sec:Stoflo}).
However, this convenience comes at a cost: unlike the flow in \eqref{matrix_flow_block}, which starts from zero and thus provides free initial estimates, the modified flow \eqref{matrix_flow_introduction} begins from a nontrivial initial condition. As a result, we must perform a systematic global analysis to establish global laws for the $G$-loops, demonstrating that they are well approximated by their deterministic limits—namely, the primitive loops defined in \Cref{Def_Ktza}.
Furthermore, the continuity estimates for $G$-loops (see \Cref{lem_ConArg}) in prior works \cite{Band1D,Band2D,truong2025localizationlengthfinitevolumerandom} rely on a rescaling of the variance profile matrix along the flow. In our setting, however, the variance profile matrix is no longer proportional to $S$, so this rescaling argument is no longer applicable.
To overcome this difficulty, we introduce a family of random band matrix models with distinct variance profile matrices, and define a decreasing flow on this family. By the end of the flow, the family contains at least one variance profile that is suitable for establishing the desired continuity estimates.

Another technical challenge in our proof lies in establishing the deterministic properties of the $\Theta$-propagators (see \Cref{def_Theta_propagators} and \Cref{lemma_properties_of_propagators}) and the $\cK$-loops (see \Cref{Def_Ktza} and \Cref{bounds_on_primitive_loops}), which form the foundation for the subsequent stochastic analysis. In previous works \cite{Band1D,Band2D,truong2025localizationlengthfinitevolumerandom}, the properties of $\Theta$-propagators were derived using Fourier series and heat kernel estimates, while the analysis of $\cK$-loops was carried out via a tree representation formula. However, in our setting, the absence of a flat variance profile within each block renders these arguments inapplicable. To recover the deterministic estimates for the $\Theta$-propagators under block-level translation invariance only, we develop a new random walk representation tailored to our setting (see \Cref{lemma_random_walk_representation}). This representation can again be analyzed using Fourier techniques and heat kernel estimates (see the proof of \Cref{lemma_properties_of_propagators} for details).
With the deterministic properties of the $\Theta$-propagators established, we employ a new dynamical method in which the evolution kernel is expressed in terms of the $\Theta$-propagators, and exploit cancellations along the flow. These cancellations yield sharp bounds in the regime of strong interaction (i.e., large $\lambda$). However, in the weak interaction regime ($\lambda\ll 1$), additional difficulties arise due to singularities in the bounds \eqref{prop:BD1} and \eqref{prop:BD2}. To address this, we make a key observation: the $\cK$-loop bound can be improved by a small factor of $\lambda^2$, which precisely cancels the singularities appearing in \eqref{prop:BD1} and \eqref{prop:BD2}, and thereby restores control in the weak interaction case.

Before concluding this section, we compare our work with the recent study \cite{erdos2025zigzagstrategyrandomband}, which extends the results of \cite{Band1D} to 1D random band matrices with non-Gaussian entries and general variance profiles under the condition $W\gg \sqrt{L}$.  Moreover,
\cite{erdos2025zigzagstrategyrandomband} establishes a stronger form of delocalization and QUE, namely isotropic delocalization and QUE for general observables.
While there is partial overlap with our results in 1D, their arguments rely on the so-called \emph{zigzag strategy}, first introduced in \cite{cipolloni2024mesoscopic} for the study of non-Hermitian random matrices (see also the introduction of \cite{erdos2025zigzagstrategyrandomband} for a detailed overview).
This strategy consists of two main components: a flow from the global scale to the local scale, with the introduction of a Gaussian component---referred to as the ``zig step"---and a Green's function comparison argument---known as the ``zag step"---that removes this Gaussian component at each stage of induction. Consequently, their conclusions hold universally, without depending on the specific distribution of the matrix entries.
In contrast, the methods employed in this paper apply only to Gaussian divisible models, i.e., random matrices with a small Gaussian component, due to the absence of the Green's function comparison argument.
We note, however, that the Green's function comparison argument (the “zag step”) does not appear to extend effectively to 2D, nor even to 1D when \smash{$W\ll \sqrt{L}$}. Extending the present approach beyond the Gaussian divisible setting would thus require new ideas. For this reason, and to keep the exposition streamlined, we state our main theorems for Gaussian random band matrices only.\footnote{To extend our results to Gaussian divisible matrices, the only modification in the proof is the need to establish the global laws in Lemmas \ref{lemma_global_law} and \ref{improved_global_laws} for random band matrices with entries from a general distribution, along with a minor adjustment to the variance-profile flow. These global laws can be proved using a similar argument based on cumulant expansions, with additional work to control higher-order cumulant terms, which do not affect the validity of our arguments. Since this argument is standard and our focus is on the new block reduction method, we omit the details for brevity.}

On the other hand, in this work, we prove delocalization for a class of RBMs with general variance profiles in both one and two dimensions under the condition $W\ge L^\e$ for an arbitrarily small constant $\e>0$. The block reduction method presented here is an extension of the approach developed in the series of works \cite{Band1D,Band2D,truong2025localizationlengthfinitevolumerandom}, and is also inspired by the mean-field reduction idea from \cite{PartI,PartII}. Compared to the methods in \cite{erdos2025zigzagstrategyrandomband}, our block reduction method simplifies the flow argument (i.e., the ``zig step") to some extent.
Additionally, the proof in \cite{erdos2025zigzagstrategyrandomband} relies on the delicate assumption of an \emph{admissible variance profile}, as defined in Definition 2.5 there, which is somewhat stronger than our \Cref{lemma_properties_of_propagators} and incorporates the properties of the $\Theta$-propagators as assumptions. As a result, \cite{erdos2025zigzagstrategyrandomband} can cover random band matrices with variance profiles (and their first and second derivatives) decaying faster than certain power laws. If we also assume the results in \Cref{lemma_properties_of_propagators}, all of our arguments still hold, and the proof can be simplified further.
However, the assumption of admissible variance profiles can be difficult to verify, especially without assumptions on the full translation invariance of the variance profile or flat variance profiles within blocks. Specifically, under our \Cref{def_considered_model}, the translation invariance requirement applies only at the block level, and the variance profile within each block could be nearly arbitrary. Consequently, the pointwise estimates for the $\Theta$-propagators may not hold. In our block reduction method, we only require ``blockwise estimates" for the $\Theta$-propagators. One key contribution of this work is to provide a systematic approach to establish such estimates by assuming only translation invariance at the block level and the irreducibility condition \eqref{S_irreducibility} (the symmetry assumption \eqref{parity_symmetry_block_2D} and the core condition \eqref{def_epsilon_full} are not necessary for this proof).

\subsection*{Organization of the remaining text}

In \Cref{sec_preliminaries}, we introduce several tools necessary for proving the main theorems, including the flow framework, $G$-loops, and their deterministic limits---referred to as primitive loops or $\cK$-loops. This section concludes with a discussion of the properties of $\Theta$-propagators and the evolution kernel for the loop hierarchy, with proofs provided in \Cref{sec_analysis_of_primitive_loops}.
Due to the general profile setting employed (notably, the absence of block structure), the deterministic analyses in \cite{Band1D,Band2D,truong2025localizationlengthfinitevolumerandom} do not directly apply to our framework. In \Cref{sec_analysis_of_primitive_loops}, we use a dynamical method to recover these deterministic results and establish upper bounds for the $\cK$-loops.
Since, in our proof, the flow does not start at $t=0$ as in previous works \cite{Band1D,Band2D,truong2025localizationlengthfinitevolumerandom}, we must derive global laws for the $G$-loops, which are presented in \Cref{sec_global_law}. Using these global laws as initial estimates, we proceed with an analysis of the flow of the loop hierarchy in \Cref{Sec:Stoflo}. Additional technical lemmas are gathered and proved in \Cref{additional_proofs}.

\subsection*{Notations} To facilitate the presentation, we introduce some necessary notations that will be used throughout this paper. We will use the set of natural numbers $\N=\{1,2,3,\ldots\}$ and the upper half complex plane $\C_+:=\{z\in \C:\im z>0\}$.
In this paper, we are interested in the asymptotic regime with $N\to \infty$. When we refer to a constant, it will not depend on $N$ or $W$. Unless otherwise noted, we will use $C$, $D$ etc.~to denote large positive constants, whose values may change from line to line. Similarly, we will use $\e$, $\delta$, $\tau$, $c$, $\fc$, $\fd$ etc.~to denote small positive constants.
For any two (possibly complex) sequences $a_N$ and $b_N$ depending on $N$, $a_N = \OO(b_N)$, $b_N=\Omega(a_N)$, or $a_N \lesssim b_N$ means that $|a_N| \le C|b_N|$ for some constant $C>0$, whereas $a_N=\oo(b_N)$ or $|a_N|\ll |b_N|$ means that $|a_N| /|b_N| \to 0$ as $N\to \infty$.
We say that $a_N \sim b_N$ if $a_N = \OO(b_N)$ and $b_N = \OO(a_N)$. For any $a,b\in\R$, we denote $\llbracket a, b\rrbracket: = [a,b]\cap \Z$, $\qq{a}:=\qq{1,a}$, $a\vee b:=\max\{a, b\}$, and $a\wedge b:=\min\{a, b\}$. For an event $\Xi$, we let $\mathbf 1_\Xi$ or $\mathbf 1(\Xi)$ denote its indicator function.
Given a vector $\mathbf v$, $|\mathbf v|\equiv \|\mathbf v\|_2$ denotes the Euclidean norm and $\|\mathbf v\|_p$ denotes the $L^p$-norm.
Given a matrix $\cal A = (\cal A_{ij})$, $\|\cal A\|$, $\|\cal A\|_{p\to p}$, and $\|\cal A\|_{\infty}\equiv \|\cal A\|_{\max}:=\max_{i,j}|\cal A_{ij}|$ denote the operator (i.e., $L^2\to L^2$) norm,  $L^p\to L^p$ norm (where we allow $p=\infty$), and maximum (i.e., $L^\infty$) norm, respectively. We will use $\cal A_{ij}$ and $ \cal A(i,j)$ interchangeably in this paper. We also introduce the following simplified notation for trace: $ \left\langle \cal A\right\rangle=\tr (\cal A) .$

For clarity of presentation, we will adopt the following notion of stochastic domination introduced in \cite{EKY_Average}.

\begin{definition}[Stochastic domination and high probability event]\label{stoch_domination}
	{\rm{(i)}} Let
	\[\xi=\left(\xi^{(N)}(u):N\in\mathbb N, u\in U^{(N)}\right),\hskip 10pt \zeta=\left(\zeta^{(N)}(u):N\in\mathbb N, u\in U^{(N)}\right),\]
	be two families of non-negative random variables, where $U^{(N)}$ is a possibly $N$-dependent parameter set. We say $\xi$ is stochastically dominated by $\zeta$, uniformly in $u$, if for any fixed (small) $\tau>0$ and (large) $D>0$,
	\[\mathbb P\bigg[\bigcup_{u\in U^{(N)}}\left\{\xi^{(N)}(u)>N^\tau\zeta^{(N)}(u)\right\}\bigg]\le N^{-D}\]
	for large enough $N\ge N_0(\tau, D)$, and we will use the notation $\xi\prec\zeta$.
	If for some complex family $\xi$ we have $|\xi|\prec\zeta$, then we will also write $\xi \prec \zeta$ or $\xi=\OO_\prec(\zeta)$.

	\vspace{5pt}
	\noindent {\rm{(ii)}} As a convention, for two deterministic non-negative quantities $\xi$ and $\zeta$, we will write $\xi\prec\zeta$ if and only if $\xi\le N^\tau \zeta$ for any constant $\tau>0$.

	\vspace{5pt}
	\noindent {\rm{(iii)}} We say that an event $\Xi$ holds with high probability (w.h.p.) if for any constant $D>0$, $\mathbb P(\Xi)\ge 1- N^{-D}$ for large enough $N$. More generally, we say that an event $\Omega$ holds $w.h.p.$ in $\Xi$ if for any constant $D>0$,
	$\P( \Xi\setminus \Omega)\le N^{-D}$ for large enough $N$.
\end{definition}

\subsection*{Acknowledgement}
Fan Yang is supported in part by the National Key R\&D Program of China (No.~2023YFA1010400) and NSFC (No.~12526201).
We would like to thank Chenlin Gu and Guangyi Zou for valuable discussions.

\section{Preliminaries}\label{sec_preliminaries}

In this section, we introduce the fundamental tools and outline the main strategy for proving \Cref{thm_locallaw,thm_diffu}. The stochastic and deterministic flows employed in the proofs are defined in \Cref{subsec:flow}, together with an auxiliary flow of variance profiles that enables the continuity estimate (see \Cref{lem_ConArg}). In \Cref{subsec:loops}, we define the loop hierarchy of $G$-loops and their deterministic limits—the primitive loops. Finally, in \Cref{subsec:propagators}, we introduce the $\Theta$-propagators and the evolution kernels, which are key tools for analyzing the loop hierarchy and the $\cK$-loops.

We consider the random band matrix $H$ defined in \Cref{def_considered_model}, whose variance profile $S^{\txt{RBM}}$ satisfies the conditions \eqref{def_epsilon_full} and \eqref{def_C_flat} with constants $\varepsilon_S = 2\varepsilon_0$ and $C_S = C_0$, respectively, where $\varepsilon_0\in(0,1/100)$ and $C_0>1$ are fixed constants. We also fix small constants $\kappa, \mathfrak{c} > 0$ and a target spectral parameter $z = E + \ii \eta$ with $\abs{E} \le 2 - \kappa$ and $\eta \in [W^{\mathfrak{c}} \eta_*, \mathfrak{c}^{-1}]$ (recall \eqref{def_eta_star} for the definition of $\eta_*$).
Since the variance profile $S^{\txt{RBM}}$ is locally $(2\varepsilon_0)$-full and $C_0$-flat, we have the following decomposition of $\SRBM$:
\begin{equation}\label{eq:SRBM}
        \SRBM=\varepsilon_0 \SE+\pa{1-\varepsilon_0}S_{\txt{o}},
\end{equation}
where the matrix $\SE$ is define by
\begin{equation}\label{eq:SE}
        \pa{\SE}_{xy}=\sum_{\qa{a}\in \wt \Z_n^d} W^{-d} \mathbf{1}\pa{x,y\in\qa{a}} , \quad \forall \ x,y\in \Z_L^d.
\end{equation}
Note that the matrix $S_{\txt{o}}$ is locally $\varepsilon_0$-full and $(2C_0)$-flat.

\subsection{Flow}\label{subsec:flow}

Throughout the proof, we fix an initial time $\ti\in \R$, which denotes the time at which the flow begins.
We will use $t_0\in \R$ to represent a potentially different starting time, which may change during the proof. In the analysis of the loop hierarchy below, we may use multiple stochastic flows with different starting times $t_0$. For further details, see \Cref{sec_analysis_of_primitive_loops}.

Given $t_0$, suppose the initial matrix $H_{t_0}$ is a random Hermitian matrix with independent Gaussian entries (up to the Hermitian symmetry) and distributed as in \eqref{eq:distr_H}, with variance matrix $S=S_{t_0}$. We now consider the following matrix evolution starting at time $t_0$ with initial value $H_{t_0}$:
\begin{equation}\label{def_stochastic_flow}
    \begin{aligned}
        \rd H_t=\SE\odot \rd B_t,\quad t\geq t_0.
    \end{aligned}
\end{equation}
Here, $\odot$ denotes the Hadamard product, and $B_t$ is an $N\times N$ matrix Brownian motion that is independent of $H_{t_0}$. Specifically, $\pa{B_t}_{xy}$ are independent complex Brownian motions up to the Hermitian symmetry, with zero mean and variance $t$.
Along this evolution, the variance profile of $H_t$ is given by:
\begin{equation}\label{def_S_t}
    \begin{aligned}
        S_t\equiv S_t\pa{t_0,S_{t_0}}:=S_{t_0}+\pa{t-t_0}\SE.
    \end{aligned}
\end{equation}
Now, given a target parameter $\mathsf{E}\in \R$, we define the deterministic flow associated with $H_t$ as follows:
\begin{equation}\label{z_t_flow}
    \begin{aligned}
        z_t\equiv z_t\p{\sE}:=\sE+\p{1-t}m\p{\sE}, \quad \forall t\in \qa{t_0,1},
    \end{aligned}
\end{equation}
where $m\p{\sE}\equiv m\p{\sE+\ii 0_+}$. We will denote $z_t=E_t+\ii \eta_t$ for $t\in\qa{t_0,1}$, with
\begin{equation}\label{E_t_eta_t_flow}
    \begin{aligned}
        E_t:=\sE+\pa{1-t}\txt{Re}\,m\p{\sE},\quad \eta_t:=\pa{1-t}\txt{Im}\, m\p{\sE}.
    \end{aligned}
\end{equation}
For $z\in \C_+$, let $m_t(z)$ be the unique solution of the equation $1+zm_t+tm_t^2=0$ with $\im m_t(z)>0$. Then, it is easy to check that
\begin{equation}\label{m_invariant}
    \begin{aligned}
        m_t\p{z_t}=m\p{\sE},\quad \forall t\in\qa{t_0,1}.
    \end{aligned}
\end{equation}
For any target spectral parameter $z$, we are interested in the original resolvent $G(z)=(H-z)^{-1}$. In our model, the study of $G\pa{z}$ can be achieved through the stochastic flow by carefully selecting the parameter $\sE$, the initial time $\ti$, and the final time $\tf$, as discussed in the following lemma.

\begin{lemma}
\label{lem:paraselect}
Fix a target spectral parameter $z=E+\ii\eta\in \C_+$ with $\im z\in (0,1]$ and $|\re z|\le 2-\kappa$ for some small constant $\kappa>0$. We choose the parameters as follows:
    \begin{equation}\label{pick_parameter_1}
        \begin{aligned}
            \tf:=\frac{\im m\p{z}}{\im m\p{z}+\eta},\quad \sE:=\sqrt{\tf}E-\frac{1-\tf}{\sqrt{\tf}}\re m\p{z},\quad t_{\txt{i}}:=\pa{1-\varepsilon_0}\tf.
        \end{aligned}
    \end{equation}
Setting $t_0=\ti$ and $z_{\ti}=\sE+[1-(1+\e_0)\tf]m(\sE)$, the flow \eqref{z_t_flow} yields
    \begin{equation}\label{eq:ztf}
        \begin{aligned}
        z_{\tf}=\sqrt{\tf}z,\quad \text{and}\quad     \sqrt{\tf}m\p{\sE}=m\p{z} .
        \end{aligned}
    \end{equation}
In particular, if $H_{\tf}\overset{\txt{d}}{=} \sqrt{\tf} H$, then
    \begin{equation}\label{equal_in_distribution_t_0}
        \begin{aligned}
            G\pa{z}\overset{\txt{d}}{=}\sqrt{\tf}G_{\tf},
        \end{aligned}
    \end{equation}
    where $G\p{z}:=\p{H-z}^{-1}$ and $G_{\tf}:=\p{H_{\tf}-z_{\tf}}^{-1}$.
    Moreover, by \eqref{E_t_eta_t_flow} and \eqref{eq:ztf}, for all $t\in [\ti,1]$,
    \begin{equation}\label{eta_t_sim_1-t}
        \begin{aligned}
            E_t=\frac{1+t}{1+\tf}\sqrt{\tf}E,\quad \eta_t=\frac{1-t}{\sqrt{\tf}}\im m\p{z},
        \end{aligned}
    \end{equation}
    which immediately implies that
    $\eta_t\sim 1-t$ for any $t\in \qa{\ti,1}$. In particular, we have $\eta_{\ti}\sim 1$ and $\eta_{\tf}\sim \im z$.
\end{lemma}
\begin{proof}
The proof follows the same argument as in \cite[Lemma 3.3]{truong2025localizationlengthfinitevolumerandom}; hence we omit the details.
\end{proof}

In the proof and discussions below, we will always take the parameters as in \eqref{pick_parameter_1}. For any given starting time $t_0<\tf$ and variance profile $S_{t_0}$, the resolvent flow induced by evolution \eqref{def_stochastic_flow} is then defined as
\begin{equation}
    \begin{aligned}
        G_t\equiv G_{t,\sE}:=\pa{H_t-z_t(\sE)}^{-1}.
    \end{aligned}
\end{equation}
By Ito's formula and \eqref{m_invariant}, we have that
\begin{equation}\label{eq:SDE_Gt}
    \begin{aligned}
        \dd G_{t}=-G_{t}(\dd H_{t})G_{t}+G_{t}\{\mathcal{S}_{\txt{B}}[G_{t}]-m\p{\sE}\}G_{t},
    \end{aligned}
\end{equation}
where $\cS_\txt{B}:\C^{N\times N}\to\C^{N\times N}$ is a linear operator associated with $\SE$ in \eqref{eq:SE}, defined by
\begin{equation*}
        \cS_{\txt{B}}\qa{X}_{ij}:=\delta_{ij}\sum_{k=1}^N \pa{\SE}_{ik}X_{kk}.
\end{equation*}
In the discussion below, we will also define the linear operator
$\cS_S:\C^{N\times N}\to\C^{N\times N}$ associated with other variance profile matrices $S$, i.e.,
\begin{equation}\label{def_cS}
        \cS_S\qa{X}_{ij}:=\delta_{ij}\sum_{k}S_{ik}X_{kk}\quad \txt{for any }X\in\C^{N\times N}.
\end{equation}

Next, we define the flow of variance profiles in the following lemma, which will enable us to rescale the matrix in the proof of \Cref{lem_ConArg} below.
\begin{lemma}[Flow of variance profiles]\label{lemma_variance_flow}
    Define the decreasing family of variance-profile matrices by
    \begin{equation}\label{def_variance_flow}
        \begin{aligned}
            \fS_t:=\ha{\frac{\tf}{s}\SRBM+\pa{1-\frac{\tf}{s}}\SE:s\in[t,\tf]},
        \end{aligned}
    \end{equation}
    where the initial time $\ti$ and the final time $\tf$ are given in \eqref{pick_parameter_1} above. (The local fullness condition ensures that every matrix in $\fS_t$ has non-negative entries.) In particular, at $t=\tf$, we have $\fS_{\tf}=\{\SRBM\}$. Now, fix any $\ti\leq t_1< t_2\leq \tf$ and $S\in t_2\fS_{t_2}$. There exists another variance profile \smash{$\wt S\in t_1\fS_{t_1}$} such that $$S_{t_2}\p{t_1,\wt S}=S,$$
    where $S_{t_2}\p{t_1,\wt S}$ is defined as in \eqref{def_S_t}. Moreover, for any $t\in\qa{t_1,t_2}$, we have
        \begin{equation}
            \begin{aligned}
                S_{t}\p{t_1,\wt S}=\wt S+\pa{t-t_1}\SE\in t\fS_{t}.
            \end{aligned}
        \end{equation}
    This shows that if we take $t_0=t_1$ and $S_{t_0}=\wt S$ in the flow \eqref{def_stochastic_flow}, then the matrix $H_{t_2}$ has variance profile $S$.
\end{lemma}
\begin{proof}
    Taking \smash{\(\wt S=S+\pa{t_1-t_2}\SE\)},
    we have \smash{$S_{t_2}\p{t_1,\wt S}=S$}. Moreover,
    by the definition of $\fS_{t_2}$ in \eqref{def_variance_flow}, we can express $S$ as
    \begin{equation}
        \begin{aligned}
            S=t_2\pa{\frac{\tf}{s}\SRBM+\pa{1-\frac{\tf}{s}}\SE}
        \end{aligned}
    \end{equation}
    for some $s\in\qa{t_2,\tf}$. Then, for any $t\in\qa{t_1,t_2}$, we have
    \begin{equation}
        \begin{aligned}
            &\frac{1}{t}S_{t}\p{t_1,\wt S}=\frac{1}{t}\big[S+\pa{t_1-t_2}\SE+\pa{t-t_1}\SE\big]\\
            =&\frac{1}{t}\qa{t_2\pa{\frac{\tf}{s}\SRBM+\pa{1-\frac{\tf}{s}}\SE}+\pa{t-t_2}\SE}=\frac{\tf}{\wt s}\SRBM+\pa{1-\frac{\tf}{\wt s}}\SE\in\fS_t,
        \end{aligned}
    \end{equation}
    where we denote $\wt s:=st/t_2\in\qa{t,\tf}$. This concludes the proof.
\end{proof}

\subsection{Loop hierarchy and primitive loops}\label{subsec:loops}

Our proofs of the main results rely on analyzing the behavior of $G$-loops and primitive loops along the stochastic and deterministic flows. We work within a time interval $\qa{t_0,t_1}\subseteq\qa{\ti,\tf}$. More precisely, for any variance profile $S_{t_1}\in t_1\fS_{t_1}$, \Cref{lemma_variance_flow} guarantees the existence of an $S_{t_0}\in t_0\fS_{t_0}$ such that, along the flow defined in \eqref{def_stochastic_flow}, the matrix $H_{t_1}$ has variance profile $S_{t_1}$.
Moreover, we consider the deterministic flow defined in \eqref{z_t_flow} with parameters chosen as in \eqref{pick_parameter_1}. In the proofs, our focus will be on the dynamics of $G_t$ for $t\in\qa{t_0,t_1}$ and the corresponding $G$-loops, defined as follows.

\begin{definition}[$G$-Loop]\label{Def:G_loop}
Along the matrix flow $H_t$, $t\in\qa{t_0,t_1}$, with variance profile $S_t\in t\fS_t$, we denote for charge $\sigma\in \{+,-\}$ that
 \begin{equation}\label{eq:Gtsig}
  G_{t}(\sigma):=\begin{cases}
       (H_t-z_t)^{-1}, \ \ \text{if} \ \  \sigma=+,\\
        (H_t-\bar z_t)^{-1}, \ \ \text{if} \ \ \sigma=-.
   \end{cases}
 \end{equation}
In other words, we denote $G_{t}(+)\equiv G_{t}$ and $G_{t}(-)\equiv G_{t}^*$. Define for $[a]\in \Zn$ the block identity matrix $I_{\qa{a}}$ and rescaled block identity matrix $E_{\qa{a}}$ as
\be\label{def:Ia} (I_{[a]})_{ij}= \delta_{ij}\cdot \mathbf 1_{i\in [a]} , \quad E_{[a]}=W^{-d}I_{[a]}.\ee
For any $\fn\in \N$, $\bsigma=(\sigma_1, \cdots \sigma_\fn)\in \{+,-\}^\fn$, and $\ba=([a_1], \ldots, [a_\fn])\in (\Zn)^\fn$, we define the $\fn$-$G$-loops by
\begin{equation}\label{Eq:defGLoop}
    {\cal L}^{(\fn)}_{t, \boldsymbol{\sigma}, \ba}:=\left \langle \prod_{i=1}^\fn \left(G_{t}(\sigma_i) E_{[a_i]}\right)\right\rangle .
\end{equation}
Furthermore, we denote
\begin{equation}\label{def_mtzk}
m (\sigma ):= \begin{cases}
    m(\sE) \equiv m_t(z_t(E)),  &\text{if} \ \ \sigma  =+ \\
    \bar m(\sE) \equiv m_t(\bar z_t(E)),  &\text{if} \ \ \sigma = -
\end{cases},
\end{equation}
and introduce the \emph{centered resolvent} as
\begin{equation}\label{Eq:defwtG}
 \Gc_t(\sigma) := G_t(\sigma) -m (\sigma).
\end{equation}
\end{definition}

To define the loop hierarchy for the $G$-loops defined above, we introduce the following loop operations from \cite[Definition 2.10]{Band1D}.

\begin{definition}[Loop operations]\label{Def:oper_loop}
We define the following operations on the $G$-loops in \eqref{Eq:defGLoop}.

 \medskip

\noindent \emph{(1)}
 For $k \in \qq{\fn}$ and $[a]\in \Zn$, we define the first type of cut-and-glue operator ${\cut}^{[a]}_{k}$ as follows:
\be\label{eq:cut1}
    {\cut}^{[a]}_{k} \circ {\cal L}^{(\fn)}_{t, \boldsymbol{\sigma}, \ba}:= \left \langle \prod_{i<k}  \left(G_{t}(\sigma_i) E_{[a_i]}\right)\left( G_{t}(\sigma_k) E_{[a]} G_{t}(\sigma_k)E_{[a_k]}\right)\prod_{i>k}  \left(G_{t}(\sigma_i) E_{[a_i]}\right)\right\rangle.
\ee
In other words, it is the $(\fn+1)$-$G$ loop obtained by replacing $G_t(\sigma_k)$ as $G_{t}(\sigma_k) E_a G_{t}(\sigma_k)$. Graphically, the operator \smash{${\cut}^{[a]}_{k}$} cuts the $k$-th $G$ edge $G_t(\sigma_k)$ and glues the two new ends with $E_{[a]}$.
This operator can also be considered as an operator on the indices $(\boldsymbol{\sigma},\ba)$:
 $${\cut}^{[a]}_{k} (\boldsymbol{\sigma}, \ba) =\big( (\sigma_1,\ldots, \sigma_{k-1}, \sigma_k,\sigma_k ,\sigma_{k+1},\ldots, \sigma_\fn ),([a_1],\ldots, [a_{k-1}], [a],[a_k],[a_{k+1}],\ldots, [a_\fn] )\big).$$
Hence, we will also express \eqref{eq:cut1} as
    $${\cut}^{[a]}_{k} \circ {\cal L}^{(\fn)}_{t, \boldsymbol{\sigma}, \ba}\equiv
     {\cal L}^{(\fn+1)}_{t, \;  {\cut}^{[a]}_{k} (\boldsymbol{\sigma}, \ba)}.
    $$

\medskip

 \noindent
 \emph{(2)} For $k < l \in \qq{\fn}$, we define the second type of cut-and-glue operator ${\cutL}^{[a]}_{k,l}$ from the left (``L") of $k$ as:
 \begin{align}\label{eq:cutL}
 {\cutL}^{[a]}_{k,l} \circ {\cal L}^{(\fn)}_{t, \boldsymbol{\sigma}, \ba}:= \left \langle \prod_{i<k}  \left[G_{t}(\sigma_i) E_{[a_i]}\right]\left( G_{t}(\sigma_k) E_{[a]} G_{t}(\sigma_l)E_{[a_l]}\right)\prod_{i>l}  \left[G_{t}(\sigma_i) E_{[a_i]}\right]\right\rangle,
 \end{align}
and define the third type of cut-and-glue operator ${\cutR}^{[a]}_{k,l}$ from the right (``R") of $k$ as:
 \begin{align}\label{eq:cutR}
 {\cutL}^{[a]}_{k,l} \circ {\cal L}^{(\fn)}_{t, \boldsymbol{\sigma}, \ba}:= \left \langle \prod_{k\le i <l}  \left[G_{t}(\sigma_i) E_{[a_i]}\right]\cdot \left( G_{t}(\sigma_l) E_{[a]}\right)\right\rangle.
 \end{align}
In other words, the second type operator cuts the $k$-th and $l$-th $G$ edges $G_t(\sigma_k)$ and $G_t(\sigma_l)$ and creates two chains: the left chain to the vertex $[a_k]$ is of length $(\fn+k-l+1)$ and contains the vertex $[a_\fn]$, while the right chain to the vertex $[a_k]$ is of length $(l-k+1)$ and does not contain the vertex $[a_\fn]$.
Then, \smash{${\cutL}^{[a]}_{k,l}\circ {\cal L}^{(\fn)}_{t, \boldsymbol{\sigma}, \ba}$} (resp.~\smash{${\cutR}^{[a]}_{k,l}\circ {\cal L}^{(\fn)}_{t, \boldsymbol{\sigma}, \ba}$}) gives a $(\fn+k-l+1)$-loop (resp.~$(l-k+1)$-loop) obtained by gluing the left chain (resp.~right chain) at the new vertex $[a]$.
Again, we can also consider the two operators to be defined on the indices $(\boldsymbol{\sigma},\ba)$:
\begin{align*}
    &{\cutL}^{[a]}_{k,l} (\boldsymbol{\sigma}, \ba) = \big((\sigma_1,\ldots, \sigma_k,\sigma_l ,\ldots, \sigma_\fn ),([a_1],\ldots, [a_{k-1}], [a],[a_l],\ldots, [a_\fn] )\big),\\
    &{\cutR}^{[a]}_{k,l} (\boldsymbol{\sigma}, \ba) = \big((\sigma_k,\ldots, \sigma_l),([a_k],\ldots, [a_{l-1}], [a])\big).
\end{align*}
Hence, we will also express \eqref{eq:cutL} and \eqref{eq:cutR} as
    $${\cutL}^{[a]}_{k,l} \circ {\cal L}^{(\fn)}_{t, \boldsymbol{\sigma}, \ba}\equiv
     {\cal L}^{(\fn+k-l+1)}_{t, \;  {\cutL}^{[a]}_{k,l} (\boldsymbol{\sigma}, \ba)},\quad {\cutR}^{[a]}_{k} \circ {\cal L}^{(\fn)}_{t, \boldsymbol{\sigma}, \ba}\equiv
     {\cal L}^{(l-k+1)}_{t, \;  {\cutR}^{[a]}_{k,l} (\boldsymbol{\sigma}, \ba)}.
    $$

\end{definition}

For $(x,y)\in \pa{\Z_L^d}^2$, we abbreviate $\partial_{xy}:=\partial_{(H_t)_{xy}}$.
Then, by It\^o's formula and equation \eqref{eq:SDE_Gt}, it is easy to derive the following SDE satisfied by the $G$-loops.

\begin{lemma}[Loop hierarchy] \label{lem:SE_basic}
Suppose that $H_t$ evolves along the flow defined in \eqref{def_stochastic_flow}. Then, the $\fn$-$G$ loop satisfies the following SDE, called the \emph{loop hierarchy}:
\begin{align}\label{eq:mainStoflow}
    \dd \mathcal{L}^{(\fn)}_{t, \boldsymbol{\sigma}, \ba} =\dd
    \mathcal{B}^{(\fn)}_{t, \boldsymbol{\sigma}, \ba}
    +\mathcal{W}^{(\fn)}_{t, \boldsymbol{\sigma}, \ba} \dd t
    +  W^d\sum_{1 \le k < l \le \fn} \sum_{[a]}  \pa{\cutL^{[a]}_{k, l} \circ \mathcal{L}^{(\fn)}_{t, \boldsymbol{\sigma}, \ba}  }\cdot \pa{ \cutR^{[a]}_{k, l} \circ \mathcal{L}^{(\fn)}_{t, \boldsymbol{\sigma}, \ba} } \dd t,
\end{align}
 where the martingale term $\mathcal{B}^{(\fn)}_{t, \boldsymbol{\sigma}, \ba}$ and the term $\mathcal{W}^{(\fn)}_{t, \boldsymbol{\sigma}, \ba}$ are defined by
 \begin{align} \label{def_Edif}
\dd\mathcal{B}^{(\fn)}_{t, \boldsymbol{\sigma}, \ba} :  = & \sum_{x,y\in \ZL}
  \left( \partial_{xy}  {\cal L}^{(\fn)}_{t, \boldsymbol{\sigma}, \ba}  \right)
 \cdot \sqrt{\pa{\SE} _{xy}}
  \left(\dd B_t\right)_{xy}, \\\label{def_EwtG}
\mathcal{W}^{(\fn)}_{t, \boldsymbol{\sigma}, \ba}: = &  {W}^d \sum_{k=1}^\fn \sum_{[a]\in \Zn} \;
 \left \langle \Gc_t(\sigma_k) E_{[a]} \right\rangle
   \cdot
  \left( {\cut}^{[a]}_{k} \circ {\cal L}^{(\fn)}_{t, \boldsymbol{\sigma}, \ba} \right) .
\end{align}
\end{lemma}

This loop hierarchy is well approximated by the primitive loops defined in \Cref{Def_Ktza} below. Our definition is given in a recursive form, whereas in \cite{Band1D, truong2025localizationlengthfinitevolumerandom} the primitive loops are formulated in a purely dynamical manner. This recursive construction is essentially the same as those in equations (11.1) and (11.2) of \cite{erdos2025zigzagstrategyrandomband}. Moreover, as we will verify in \eqref{pro_dyncalK}, this recursive scheme coincides with the formulations in \cite[Definition 2.12]{Band1D} and \cite[Definition 3.7]{truong2025localizationlengthfinitevolumerandom}.
For generality, in \Cref{Def_Ktza} we allow an arbitrary spectral parameter $z$ and any variance profile matrix $S$ with nonnegative entries, provided that the matrix $1-m\p{\sigma}m\p{\sigma'}S$ is invertible for all $\sigma,\sigma'\in\ha{+,-}$, where $m\p{+}\equiv m\p{z}$ and $m\p{-}\equiv \ol{m}\p{z}$.

\begin{definition}[Primitive loops]\label{Def_Ktza}
    We define the \emph{entrywise primitive loop} of length $1$ by
    \be\label{eq:loop1}{\wh\cK}^{(1)}_{\sigma,x}=m(\sigma),\quad \forall \sigma\in \{+,-\},\ x\in\ZL.\ee
    For $\fn\ge 1$, we define ${\wh \cK}^{(\fn+1)}_{ \boldsymbol{\sigma}, \bx}$ recursively for $\bsigma\in \{+,-\}^{\fn+1}$ and $\bx\in (\ZL)^{\fn+1}$ by
    \begin{align}\label{wh_cK_recursive_relation}
            \wh \cK_{\bsigma,\bx}^{\pa{\fn+1}}=&~m\p{\sigma_1}\wh \cK_{\pa{\sigma_2,\ldots,\sigma_{\fn+1}},\pa{x_2,\ldots,x_\fn,x_1}}^{\pa{\fn}}\pa{1-m\p{\sigma_1}m\p{\sigma_{\fn+1}}S}^{-1}_{x_1x_{\fn+1}}\\
            +&~m\p{\sigma_1}\sum_{k=2}^{\fn}\sum_{x,y}\wh \cK_{\pa{\sigma_1,\ldots,\sigma_k},\pa{x_1,\ldots,x_{k-1},y}}^{\pa{k}}S_{xy}\cdot\wh \cK_{\pa{\sigma_k,\ldots,\sigma_{\fn+1}},\pa{x_k,\ldots,x_\fn,x}}^{\pa{\fn-k+2}}\pa{1-m\p{\sigma_1}m\p{\sigma_{\fn+1}}S}^{-1}_{xx_{\fn+1}}. \nonumber
        \end{align}
    If the spectral parameter is set to \(z=\sE\) and the variance profile is given by \(S_t\) from \eqref{def_S_t}, then \( m(+)=m(\sE)\) and \(m(-)=\overline{m}(\sE),\) and we write the corresponding primitive loop as \smash{$\wh \cK_{t,\bsigma,\bx}^{\pa{\fn}}$}. Additionally, \smash{$\wh \cK_{t,\bsigma,\bx}^{\pa{\fn}}$} satisfies the following system of differential equations (see \Cref{lemma_evolution_of_primitive_loops}):
     \begin{align}\label{entrywise_pro_dyncalK}
       \frac{\dd}{\dd t}\,{\wh\cK}^{(\fn)}_{t, \boldsymbol{\sigma}, \bx}
       =
        \sum_{1\le k < l \le \fn} \sum_{a, b\in \ZL}  \pa{\cutL^{\pa{a}}_{k, l} \circ \wh{\mathcal{K}}^{(\fn)}_{t, \boldsymbol{\sigma}, \bx} } \cdot \pa{\SE}_{ab}\cdot  \pa{\cutR^{\pa{b}}_{k, l} \circ \wh{\mathcal{K}}^{(\fn)}_{t, \boldsymbol{\sigma}, \bx} } ,
    \end{align}
    where the entrywise operators $\cutL$ and $\cutR$ (analogous to the blockwise operators in \Cref{Def:oper_loop}) act on \smash{${\wh\cK}^{(\fn)}_{t, \boldsymbol{\sigma}, \bx}$} through the actions on indices:
  \be\label{entrywise_calGonIND}
     {\cutL}^{(a)}_{k,l}  \circ {\wh\cK}^{(\fn)}_{t, \boldsymbol{\sigma}, \bx} := {\wh\cK}^{(\fn+k-l+1)}_{t,  {\cutL}^{\pa{a}}_{k,l}  (\boldsymbol{\sigma}, \bx)} , \quad
     \quad  {\cutR}^{(b)}_{k,l}  \circ {\wh\cK}^{(\fn)}_{t, \boldsymbol{\sigma}, \bx} := {\wh\cK}^{(l-k+1)}_{t,  {\cutR}^{\pa{b}}_{k,l}  (\boldsymbol{\sigma}, \bx)}.
       \ee

    As the deterministic limit of the $\cL$-loops, we define the (blockwise) {\bf primitive loops} for $\fn\in \N$, $\bsigma\in\ha{+,-}^{\fn}$, and \smash{$\ba\in\p{\Zn}^{\fn}$} as the averaged tensor of the entrywise primitive loop (recall \eqref{def_average_tensor}):
    \begin{equation}\label{def_cK_block}
            \cK^{(\fn)}_{t,\bsigma,\ba}:=\qq{\wh \cK^{(\fn)}_{t,\bsigma,\cdot}}_{\ba}=W^{-\fn d}\sum_{\bx\in\qa{a_1}\times\cdots\times \qa{a_{\fn}}}\wh \cK^{(\fn)}_{t,\bsigma,\bx}.
    \end{equation}
    Equation \eqref{entrywise_pro_dyncalK} then yields
    \begin{align}\label{pro_dyncalK}
       \frac{\dd}{\dd t}\,{\cK}^{(\fn)}_{t, \boldsymbol{\sigma}, \ba}
       =
        W^d\sum_{1\le k < l \le \fn} \sum_{\qa{a}\in \Zn}  \pa{\cutL^{\qa{a}}_{k, l} \circ {\mathcal{K}}^{(\fn)}_{t, \boldsymbol{\sigma}, \ba}}  \cdot \pa{ \cutR^{\qa{a}}_{k, l} \circ {\mathcal{K}}^{(\fn)}_{t, \boldsymbol{\sigma}, \ba} } ,
    \end{align}
   where the operators $\cutL$ and $\cutR$ again act on \smash{${\cal K}^{(\fn)}_{t, \boldsymbol{\sigma}, \ba}$} through the actions on indices:
  \be\label{calGonIND}
     {\cutL}^{[a]}_{k,l}  \circ {\cK}^{(\fn)}_{t, \boldsymbol{\sigma}, \ba} := {\cK}^{(\fn+k-l+1)}_{t,  {\cutL}^{\qa{a}}_{k,l}  (\boldsymbol{\sigma}, \ba)} , \quad
     \quad  {\cutR}^{[a]}_{k,l}  \circ {\cK}^{(\fn)}_{t, \boldsymbol{\sigma}, \ba} := {\cK}^{(l-k+1)}_{t,  {\cutR}^{\qa{a}}_{k,l}  (\boldsymbol{\sigma}, \ba)}.
       \ee
 Equation \eqref{pro_dyncalK} is precisely the defining equation for primitive loops in \cite{Band1D, truong2025localizationlengthfinitevolumerandom}. Hence, our definition coincides with theirs.
\end{definition}

For clarity of presentation, we will also call $G$-loops and primitive loops as $\cL$-loops and $\cK$-loops, respectively. Moreover, we will call \smash{$(\cL-{\cal K})^{(\fn)}_{t, \boldsymbol{\sigma}, \ba}\equiv \cL^{(\fn)}_{t, \boldsymbol{\sigma}, \ba}-\cK^{(\fn)}_{t, \boldsymbol{\sigma}, \ba}$} an $(\cL-\cK)$-loop.

\subsection{Propagators and evolution kernels}\label{subsec:propagators}
Similar to previous works \cite{Band1D,Band2D,truong2025localizationlengthfinitevolumerandom,Bandedge}, the $\Theta$-propagators defined in \Cref{def_Theta_propagators} below served as a basic tool for the analysis of the loop hierarchy and primitive loops. In particular, they are the basic building blocks for the tree representation formulas of $\cK$-loops therein.

\begin{definition}[$\Theta$-propagators]\label{def_Theta_propagators}
Given any pair of charges $\sigma,\sigma'\in\{+,-\}$, we define the entrywise propagator at time $t$ as the $\ZL\times \ZL$ matrix:
    \begin{equation}\label{eq:defwhTheta}
            \wh\Theta_t^{\p{\sigma,\sigma'}}:=\pa{1-m\p{\sigma}m\p{\sigma'}S_t}^{-1},
    \end{equation}
    where $S_t$ is defined in \eqref{def_S_t}. The corresponding (blockwise) propagator is the $\Zn\times \Zn$ matrix obtained by projecting the entrywise propagator:
    \begin{equation}\label{def_theta}
        \begin{aligned}
            \Theta_t^{\p{\sigma,\sigma'}}:=\mathscr{P}\p{\wh \Theta_t^{\p{\sigma,\sigma'}}},
        \end{aligned}
    \end{equation}
    where the projection operator $\scr{P}$ acts on any matrix $A\in\C^{\ZL\times \ZL}$ by
    \begin{equation}\label{def_projection_operator}
            \scr{P}\pa{A}_{(\qa{a},\qa{b})}:=\frac{1}{W^d}\sum_{x\in\qa{a},y\in\qa{b}}A_{xy},\quad \forall \ \qa{a},\qa{b}\in\Zn.
    \end{equation}
    Moreover, by the definition of the $\cK$-loops, we immediately obtain
    \begin{equation}\label{cK_and_Theta}
        \begin{aligned}
            \cK_{t,\pa{\sigma,\sigma'},\p{\qa{a},\qa{b}}}^{\pa{2}}=W^{-d}m\p{\sigma}m\p{\sigma'}\Theta_t^{\p{\sigma,\sigma'}}\pa{\qa{a},\qa{b}},\quad \forall \ \qa{a},\qa{b}\in\Zn.
        \end{aligned}
    \end{equation}
\end{definition}

We now list several basic properties of $\Theta$-propagators, whose proofs will be presented in \Cref{sec_analysis_of_primitive_loops}.
\begin{lemma}[Properties of $\Theta$-propagators]\label{lemma_properties_of_propagators}
    Fix a time parameter $t\in \qa{\ti,\tf}$ and a variance profile $S_t\in t\fS_t$. We define
    \begin{equation}\label{eq:ellt}
        \begin{aligned}
            \ell_t:=\min\pa{\lambda\pa{1-t}^{-1/2}+1,n}.
        \end{aligned}
    \end{equation}
    Note that $\ell_t\sim \ell\pa{\eta_t}$, where $\ell(\eta_t)$ is defined in \eqref{eq:elleta} and $\eta_t$ is given in \eqref{E_t_eta_t_flow}. The following properties hold for the propagators introduced in \Cref{def_Theta_propagators}.

    \begin{enumerate}
    \item {\bf Transposition}: We have $(\Theta_{t}^{(\sigma_1,\sigma_2)})^\top =\Theta_{t}^{(\sigma_2,\sigma_1)}.$

\item {\bf Symmetry}: For any $[x],[y],[a]\in \Zn$, we have
\be\label{symmetry}
    \Theta^{(\sigma_1,\sigma_2)}_{t}([x]+[a],[y]+[a])= \Theta^{(\sigma_1,\sigma_2)}_{t}([x],[y]).
    \ee
In addition, we have the parity symmetry: $\Theta^{(\sigma_1,\sigma_2)}_{t}(0,[x])= \Theta^{(\sigma_1,\sigma_2)}_{t}(0,-[x]).$

\item {\bf Exponential decay on length scale $\ell_t$}: There exists a constant $c>0$ such that, for any large constant $D>0$, the following estimate holds when $\sigma_1\ne \sigma_2$:
\begin{equation}\label{prop:ThfadC}
\qquad\quad \Theta^{(\sigma_1,\sigma_2)}_{t}(0,[x])\prec \frac{e^{-c |[x]|/ {\ell}_t}}{|1-t|  {\ell}_t^d}+ W^{-D}.
\end{equation}
When $\sigma_1=\sigma_2$, there exists a (possibly different) constant $c>0$ such that
\begin{equation}\label{prop:ThfadC_short}     \Theta^{(\sigma_1,\sigma_2)}_{t}(0,[x])\le c^{-1}e^{-c\absa{\qa{x}}}.
\end{equation}
Moreover, for $\qa{x}\neq \qa{0}$, the bounds \eqref{prop:ThfadC} and \eqref{prop:ThfadC_short} can be improved as follows: for $\sigma_1\neq \sigma_2$,
\begin{equation}\label{prop:ThfadC_improved}
    \begin{aligned}
         \qquad\quad \Theta^{(\sigma_1,\sigma_2)}_{t}(0,[x])\prec \frac{\lambda^2}{\lambda^2+(1-t)}\frac{e^{-c |[x]|/ {\ell}_t}}{|1-t|  {\ell}_t^d}+ W^{-D},
    \end{aligned}
\end{equation}
and for $\sigma_1=\sigma_2$,
\begin{equation}\label{prop:ThfadC_short_improved}
    \begin{aligned}
         \qquad\quad \Theta^{(\sigma_1,\sigma_2)}_{t}(0,[x])\prec \lambda^2 e^{-c|\qa{x}|}.
    \end{aligned}
\end{equation}

\item {\bf Entrywise bounds}: There exists a constant $c>0$ such that the following estimates hold for the entrywise propagator:
\begin{align}\label{prop:ThfadC_pointwise}
\wh \Theta_t^{\pa{\sigma_1,\sigma_2}}\pa{x,y}&\prec \delta_{xy}+ \frac{e^{-c |x-y|/(W {\ell}_t)}}{W^d  {\ell}_t^d|1-t|}+ W^{-D},\quad \txt{ for }\ \ \sigma_1\neq\sigma_2,\\
\label{prop:ThfadC_short_pointwise}
\wh \Theta_t^{\pa{\sigma_1,\sigma_2}}\pa{x,y}&\prec \delta_{xy}+ \frac{e^{-c |x-y|/W}}{W^d}+ W^{-D},\quad  \txt{ for }\ \  \sigma_1=\sigma_2.
\end{align}

\item {\bf First-order finite difference}: The following estimate holds for all $[x],[y]\in \Zn$ and $\sigma_1\ne \sigma_2$: \begin{equation}\label{prop:BD1}
    \left| \Theta^{(\sigma_1,\sigma_2)}_{t}(0, [x])-\Theta^{(\sigma_1,\sigma_2)}_{t}(0, [y])\right|\prec \frac{1}{ \lambda^{2}+(1-t)}\frac{|[x]-[y]|}{\langle [x]\rangle^{d-1}+\avga{[y]}^{d-1}}.
     \end{equation}

\item {\bf Second-order finite difference}: The following estimate holds for all $[x],[y]\in \Zn$ and $\sigma_1\ne \sigma_2$:
\begin{equation}\label{prop:BD2}
 \Theta^{(\sigma_1,\sigma_2)}_{t} (0,[x]+[y]) + \Theta^{(\sigma_1,\sigma_2)}_{t} (0,[x]-[y])-  2\Theta^{(\sigma_1,\sigma_2)}_{t} (0,[x])
\prec \frac{1}{\lambda^{2}+(1-t)}\frac{|[y]|^2}{\langle [x]\rangle^{d}} .
 \end{equation}
\end{enumerate}
\end{lemma}

In the main proof, the $\Theta$-propagators naturally arise in the evolution equations. We now introduce the associated evolution operators.

\begin{definition}[Evolution kernel]\label{def_evolution_kernel}
Let $S_t$ evolve according to \eqref{def_S_t} for $t\in [t_0,t_1]$, and let $\bsigma=(\sigma_1,\ldots,\sigma_\fn)\in \{+,-\}^\fn$. For any $\fn$-dimensional tensor ${\cal A}: (\Zn)^{\fn}\to \mathbb C$, we define the linear operator \smash{${\vartheta}^{(\fn)}_{t, \boldsymbol{\sigma}}$} by
\begin{align}\label{def:op_thn}
    \left({{\vartheta}}^{(\fn)}_{t, \boldsymbol{\sigma}} \circ{ \mathcal{A}}\right)_{\ba} := W^d \sum_{i=1}^\fn \sum_{[b_i]\in\Zn} \cK_{t,\pa{\sigma_i,\sigma_{i+1}},\pa{[a_i], [b_i]}}^{\pa{2}}  \mathcal{A}_{\ba^{(i)}([b_i])},\quad \forall \ba=([a_1],\ldots, [a_\fn])\in (\Zn)^\fn,
\end{align}
where we adopt the convention $\sigma_{\fn+1}=\sigma_1$, and $\ba^{(i)}$ is defined as
\begin{align}
    \ba^{(i)}([b_i]):= ([a_1], \ldots, [a_{i-1}], [b_i], [a_{i+1}], \ldots, [a_\fn]).
\end{align}
The evolution kernel on an interval $\qa{s,t}\subseteq\qa{t_0,t_1}$ corresponding to ${{\vartheta}}^{(\fn)}_{t, \boldsymbol{\sigma}}$  is defined by
    \begin{align}\label{def_Ustz}
        \left(\mathcal{U}_{s, t, \boldsymbol{\sigma}}^{(\fn)}\circ{ \mathcal{A}}\right)_{\ba}:=& \sum_{\mathbf b\in (\Zn)^\fn} \prod_{i=1}^\fn \scr{P}\left[\p{1-m\p{\sigma_i}m\p{\sigma_{i+1}}S_t}^{-1}\p{1-m\p{\sigma_i}m\p{\sigma_{i+1}}S_s}\right]_{[a_i] [b_i]} \cdot \mathcal{A}_{\mathbf{b}}\\
        =& \sum_{\mathbf b\in (\Zn)^\fn} \prod_{i=1}^\fn \left[I+\pa{t-s}m\p{\sigma_i}m\p{\sigma_{i+1}}\Theta_t^{\p{\sigma_i,\sigma_{i+1}}}\right]_{[a_i] [b_i]} \cdot \mathcal{A}_{\mathbf{b}}, \quad \text{with} \ \ \mathbf{b} = ([b_1], \ldots, [b_{\fn}]),\nonumber
    \end{align}
where $\scr{P}$ is defined in \eqref{def_projection_operator}. Using the identity $\partial_t\Theta_t^{\p{\sigma_i,\sigma_{i+1}}}=W^d\Theta_t^{\p{\sigma_i,\sigma_{i+1}}}\cK_{t,\p{\sigma_i,\sigma_{i+1}}}^{\pa{2}}$, we obtain
    \begin{equation}
        \begin{aligned}
            \frac{\rd}{\rd t}\mathcal{U}_{s, t, \boldsymbol{\sigma}}^{(\fn)}\circ{\cal A}={\vartheta^{(\fn)}_{t,\bsigma}}\circ\pa{\mathcal{U}_{s, t, \boldsymbol{\sigma}}^{(\fn)}\circ{\cal A}}, \quad \txt{for any }\fn\txt{-dimensional tensor } \cal A.
        \end{aligned}
    \end{equation}
\end{definition}

With the propagator estimates established in \Cref{lemma_properties_of_propagators}, one can derive the following bounds on the evolution kernel in \Cref{lem:sum_Ndecay,TailtoTail,lem:sum_decay}. The proofs follow the same strategy as in \cite{Band1D,Band2D,truong2025localizationlengthfinitevolumerandom}; see, for example, the proofs of Lemmas 7.1--7.3 in \cite{Band1D} and Lemma 7.7 in \cite{truong2025localizationlengthfinitevolumerandom}. For this reason, we omit the details.

\begin{lemma}
\label{lem:sum_Ndecay}
Let \smash{${\cal A}: (\Zn)^{\fn}\to \mathbb C$} be an $\fn$-dimensional tensor for a fixed $\fn\in \N$ with $\fn\ge 2$. Then, for each $t_0\le s \le t \le t_1$, we have that
\begin{align}\label{sum_res_Ndecay}
   \| {\cal U}_{s,t,\boldsymbol{\sigma}}\;\circ {\cal A}\|_{\infty} \prec \left(  \eta_s/ \eta_t\right)^{\fn }\cdot \|{\cal A}\|_{\infty} ,
\end{align}
where the $L^\infty$-norm of ${\cal A}$ is defined as $\|{\cal A}\|_{\infty}=\max_{\ba\in (\Zn)^\fn}|\cal A_{\ba}|$.
\end{lemma}

Given $t_0\leq s<t\leq t_1$, if a two-dimensional tensor $\cal A$ decays exponentially on the scale $\ell_s$, we can show that ${\cal U}^{(2)}_{s,t,\bsigma}\circ \cal A$ decays on the scale $\ell_t$ for $\bsigma=(+,-)$ or $(-,+)$.

\begin{lemma}
\label{TailtoTail}
For any $t\in[\ti,\tf]$ and $\ell\ge 0$, introduce the notation
 $${\cal T}_{t}(\ell) := (W^d\ell_t^d\eta_t)^{-2} \exp \big(- \left( \ell /\ell_t\right)^{1/2} \big).
$$
For $\bsigma\in\{(+,-),(-,+)\}$ and $s\in [t_0,t_1]$, suppose ${\cal A}_{\ba}$ satisfies that
$$
|{\cal A}_{\ba}|\le {\cal T}_{s}(|[a_1]-[a_2]|)+W^{-D},\quad \forall \ba=([a_1],[a_2]),
$$
for some constant $D>0$.
Then, for any $t\in[s,t_1]$, we have that
\begin{align}
\label{neiwuj}
\left({\cal U}^{(2)}_{s,t,\boldsymbol{\sigma}} \circ
{\cal A}\right)_{\ba} & \prec
 {\cal T}_{t}( |[a_1]-[a_2]|)+W^{-D} \cdot(\eta_s/\eta_t)^2, \quad {\rm if}\quad  |[a_1]-[a_2]|\ge (\log W)^{3/2}\ell_t .
\end{align}
\end{lemma}

If the tensor $\cal A$ exhibits faster-than-polynomial decay at scales larger than $\ell_s$, then we obtain a stronger bound than the $({\infty\to \infty})$-norm bound given by \Cref{lem:sum_Ndecay}. This bound can be further improved if $\cal A$ satisfies certain sum-zero property or symmetry.

\begin{lemma}\label{lem:sum_decay}
Let ${\cal A}: (\Zn)^{\fn}\to \mathbb C$ be an $\fn$-dimensional tensor for a fixed $\fn\in \N$ with $\fn\ge 2$. Suppose it satisfies the following property for some small constant $\e\in(0,1)$ and large constant $D>1$,
\begin{equation}\label{deccA0}
\max_{i,j\in \Zn}|[a_i]-[a_j]|\ge W^{\e}\ell_s \ \ \text{for} \ \ \ba=([a_1],\ldots, [a_\fn])\in (\Zn)^{\fn}  \implies  |\cal A_{\ba}|\le W^{-D} .
\end{equation}
Fix any $t_0\le s \le t \le t_1$ such that $(1-t)/(1-s)\ge W^{-d}$.
There exists a constant $C_\fn>0$ that does not depend on  $\e$ or $D$ such that the following bound holds (note  $\ell_t^d\eta_t\lesssim\ell_s^d\eta_s$ for $d\in\{1,2\}$ by the definition \eqref{eq:ellt}):
\begin{align}\label{sum_res_1}
    \left\|{\cal U}^{(\fn)}_{s,t,\boldsymbol{\sigma}} \circ {\cal A}\right\|_\infty \le W^{C_\fn\e}\frac{\ell_t^d }{\ell_s^d }\left(\frac{\ell_s^d \eta_s}{\ell_t^d \eta_t}\right)^{\fn} \|{\cal A}\|_\infty
    +W^{-D+C_\fn}.
\end{align}
In addition, stronger bounds hold in the following cases:
\begin{itemize}
\item[(I)] If we have $\sigma_1=\sigma_2$ for $\boldsymbol{\sigma}=(\sigma_1,\cdots, \sigma_\fn),$ then
\begin{align}\label{sum_res_2_NAL}
    \left\|{\cal U}^{(\fn)}_{s,t,\boldsymbol{\sigma}} \circ {\cal A}\right\|_\infty \le W^{C_\fn\e}  \left(\frac{\ell_s^d\eta_s}{\ell_t^d\eta_t}\right)^{\fn-1}  \|{\cal A}\|_{\infty}
   +W^{-D+C_\fn}
\end{align}
for a constant $C_\fn>0$ that does not depend on $\e$ or $D$.

\item[(II)] Suppose ${\cal A}$ satisfies the following sum-zero property:
\begin{align}\label{sumAzero}
 \sum_{[a_2],\ldots,[a_\fn]\in \Zn}{\cal A}_{\ba}=0 \quad \forall [a_1]\in \Zn.
\end{align}
Then, the following bound holds for a constant $C_\fn>0$ that does not depend on $\e$ or $D$:
\begin{align}\label{sum_res_2}
    \left\|{\cal U}^{(\fn)}_{s,t,\boldsymbol{\sigma}} \circ {\cal A}\right\|_\infty \le W^{C_\fn\e} \frac{\ell_t^{d-1}}{\ell_s^{d-1}} \left(\frac{\ell_s^d\eta_s}{\ell_t^d\eta_t}\right)^{\fn}  \|{\cal A}\|_{\infty}
   +W^{-D+C_\fn}.
\end{align}
If ${\cal A}$ further satisfies the following symmetry:
\be\label{eq:A_zero_sym}
\cal A_{([a],[a]+[b_2],\ldots, [a]+[b_\fn])} = \cal A_{([a],[a]-[b_2],\ldots, [a]-[b_\fn])},\quad \forall [a],[b_2],\ldots, [b_\fn]\in \Zn,
\ee
then the bound \eqref{sum_res_2} can be improved as follows:\footnote{
In previous works \cite{Band1D,Band2D,truong2025localizationlengthfinitevolumerandom}, the improved evolution bound was established only with the factor $\pa{\ell_s^d\eta_s/\ell_t^d\eta_t}^{\fn}$. However, a careful inspection of the arguments therein—see, for example, the proof of Lemma 7.7 in \cite{truong2025localizationlengthfinitevolumerandom}—reveals that the exponent can in fact be reduced to $\fn-1$.}
\begin{align}\label{sum_res_2_sym}
    \left\|{\cal U}^{(\fn)}_{s,t,\boldsymbol{\sigma}} \circ {\cal A}\right\|_\infty \le W^{C_\fn\e} \left(\frac{\ell_s^d\eta_s}{\ell_t^d\eta_t}\right)^{\fn-1}  \|{\cal A}\|_{\infty}
   +W^{-D+C_\fn}.
\end{align}
\end{itemize}

\end{lemma}

\section{Deterministic analysis}
\label{sec_analysis_of_primitive_loops}

Throughout this section, we fix a target spectral parameter $z=E+\ii \eta$ with $|E|\leq 2-\kappa$ and $\eta\in\qa{W^{\fc}\eta_*,\fc^{-1}}$ for some small constants $\kappa,\fc>0$, as specified in \Cref{sec_preliminaries}. We then select the deterministic flow according to \Cref{lem:paraselect}.
Our aim is to state and prove several deterministic estimates for the $\cK$-loops and $\Theta$-propagators that will be used in the subsequent analysis. These types of estimates have previously been established in \cite{Band1D, Band2D, truong2025localizationlengthfinitevolumerandom} for random band matrices with block variance profiles. The derivations in those works rely on the random walk representation of the $\Theta$-propagators and the tree representation formula for $\cK$-loops.
However, for the more general, non-translation-invariant variance profiles considered in \Cref{def_considered_model}, the conventional random walk representation of the $\Theta$-propagators is no longer applicable. Furthermore, in our setting, the tree representation formula can only be formulated for entrywise primitive loops, which are composed of entrywise $\Theta$-propagators. Since we only have access to the pointwise upper bounds \eqref{prop:ThfadC_pointwise} and \eqref{prop:ThfadC_short_pointwise}, and lack sufficient control over their first- and second-order finite differences, we are unable to derive satisfactory estimates for the entrywise primitive loops directly.

To establish the bounds \eqref{prop:ThfadC}–\eqref{prop:BD2} for both the (entrywise and blockwise) $\Theta$-propagators, we utilize the variance-profile flow \eqref{def_S_t} and develop a dynamical approach that yields a new random–walk representation of \smash{$\Theta_t^{\pa{+,-}}$}. The resulting transition matrix enjoys properties analogous to those of $\scr{P}\pa{S}=W^d\qq{S}$ (see \eqref{def_average_tensor} and \eqref{def_projection_operator}). This allows us to apply the random–walk and Fourier analytic techniques used in the proofs of \cite[Lemma 3.10]{truong2025localizationlengthfinitevolumerandom} and \cite[Lemma 2.14]{Band2D}. Further details are provided in the proof of \Cref{lemma_properties_of_propagators}.

To estimate the $\cK$-loops, we again rely on a dynamical analysis via their evolution equation (see \eqref{pro_dyncalK}). The general strategy is to first derive bounds in the global regime and then propagate these bounds into the local regime using the evolution kernel defined in \Cref{def_evolution_kernel}. However, the standard estimates on the evolution kernel from \Cref{lem:sum_Ndecay,lem:sum_decay} are insufficient for obtaining the optimal bounds. To overcome this, we make further use of the structural features of the evolution equation \eqref{pro_dyncalK} and identify a crucial cancellation, which ultimately yields the optimal estimates on $\cK$-loops.
We also note that a related dynamical analysis of $\cK$-loops was carried out independently in \cite[Section 11]{erdos2025zigzagstrategyrandomband}. Our approach, however, requires additional work: unlike \cite{erdos2025zigzagstrategyrandomband}, which assumes sufficiently strong bounds on the entrywise $\Theta$-propagators, we must exploit further block-reduction structures arising from the variance-profile flow \eqref{def_S_t} and from the definition of the $\cK$-loops themselves.

\subsection{Properties of primitive loops}

As shown in \cite[Section 3]{Band1D} and \cite[Section 4]{truong2025localizationlengthfinitevolumerandom}, the $\cK$-loops defined in \Cref{Def_Ktza} admit a tree–representation formula, from which many fundamental properties follow directly, using essentially the same arguments as in \cite{Band1D,truong2025localizationlengthfinitevolumerandom}. For completeness, we collect these properties in \Cref{lem:baiscK}. We remark that all of them can, in principle, also be derived from the recursive relation \eqref{wh_cK_recursive_relation} together with the evolution equation \eqref{entrywise_pro_dyncalK}, but we do not pursue this approach here in order to keep the presentation streamlined.

\begin{lemma}\label{lem:baiscK}
Let $S$ be any symmetric doubly stochastic matrix defined on $\ZL$, and let $t \in (0,1]$. With a slight abuse of notation, define the entrywise primitive loop \smash{$\wh\cK^{(\fn)}_{t,\bsigma,\bx}$} for any $\fn\in\N$, $\bsigma\in\{+,-\}^{\fn}$, and $\bx\in(\ZL)^{\fn}$ by replacing $S$ with $tS$ in \eqref{wh_cK_recursive_relation} and fixing the spectral parameter at $\sE$, i.e., $m(+)=m(\sE)$ and $m(-)=\overline{m}(\sE)$ (where $\sE$ is selected as in \eqref{pick_parameter_1}). The corresponding blockwise primitive loops are then defined via \eqref{def_cK_block} by
\[\cK_{t,\bsigma,\ba}^{\pa{n}}=\qq{\wh \cK_{t,\bsigma,\cdot}^{\pa{n}}}_{\ba},\quad \forall\ \ba\in\p{\Zn}^{\fn}.
\]
Then, the following statements hold.
    \begin{enumerate}
        \item {\bf Ward's identity}: For any integer $\fn\ge 2$ and $\bsigma\in\ha{+,-}^{\fn}$ with $\sigma_1=-\sigma_{\fn}$, the primitive loops satisfy the following \emph{Ward identity} at vertex $[a_\fn]$:
\begin{align}\label{WI_calK}
\sum_{a_\fn}{\cal K}^{(\fn)}_{t, \boldsymbol{\sigma}, \ba}=
\frac{1}{2\ii W^d\eta_t}\left( {\cal K}^{(\fn-1)}_{t,  \wh\bsig^{(+,\fn)}, \wh\ba^{(\fn)}}- {\cal K}^{(\fn-1)}_{t,  \wh\bsig^{(-,\fn)}, \wh\ba^{(\fn)}}\right) ,
\end{align}
where $\wh\bsig^{(\pm,\fn)}$ is defined by removing $\sigma_\fn$ from $\boldsymbol{\sigma}$ and replacing $\sigma_1$ with $\pm$, i.e., $\wh\bsig^{(\pm,\fn)}=(\pm, \sigma_2, \cdots \sigma_{\fn-1})$, and  $\wh\ba^{(\fn)}$ is obtained by removing $a_\fn$ from $\ba$, i.e.,
$\wh\ba^{(\fn)}=(a_1, a_2,\cdots, a_{\fn-1}).$

        \item {\bf Shift invariance}: For any cyclic shift $\tau_k$ acting on $\pa{s_1,\dots,s_\fn}$ by $\tau_{k}\pa{s_1\ldots,s_\fn}=\pa{s_{k+1},\ldots,s_{k+\fn}}$,with the convention $s_i \equiv s_j$ whenever $i=j \mod \fn$, we have for any $k\in{1,\dots,\fn}$:
        \begin{equation}\label{shift_invariance}
                \wh \cK_{t,\bsigma,\bx}^{\pa{\fn}}=\wh \cK_{t,\tau_k{\bsigma},\tau_k\bx}^{\pa{\fn}}.
        \end{equation}
        \item {\bf Translation invariance}: If $S$ satisfies the block translation invariance  \eqref{def_block_translation_invarinace}, then the tensor \smash{$\cal{A}_{\ba}:=\cK_{t,\bsigma,\ba}^{\pa{\fn}}$} is translationally invariant in the sense
        \begin{equation}\label{eq:traninv}
                 \cal A_{([a]+\qa{b_1},[a]+[b_2],\ldots, [a]+[b_\fn])} =\cal A_{(\qa{b_1},[b_2],\ldots, [b_\fn])},\quad \forall [a],\qa{b_1},[b_2],\ldots, [b_\fn]\in \Zn.
        \end{equation}
        \item {\bf Parity symmetry}: If $S$ satisfies both block translation invariance \eqref{def_block_translation_invarinace} and parity symmetry in \eqref{parity_symmetry_block_2D}, then the tensor \smash{$\cal{A}_{\ba}:=\cK_{t,\bsigma,\ba}^{\pa{\fn}}$} satisfies
        \begin{equation}\label{partity_symmetry_cK}
                \cal A_{([a],[a]+[b_2],\ldots, [a]+[b_\fn])} = \cal A_{([a],[a]-[b_2],\ldots, [a]-[b_\fn])},\quad \forall [a],[b_2],\ldots, [b_\fn]\in \Zn.
        \end{equation}
    \end{enumerate}
\end{lemma}
\begin{proof}
Using the recursive equation \eqref{wh_cK_recursive_relation}, and applying an argument similar to the proof of \Cref{lemma_evolution_of_primitive_loops} below, one obtains the following evolution equation:
    \begin{equation}\label{dtKgen}
    \partial_t \Kgen_{t,\bm{\sigma},\bx}^{(\fn)}
      = \sum_{1 \leq k < \ell \leq \fn} \sum_{c, d\in \ZL}
        \pa{(\mathcal{G}_L)^{\p{c}}_{k\ell} \circ \Kgen^{(\fn)}_{t,\bm{\sigma},\bx}}
        S_{cd}
        \pa{(\mathcal{G}_R)^{\p{d}}_{k\ell} \circ \Kgen^{(\fn)}_{t,\bm{\sigma},\bx}}.
  \end{equation}
  At the initial time $t=0$, the recursive relation \eqref{wh_cK_recursive_relation} reduces to
  \begin{equation}\label{dtKgen_initial}
      \begin{aligned}
          \wh \cK_{0,\bsigma,\bx}^{\pa{\fn+1}}=m\p{\sigma_1}\wh \cK_{0,\pa{\sigma_2,\ldots,\sigma_{\fn+1}},\pa{x_2,\ldots,x_\fn,x_1}}^{\pa{\fn}}\delta_{x_1x_{\fn+1}},
      \end{aligned}
  \end{equation}
  which, together with \eqref{eq:loop1}, yields that $\wh \cK_{0,\bsigma,\bx}^{\pa{\fn}}=\prod_{i=1}^{\fn}\q{m\pa{\sigma_i}\delta_{x_ix_{i+1}}}$. The system of ODEs \eqref{dtKgen} with the initial condition \eqref{dtKgen_initial} coincides with the defining evolution of the \smash{$\Kgen$}-loops in \cite[Definition 4.15]{truong2025localizationlengthfinitevolumerandom}. Consequently, Ward’s identity \eqref{WI_calK} follows directly from \cite[Lemma 4.20]{truong2025localizationlengthfinitevolumerandom}. Properties (ii)–(iii) then follow from the tree–representation formula for \smash{$\Kgen$}-loops, given in \cite[Lemma 4.16]{truong2025localizationlengthfinitevolumerandom}. We omit the details.
\end{proof}
We now establish several upper bounds for the primitive loops. For clarity of presentation, we work under the setting introduced at the beginning of \Cref{sec_preliminaries} and use the notations developed therein.
\begin{lemma}[Bounds on primitive loops]\label{bounds_on_primitive_loops}
    Consider any $t\in \qa{\ti,\tf}$ and let $S_t\in t\fS_t$ (recall \eqref{pick_parameter_1} and \eqref{def_variance_flow}). For any integer $\fn\geq 2$, $\bsigma\in\ha{+,-}^{\fn}$, and $\ba\in\p{\Zn}^\fn$, denote the $\cK$-loop defined in \Cref{Def_Ktza} as \smash{$\cK_{t,\bsigma,\ba}^{\pa{\fn}}$}.
    \begin{enumerate}
        \item[(1)] {\bf Fast decay}: For any constant $\varepsilon>0$ and $D>0$, we have
        \begin{equation}\label{fast_decay_K}
                \max_{i,j\in \qq{\fn}}|[a_i]-[a_j]|\ge W^{\e} \ell_t \ \implies \ \cK_{t,\bsigma,\ba}^{\pa{\fn}}=\OO(W^{-D}) \quad \text{for}\quad \ba=([a_1],[a_2],\ldots, [a_\fn]) .
        \end{equation}
        \item[(2)] {\bf Upper bound}: For any $\bsigma\in\ha{+,-}^{\fn}$ and $\ba\in\p{\Zn}^\fn$, we have
        \begin{equation}\label{eq:bcal_k}
            {\cal K}^{(\fn)}_{t, \boldsymbol{\sigma}, \ba}  \prec \left(W^d \ell_t^d (1-t) \right)^{-\fn+1}\sim \left(W^d \ell_t^d \eta_t \right)^{-\fn+1}.
        \end{equation}
        Moreover, if $\qa{a_i}\neq \qa{a_j}$ for some $i,j\in \qq{\fn}$, the upper bound can be improved as
        \begin{equation}\label{eq:bcal_k_improved}
            \begin{aligned}{\cal K}^{(\fn)}_{t, \boldsymbol{\sigma}, \ba}  \prec \frac{\lambda^2}{\lambda^2+\eta_t} \left(W^d \ell_t^d \eta_t \right)^{-\fn+1}.
            \end{aligned}
        \end{equation}
    \end{enumerate}
\end{lemma}
\begin{proof}
These properties for the 2-$\cK$-loop in \eqref{cK_and_Theta} have already been established in \Cref{lemma_properties_of_propagators}. Therefore, in the remainder of the proof we assume $n\geq 3$.
By \Cref{lemma_variance_flow}, there exists some $S_{\ti}\in \ti\fS_{\ti}$ such that $S_t\pa{\ti,S_{\ti}}=S_t$ and $S_s\pa{\ti,S_{\ti}}\in s\fS_s$ for all $s\in\qa{\ti,t}$, where the flow $S_s\pa{\ti,S_{\ti}}$ is defined by \eqref{def_S_t}. Applying the evolution equation \eqref{pro_dyncalK}, we obtain
    \begin{equation}
        \begin{aligned}
            &\frac{\dd}{\dd s}\,{\cK}^{(\fn)}_{s, \boldsymbol{\sigma}, \ba}
       =
        W^d\sum_{1\le k < l \le \fn} \sum_{\qa{a}\in \Zn}  \pa{\cutL^{\qa{a}}_{k, l} \circ {\mathcal{K}}^{(\fn)}_{s, \boldsymbol{\sigma}, \ba}}  \cdot \pa{\cutR^{\qa{a}}_{k, l} \circ {\mathcal{K}}^{(\fn)}_{s, \boldsymbol{\sigma}, \ba}}\\
        =&\left({{\vartheta}}^{(\fn)}_{s, \boldsymbol{\sigma}} \circ{ \mathcal{K}_{s,\bsigma,\cdot}^{\pa{\fn}}}\right)_{\ba}+W^d\sum_{1\le k < l \le \fn:2\leq l-k\leq n-2} \sum_{\qa{a}\in \wt\Z_n^d}  \pa{\cutL^{\qa{a}}_{k, l} \circ {\mathcal{K}}^{(\fn)}_{s, \boldsymbol{\sigma}, \ba} } \cdot  \pa{\cutR^{\qa{a}}_{k, l} \circ {\mathcal{K}}^{(\fn)}_{s, \boldsymbol{\sigma}, \ba}},
        \end{aligned}
    \end{equation}
    where recall that the operator $\vartheta_{s,\bsigma}^{\pa{\fn}}$ is defined in \eqref{def:op_thn}. Applying Duhamel's principle, we get
    \begin{equation}\label{cK_Duhamel}
            {\cK}^{(\fn)}_{s, \boldsymbol{\sigma}, \ba}=\scr{U}_1\pa{s}+W^{d}\sum_{1\leq k<l\leq \fn: 2\leq l-k\leq \fn-2}\scr{U}_2^{\pa{k,l}}\pa{s},
    \end{equation}
    where we denote
    \begin{align}\label{scrU_1}
            \scr{U}_1\pa{s}&:=\left(\mathcal{U}^{(\fn)}_{\ti, s, \boldsymbol{\sigma}} \circ  \mathcal{K}^{(\fn)}_{\ti, \boldsymbol{\sigma},\cdot}\right)_{\ba}=\sum_{\mathbf b\in (\wt\Z_n^d)^\fn} \prod_{i=1}^\fn \left[I+\pa{s-\ti}m\p{\sigma_i}m\p{\sigma_{i+1}}\Theta^{\p{\sigma_{i},\sigma_{i+1}}}_s\right]_{[a_i] [b_i]} \cdot \cK_{\ti,\bsigma,\mathbf{b}}^{\pa{\fn}},\\
            \scr{U}_2^{\pa{k,l}}\pa{s}&:=\int_{\ti}^s\rd u\sum_{\mathbf b\in (\wt\Z_n^d)^\fn} \prod_{i=1}^\fn \left[I+\pa{s-u}m\p{\sigma_i}m\p{\sigma_{i+1}}\Theta^{\p{\sigma_{i},\sigma_{i+1}}}_s\right]_{[a_i] [b_i]} \nonumber\\
            &\qquad\qquad\times\sum_{\qa{a}\in \wt\Z_n^d}  \pa{\cutL^{\qa{a}}_{k, l} \circ {\mathcal{K}}^{(\fn)}_{u, \boldsymbol{\sigma}, \mathbf{b}} }  \cdot \pa{ \cutR^{\qa{a}}_{k, l} \circ {\mathcal{K}}^{(\fn)}_{u, \boldsymbol{\sigma}, \mathbf{b}}}.\label{scrU_2}
    \end{align}
    The remainder of the proof is organized into the following three steps.
    \begin{enumerate}
        \item {\bf Global bound}:
        We first show that the estimates \eqref{fast_decay_K}, \eqref{eq:bcal_k}, and \eqref{eq:bcal_k_improved} hold $t$ replaced by $\ti$.
        \item {\bf Pure loop bound}: If $\bsigma=\pa{\sigma_1,\sigma_2,\ldots,\sigma_{\fn}}$ is a pure loop, i.e., $\sigma_1=\sigma_2=\cdots=\sigma_{\fn}$, then
        \begin{equation}\label{pure_loop_bound}           \sup_{s\in\qa{\ti,t}}\max_{\ba}\abs{\cal{K}^{(\fn)}_{s,\bm{\sigma},\ba} }= \OO_{\prec}\p{W^{-d(\fn-1)}}.
        \end{equation}
        Moreover, for any general $\bsigma\in\ha{+,-}^{\fn}$, the pure loop bound implies that uniformly for $s\in\qa{\ti,t}$,
        \begin{equation}\label{sum_zero_cK}
                \sum_{\qa{a_2},\ldots,\qa{a_\fn}} \mathcal{K}^{(\fn)}_{s,\bm{\sigma},\ba}= \OO_{\prec}\p{\p{W^d\eta_s}^{-\fn+1}}.
        \end{equation}
        \item {\bf General bound}: Finally, we show that the bounds \eqref{fast_decay_K}, \eqref{eq:bcal_k}, and \eqref{eq:bcal_k_improved} hold for all $\bsig\in \{+,-\}^\fn$.
    \end{enumerate}

    \medskip
    \noindent{\bf Step }\txt{(i)}. In the global regime (with $t=\ti$), we have $\eta_{\ti}\sim 1-\ti\sim 1$ by \eqref{pick_parameter_1} and \eqref{eta_t_sim_1-t}. To establish the fast decay property, we consider the entrywise primitive loops \smash{$\wh \cK_{\bsigma,\bx}^{\pa{\fn}}\equiv\wh \cK_{\ti,\bsigma,\bx}^{\pa{\fn}}$} for $\bsigma\in\ha{+,-}^{\fn}$ and $\bx=\pa{x_1,x_2,\ldots,x_\fn}$, and prove the following statement by induction on $\fn$: for any constants \(\varepsilon,D>0\),
    \begin{equation}\label{eq;decayglobal}
          \max_{i,j\in\qa{\fn}}\absa{x_i-x_j}\geq W^{1+\epsilon}\ \implies \ \wh \cK_{\bsigma,\bx}^{\pa{\fn}}=\OO(W^{-D}), \quad \text{for}\quad \bx=\pa{x_1,x_2,\ldots,x_\fn} .
    \end{equation}
    This immediately yields the desired decay estimate for the blockwise loops, since $\cK_{\ti,\bsigma,\cdot}^{\p{\fn}}=\qq{\wh \cK_{\ti,\bsigma,\cdot}^{\p{\fn}}}$.

    Using a geometric-series expansion and the fact that $\|S_{\ti}\|_{\infty\to\infty}\le t_i<1-\e_0$ (since matrices in $\fS_{\ti}$ are doubly stochastic), we obtain the following for some constant $c>0$:
    \begin{equation}\label{global_exp_decay}
        \begin{aligned}
            \pa{1-m\p{\sigma}m\p{\sigma'}S_{\ti}}^{-1}_{xy}= \delta_{xy}+\OO\p{W^{-d} e^{-c|x-y|/W}}, \quad \forall \sigma,\sigma'\in\{+,-\}.
        \end{aligned}
    \end{equation}
    This proves \eqref{fast_decay_K} for the case $\fn=2$. Now, assume that \eqref{fast_decay_K} holds for all loop lengths $2,\ldots,\fn$.
    We now prove it for $(\fn+1)$-$\cK$-loops. By shift invariance \eqref{shift_invariance}, we may assume without loss of generality that $|x_{\fn+1}-x_{1}|/W\geq W^{\varepsilon}$.
    Using the recursive relation \eqref{wh_cK_recursive_relation} together with \eqref{global_exp_decay}, we deduce the desired decay for the $(\fn+1)$-$\cK$-loops. This completes the induction.

    Next, we prove \eqref{eq:bcal_k} by induction on $\fn$, using a slightly stronger formulation, stated as in the following claim. For simplicity, we introduce the following notation for an arbitrary diagonal matrix $\fB$ (recall \eqref{def:Ia}):
        \begin{equation}\label{eq:Baa}
                \fB_{\qa{a}}:=I_{\qa{a}}\fB I_{\qa{a}} \in \C^{N\times N}.
        \end{equation}
    \begin{claim}\label{claim_global_l_1_l_infty_norm}
    Let $\ha{\fB^{\pa{i}}}_{i=1}^{\fn}$ be arbitrary diagonal matrices. For any $\bsigma\in\ha{+,-}^{\fn}$, define the generalized $\cK$-loop
        \begin{equation}\label{def_generalized_cK}                \cK_{\bsigma}^{\pa{\fn}}\pa{\fB^{\pa{1}},\ldots,\fB^{\pa{\fn}}}:=\sum_{\bx\in\p{\ZL}^{\fn}}\wh \cK_{\ti,\bsigma,\bx}^{\pa{\fn}}\prod_{i=1}^{\fn}\fB_{x_ix_i}^{\pa{i}}.
        \end{equation}
       Then, for all $\bsigma\in\ha{+,-}^{\fn}$ and $\ba\in\p{\Zn}^\fn$, we have
        \begin{equation}\label{bound_cK_global_l_1_l_infty_norm}
            \begin{aligned}
                \absa{\cK_{\bsigma}^{\pa{\fn}}\pa{\fB^{\pa{1}}_{\qa{a_1}},\ldots,\fB^{\pa{\fn}}_{\qa{a_\fn}}}}\prec \min_{i=1}^{\fn} \norm{\fB^{\pa{i}}_{\qa{a_i}}}_{1} \prod_{j\neq i}\norm{\fB^{\pa{j}}_{\qa{a_j}}},
            \end{aligned}
        \end{equation}
        where $\norm{\cdot}$ denotes the operator norm, and $\norm{A}_{1}=\sum_{i,j}|A_{ij}|$ denotes the matrix $L^1$-norm for any matrix $A$.
    \end{claim}
    \begin{proof}
        Given any $\bsigma=\pa{\sigma_1,\sigma_2}\in\{+,-\}^2$, $\ba=\pa{\qa{a_1},\qa{a_2}}\in (\Zn)^2$, and matrices $(\fB^{\pa{1}},\fB^{\pa{2}})$, we apply the bound \eqref{global_exp_decay}. This yields a constant $c>0$ such that
            \begin{align}
                &~\cK_{\bsigma}^{\pa{2}}\pa{\fB^{\pa{1}}_{\qa{a_1}},\fB^{\pa{2}}_{\qa{a_2}}}=\sum_{x_1\in\qa{a_1}}\sum_{x_2\in\qa{a_2}}m\p{\sigma_1}m\p{\sigma_2}\pa{1-m\p{\sigma_1}m\p{\sigma_2}S_{\ti}}^{-1}_{x_1x_2}\fB^{\pa{1}}_{x_1x_1}\fB^{\pa{2}}_{x_2x_2} \label{eq:2Kloop}\\
                \prec&~ \sum_{x_1\in\qa{a_1}}\sum_{x_2\in\qa{a_2}}\pa{\delta_{x_1x_2}+W^{-d} }|\fB^{\pa{1}}_{x_1x_1}||\fB^{\pa{2}}_{x_2x_2}|\lesssim \pa{\norm{\fB^{\pa{1}}_{\qa{a_1}}}_{1}\norm{\fB^{\pa{2}}_{\qa{a_2}}}}\wedge \pa{\norm{\fB^{\pa{1}}_{\qa{a_1}}}\norm{\fB^{\pa{2}}_{\qa{a_2}}}_{1}},\nonumber
            \end{align}
        which establishes \eqref{bound_cK_global_l_1_l_infty_norm} for the case $\fn=2$

        Next, suppose that \eqref{bound_cK_global_l_1_l_infty_norm} holds for all loop lengths $2,\ldots,\fn$. We prove \eqref{bound_cK_global_l_1_l_infty_norm} for generalized $\cK$-loops of length $\fn+1$. Consider any sequence of diagonal matrices \smash{$\ha{\fB^{\pa{i}}}_{i=1}^{\fn+1}$}, any \smash{$\bsigma\in\ha{+,-}^{\fn+1}$}, and any \smash{$\ba\in\p{\Zn}^{\fn+1}$}. Without loss of generality, assume that the minimum of the right-hand side (RHS) of \eqref{bound_cK_global_l_1_l_infty_norm} is attained at $i=\fn+1$.
        Then, taking the weighted summation
    \begin{equation}
        \begin{aligned}            \sum_{\bx\in\p{\Z_L^d}^{\fn+1}}\prod_{i=1}^{\fn+1}\p{\mathfrak{B}_{\qa{a_i}}^{\pa{i}}}_{x_ix_i}
        \end{aligned}
    \end{equation}
    on both sides of the recursive relation \eqref{wh_cK_recursive_relation}, we get that
    \begin{equation}\label{general_cK_recursion}
        \begin{aligned}            &\cK_{\bsigma}^{\pa{\fn+1}}\pa{\mathfrak{B}_{\qa{a_1}}^{\pa{1}},\ldots,\mathfrak{B}_{\qa{a_{\fn+1}}}^{\pa{\fn+1}}}=m\p{\sigma_1}\cK^{(\fn)}_{\pa{\sigma_2,\ldots,\sigma_{\fn+1}}}\pa{\mathfrak{B}_{\qa{a_2}}^{\pa{2}},\ldots,\mathfrak{B}_{\qa{a_{\fn}}}^{\pa{\fn}},\wt \fB_{\fn+1}{\mathfrak{B}_{\qa{a_1}}^{\pa{1}}}}\\      &+m\p{\sigma_1}\sum_{k=2}^{\fn}\sum_{x}\cK_{\pa{\sigma_1,\ldots,\sigma_k}}^{\pa{k}}\pa{\mathfrak{B}_{\qa{a_1}}^{\pa{1}},\ldots,\mathfrak{B}_{\qa{a_{k-1}}}^{\pa{k-1}},S^{\pa{x}}_{\ti}}\cdot\cK_{\pa{\sigma_k,\ldots,\sigma_{\fn+1}}}^{\pa{\fn-k+2}}\pa{\mathfrak{B}_{\qa{a_k}}^{\pa{k}},\ldots,\mathfrak{B}_{\qa{a_\fn}}^{\pa{\fn}},\wt \fB_{\fn+1}{F_x}},
        \end{aligned}
    \end{equation}
    where the matrix $\wt \fB_{\fn+1}$ is defined as (recall \eqref{def_cS} for the definition of the operator $\cS_{\ti}\equiv \cS_{S_{\ti}}$)
    \begin{align}\label{eq:defmathB}
     \wt \fB_{\fn+1}=\pa{1-m\p{\sigma_1}m\p{\sigma_{\fn+1}}\cS_{\ti}\qa{\cdot}}^{-1}\fB^{\pa{\fn+1}}_{\qa{a_{\fn+1}}} ,
    \end{align}
    $S^{\pa{x}}$ is the diagonal matrix defined as \be\label{eq:Sxij}S^{\pa{x}}_{ij}:=\delta_{ij}S_{jx},\ee
    and $F_x$ is the diagonal matrix defined as $\pa{F_x}_{ij}:=\delta_{ij}\delta_{jx}$. Note the operator inverse in \eqref{eq:defmathB} is defined as follows for any variance profile $S$ and diagonal matrix $A$:
    \begin{equation}\label{operator_to_matrix}
        \begin{aligned}
            \pa{\pa{1-m\p{\sigma}m\p{\sigma'}\cS_S\qa{\cdot}}^{-1}A}_{xy}=\delta_{xy}\sum_{w}\pa{1-m\p{\sigma}m\p{\sigma'}S}^{-1}_{yw}A_{ww}.
        \end{aligned}
    \end{equation}
    Applying the induction hypothesis to the first term on the RHS of \eqref{general_cK_recursion} yields
    \begin{equation}\label{eq:firstinduction}
        \begin{aligned}
            \absa{\cK^{(\fn)}_{\pa{\sigma_2,\ldots,\sigma_{\fn+1}}}\pa{\mathfrak{B}_{\qa{a_2}}^{\pa{2}},\ldots,\mathfrak{B}_{\qa{a_{\fn}}}^{\pa{\fn}},\wt \fB_{\fn+1}{\mathfrak{B}_{\qa{a_1}}^{\pa{1}}}}}\prec \norma{\wt \fB_{\fn+1}{\fB^{\pa{1}}_{\qa{a_1}}}}_1\prod_{i=2}^{\fn+1}\norma{\fB^{\pa{i}}_{\qa{a_i}}}.
        \end{aligned}
    \end{equation}
    Using the definition \eqref{operator_to_matrix} and the bound \eqref{global_exp_decay}, we obtain
    \begin{equation}\label{eq:Bfx}
        \begin{aligned}
            \norma{\wt \fB_{\fn+1}{\fB^{\pa{1}}_{\qa{a_1}}}}_1 &\lesssim \sum_{x\in\qa{a_1}}\sum_{y\in\qa{a_{\fn+1}}}\absa{\pa{1-m\p{\sigma_1}m\p{\sigma_2}S_{\ti}}^{-1}_{xy}}|\fB^{\pa{1}}_{xx}||\fB^{\pa{\fn+1}}_{yy}|\\
            &\prec \sum_{x\in\qa{a_1}}\sum_{y\in\qa{a_{\fn+1}}}\pa{\delta_{xy}+W^{-d}}|\fB^{\pa{1}}_{xx}||\fB^{\pa{\fn+1}}_{yy}|\lesssim \norm{\fB^{\pa{\fn+1}}_{\qa{a_{\fn+1}}}}_{1}\norm{\fB^{\pa{1}}_{\qa{a_{1}}}}.
        \end{aligned}
    \end{equation}
    Substituting it into \eqref{eq:firstinduction} gives exactly the desired contribution on the RHS of \eqref{bound_cK_global_l_1_l_infty_norm}.
    The second term in \eqref{general_cK_recursion} can be bounded similarly. More precisely, we denote
    \begin{equation}\label{def_mathscrAx}
        \begin{aligned}            \mathfrak{A}_k:=\sum_{x}\cK_{\pa{\sigma_1,\ldots,\sigma_k}}^{\pa{k}}\pa{\mathfrak{B}_{\qa{a_1}}^{\pa{1}},\ldots,\mathfrak{B}_{\qa{a_{k-1}}}^{\pa{k-1}},S^{\pa{x}}}\cdot\wt \fB_{\fn+1}{F_x}.
        \end{aligned}
    \end{equation}
    Applying the induction hypothesis to
 \smash{$\cK_{\pa{\sigma_k,\ldots,\sigma_{\fn+1}}}^{\pa{\fn-k+2}}$}
    gives
    \begin{equation}\label{first_bound_cK_2}
        \begin{aligned}
            \cK_{\pa{\sigma_k,\ldots,\sigma_{\fn+1}}}^{\pa{\fn-k+2}}\pa{\mathfrak{B}_{\qa{a_k}}^{\pa{k}},\ldots,\mathfrak{B}_{\qa{a_\fn}}^{\pa{\fn}},\mathfrak{A}_k}\prec {\norm{\mathfrak{A}_{k}}_1\prod_{i=k}^{\fn}\norm{\fB^{\pa{i}}_{\qa{a_i}}}}.
        \end{aligned}
    \end{equation}
   Next, applying the induction hypothesis to  \smash{$\cK_{\pa{\sigma_1,\ldots,\sigma_k}}^{\pa{k}}
   $} inside $\mathfrak{A}_k$, and again using \eqref{global_exp_decay}, we get
    \begin{equation}\label{bound_mathscrAk}
        \begin{aligned}
            \norm{\mathfrak{A}_{k}}_1 &\prec \sum_{x}\absa{\cK_{\pa{\sigma_1,\ldots,\sigma_k}}^{\pa{k}}\pa{\mathfrak{B}_{\qa{a_1}}^{\pa{1}},\ldots,\mathfrak{B}_{\qa{a_{k-1}}}^{\pa{k-1}},S^{\pa{x}}}}\cdot\abs{\p{\wt \fB_{\fn+1}}_{xx}}\\
            &\prec \sum_{x}\sum_{y\in\qa{a_{\fn+1}}}\prod_{i=1}^{k-1}\norm{\fB^{\pa{i}}_{\qa{a_i}}} \cdot \absa{\pa{1-m\p{\sigma_1}m\p{\sigma_{\fn+1}}S_{\ti}}^{-1}_{xy}}\abs{\fB^{\pa{\fn+1}}_{yy}}\lesssim \norm{\fB^{\pa{\fn+1}}_{\qa{a_{\fn+1}}}}_1\prod_{i=1}^{k-1}\norm{\fB^{\pa{i}}_{\qa{a_i}}}.
        \end{aligned}
    \end{equation}
    Plugging this bound into \eqref{first_bound_cK_2} yields exactly the desired contribution on the RHS of \eqref{bound_cK_global_l_1_l_infty_norm}. This completes the induction step for loop length $\fn+1$ and thereby completes the proof of \Cref{claim_global_l_1_l_infty_norm}.
    \end{proof}

    Clearly, \Cref{claim_global_l_1_l_infty_norm} implies \eqref{eq:bcal_k} as a special case in the global regime. It remains to establish the refined estimate \eqref{eq:bcal_k_improved} at time~$\ti$. To this end, we introduce additional notations.
    Let $S_{\ti}^{\txt{d}}$ denote the block-diagonal part of $S_{\ti}$; that is, \smash{\(\p{S_{\ti}^{\txt{d}}}_{ij}:=\sum_{\qa{a}}\mathbf{1}_{i,j\in\qa{a}}\pa{S_{\ti}}_{ij}\)}. We then decompose $S_{\ti}$ as
    \begin{equation}\label{decomposition_S_d_od}
        \begin{aligned}
            S_{\ti}=:S_{\ti}^{\txt{d}}+S_{\ti}^{\txt{od}}.
        \end{aligned}
    \end{equation}
    Next, we define the diagonal part of the generalized $\cK$-loops by
    \begin{equation}\label{def_fd_cK}
        \begin{aligned}
            \fd\pa{\cK}_{\bsigma}^{\pa{\fn}}\pa{\fB^{\pa{1}}_{\qa{a_1}},\ldots,\fB^{\pa{\fn}}_{\qa{a_\fn}}},
        \end{aligned}
    \end{equation}
    constructed as follows. We first define the $\fd\pa{\cK}$-loops by replacing the variance profile $S$ and spectral parameter $z$ in \eqref{wh_cK_recursive_relation} with $S_{\ti}^{\txt{d}}$ and $\sE$. The generalized $\fd\pa{\cK}$-loops in \eqref{def_fd_cK} are then defined analogously to \eqref{def_generalized_cK}. From \eqref{wh_cK_recursive_relation}, it is immediate that \smash{$\fd\pa{\cK}_{\bsigma}^{\pa{\fn}}\p{\fB^{\pa{1}}_{\qa{a_1}},\ldots,\fB^{\pa{\fn}}_{\qa{a_\fn}}}\neq 0$} only if $\qa{a_1}=\cdots=\qa{a_{\fn}}$. Hence, \eqref{eq:bcal_k_improved} follows directly from the next claim. This completes Step (i) of the proof of \Cref{bounds_on_primitive_loops}.

    \begin{claim}\label{claim_difference_cK_cKd}
    Let $\ha{\fB^{\pa{i}}}_{i=1}^{\fn}$ be arbitrary diagonal matrices. For any $\bsigma\in\ha{+,-}^{\fn}$ and $\ba\in\p{\Zn}^\fn$, we have
        \begin{equation}\label{bound_difference_cK_cKd}
            \begin{aligned}
                \cK_{\bsigma}^{\pa{\fn}}\pa{\fB^{\pa{1}}_{\qa{a_1}},\ldots,\fB^{\pa{\fn}}_{\qa{a_\fn}}}-\fd\pa{\cK}_{\bsigma}^{\pa{\fn}}\pa{\fB^{\pa{1}}_{\qa{a_1}},\ldots,\fB^{\pa{\fn}}_{\qa{a_\fn}}} \prec {\lambda^2W^d\prod_{i=1}^{\fn}\norm{\fB^{\pa{i}}_{\qa{a_i}}}}.
            \end{aligned}
        \end{equation}
    \end{claim}
    \begin{proof}
 We again employ an induction argument on $\fn$. For $\fn=1$, by \eqref{eq:loop1}, we have that
        \begin{equation}
            \begin{aligned}
\cK_{\sigma}^{\pa{1}}\pa{\fB^{\pa{1}}_{\q{a_1}}}=\fd\pa{\cK}_{\sigma}^{\pa{1}}\pa{\fB^{\pa{1}}_{\q{a_1}}}=m\p{\sigma}\avga{\fB^{\pa{1}}_{\q{a_1}}}.
            \end{aligned}
        \end{equation}
Now assume that \eqref{bound_difference_cK_cKd} holds for all loop lengths $1,\ldots,\fn$. For the case $\fn+1$, we obtain from \eqref{general_cK_recursion} the recursive relation:
\begin{align}            &~\fd\p{\cK}_{\bsigma}^{\pa{\fn+1}}\pa{\mathfrak{B}_{\qa{a_1}}^{\pa{1}},\ldots,\mathfrak{B}_{\qa{a_{\fn+1}}}^{\pa{\fn+1}}}=m\p{\sigma_1}\fd\p{\cK}^{(\fn)}_{\pa{\sigma_2,\ldots,\sigma_{\fn+1}}}\pa{\mathfrak{B}_{\qa{a_2}}^{\pa{2}},\ldots,\mathfrak{B}_{\qa{a_{\fn}}}^{\pa{\fn}},\fd\p{\wt \fB_{\fn+1}}{\mathfrak{B}_{\qa{a_1}}^{\pa{1}}}} \label{comparison_cK_cKd_cKd}\\
+&~m\p{\sigma_1}\sum_{k=2}^{\fn}\sum_{x}\fd\p{\cK}_{\pa{\sigma_1,\ldots,\sigma_k}}^{\pa{k}}\pa{\mathfrak{B}_{\qa{a_1}}^{\pa{1}},\ldots,\mathfrak{B}_{\qa{a_{k-1}}}^{\pa{k-1}},\p{S_{\ti}^{\txt{d}}}^{\pa{x}}}\cdot\fd\p{\cK}_{\pa{\sigma_k,\ldots,\sigma_{\fn+1}}}^{\pa{\fn-k+2}}\pa{\mathfrak{B}_{\qa{a_k}}^{\pa{k}},\ldots,\mathfrak{B}_{\qa{a_\fn}}^{\pa{\fn}},\fd\p{\wt \fB_{\fn+1}}{F_x}}.\nonumber
 \end{align}
Here, the matrix $\fd\p{\wt \fB_{\fn+1}}$ is defined by
\[\fd\p{\wt \fB_{\fn+1}}:=\pa{1-m\p{\sigma_1}m\p{\sigma_{\fn+1}}\cS_{\ti}^{\txt{d}}[\cdot]}^{-1}\fB^{\pa{\fn+1}}_{\qa{a_{\fn+1}}},\]
where $\cS_{\ti}^{\txt{d}}$ denotes the operator in \eqref{def_cS} with the variance profile $S$ replaced by $S_{\ti}^{\txt{d}}$. Using the geometric-series representation and the bound \smash{$\|S_{\ti}^{\txt{d}}\|_{\infty\to\infty}\le \|S_{\ti}\|_{\infty\to\infty}\le t_i<1-\e_0$}, we obtain exactly the same exponential-decay estimate as in \eqref{global_exp_decay}:
    \begin{equation}\label{global_exp_decay2}
        \begin{aligned}
            \pa{1-m\p{\sigma}m\p{\sigma'}S^{\txt{d}}_{\ti}}^{-1}_{xy}= \delta_{xy}+\OO\p{W^{-d} e^{-c|x-y|/W}}, \quad \forall \sigma,\sigma'\in\{+,-\}.
        \end{aligned}
    \end{equation}

Subtracting the first terms on the RHS of \eqref{general_cK_recursion} and \eqref{comparison_cK_cKd_cKd}, and applying the induction hypothesis, we obtain that
        \begin{equation*}
            \begin{aligned}
                &~\cK^{(\fn)}_{\pa{\sigma_2,\ldots,\sigma_{\fn+1}}}\pa{\mathfrak{B}_{\qa{a_2}}^{\pa{2}},\ldots,\mathfrak{B}_{\qa{a_{\fn}}}^{\pa{\fn}},\wt \fB_{\fn+1}{\mathfrak{B}_{\qa{a_1}}^{\pa{1}}}}-\fd\pa{\cK}^{(\fn)}_{\pa{\sigma_2,\ldots,\sigma_{\fn+1}}}\pa{\mathfrak{B}_{\qa{a_2}}^{\pa{2}},\ldots,\mathfrak{B}_{\qa{a_{\fn}}}^{\pa{\fn}},\fd\p{\wt \fB_{\fn+1}}{\mathfrak{B}_{\qa{a_1}}^{\pa{1}}}}\\
                =&~\cK_{\pa{\sigma_2,\ldots,\sigma_{\fn+1}}}\pa{\mathfrak{B}_{\qa{a_2}}^{\pa{2}},\ldots,\mathfrak{B}_{\qa{a_{\fn}}}^{\pa{\fn}},\pa{\wt \fB_{\fn+1}-\fd\p{\wt \fB_{\fn+1}}}{\mathfrak{B}_{\qa{a_1}}^{\pa{1}}}}+\opr{\lambda^2W^d\norma{\fd\p{\wt \fB_{\fn+1}}{\mathfrak{B}_{\qa{a_1}}^{\pa{1}}}}\prod_{i=2}^{\fn}\norm{\mathfrak{B}_{\qa{a_{i}}}^{\pa{i}}}}\\
                \prec &~\norma{\pa{\wt \fB_{\fn+1}-\fd\p{\wt \fB_{\fn+1}}} {\fB^{\pa{1}}_{\qa{a_{1}}}}}_1\prod_{i=2}^{\fn}\norma{\mathfrak{B}_{\qa{a_{i}}}^{\pa{i}}}+\lambda^2W^d\prod_{i=1}^{\fn+1}\norma{\mathfrak{B}_{\qa{a_{i}}}^{\pa{i}}}
                \prec \lambda^2W^d\prod_{i=1}^{\fn+1}\norma{\mathfrak{B}_{\qa{a_{i}}}^{\pa{i}}},
            \end{aligned}
        \end{equation*}
which yields the desired contribution on the RHS of \eqref{bound_difference_cK_cKd}.
Here, in the second step, we use that
\[\norma{\fd\p{\wt \fB_{\fn+1}} {\mathfrak{B}_{\qa{a_1}}^{\pa{1}}}}  \lesssim \norma{\fB^{\pa{1}}_{\qa{a_{1}}}}\norma{\fB^{\pa{\fn+1}}_{\qa{a_{\fn+1}}}},\]
which follows directly from the definition \eqref{operator_to_matrix} together with the bound $\|\pa{1-m\p{\sigma_1}m\p{\sigma_{\fn+1}}S_{\ti}^{\rd}}^{-1}\|_{\infty\to\infty}\lesssim 1$, a consequence of \eqref{global_exp_decay2}.
In the third step, we apply \eqref{bound_cK_global_l_1_l_infty_norm} in combination with the estimate:
\begin{align}
        &~\norma{\p{\wt \fB_{\fn+1}-\fd\p{\wt \fB_{\fn+1}}}{\fB^{\pa{1}}_{\qa{a_{1}}}}}_1\nonumber\\
                \leq&~ \sum_{x\in\qa{a_1}}\absa{\sum_{y\in\qa{a_{\fn+1}}}\qa{\pa{1-m\p{\sigma_1}m\p{\sigma_{\fn+1}}S_{\ti}}^{-1}-\pa{1-m\p{\sigma_1}m\p{\sigma_{\fn+1}}S_{\ti}^{\txt{d}}}^{-1}}_{xy}\fB^{\pa{\fn+1}}_{yy}\fB^{\pa{1}}_{xx}}\nonumber\\
                \lesssim&~ \sum_{x\in\qa{a_1}}\sum_{y\in\qa{a_{\fn+1}}}\absa{\qa{\pa{1-m\p{\sigma_1}m\p{\sigma_{\fn+1}}S_{\ti}}^{-1}S_{\ti}^{\txt{od}}\pa{1-m\p{\sigma_1}m\p{\sigma_{\fn+1}}S_{\ti}^{\txt{d}}}^{-1}}_{xy}}\norma{\fB^{\pa{1}}_{\qa{a_{1}}}}\norma{\fB^{\pa{\fn+1}}_{\qa{a_{\fn+1}}}}\nonumber\\
                \lesssim&~ \norma{\fB^{\pa{1}}_{\qa{a_{1}}}} \norma{\fB^{\pa{\fn+1}}_{\qa{a_{\fn+1}}}} \sum_{x\in\qa{a_1}}\sum_{y\in\qa{a_{\fn+1}}}\sum_{w_1}\sum_{w_2\in\qa{a_{\fn+1}}}\pa{\delta_{xw_1}+\frac{1}{W^d} e^{-c|x-w_1|/W}}\p{S_{\ti}^{\txt{od}}}_{w_1w_2}\pa{\delta_{w_2y}+\frac{1}{W^d}e^{-c|y-w_2|/W}}\nonumber\\
                \lesssim&~ \lambda^2W^d \norm{\fB^{\pa{1}}_{\qa{a_{1}}}} \norm{\fB^{\pa{\fn+1}}_{\qa{a_{\fn+1}}}},\label{example_difference_1}
            \end{align}
where \eqref{global_exp_decay} and \eqref{global_exp_decay2} are used in the third step, and the final step relies on the fact that the block-diagonal part of $S_{\ti}^{\txt{od}}$ vanishes.

For the second terms on the RHS of \eqref{general_cK_recursion} and \eqref{comparison_cK_cKd_cKd}, we bound their difference by replacing $\cK$ and $S_{\ti}$ in \eqref{general_cK_recursion} with their block-diagonal counterparts $\fd\pa{\cK}$ and $S_{\ti}^{\rd}$ one by one. For example, one contribution to the difference can be bounded by
        \begin{equation}\label{eq:second_diff1}
            \begin{aligned}
                &~{\sum_{x}\cK_{\pa{\sigma_1,\ldots,\sigma_k}}^{\pa{k}}\pa{\mathfrak{B}_{\qa{a_1}}^{\pa{1}},\ldots,\mathfrak{B}_{\qa{a_{k-1}}}^{\pa{k-1}},\p{S_{\ti}^{\txt{od}}}^{\pa{x}}}\cdot\cK_{\pa{\sigma_k,\ldots,\sigma_{\fn+1}}}^{\pa{\fn-k+2}}\pa{\mathfrak{B}_{\qa{a_k}}^{\pa{k}},\ldots,\mathfrak{B}_{\qa{a_\fn}}^{\pa{\fn}},\wt \fB_{\fn+1} {F_x}}}\\
                =&~ \cK_{\pa{\sigma_k,\ldots,\sigma_{\fn+1}}}^{\pa{\fn-k+2}}\pa{\mathfrak{B}_{\qa{a_k}}^{\pa{k}},\ldots,\mathfrak{B}_{\qa{a_\fn}}^{\pa{\fn}},\mathfrak{A}_{k}^{\txt{od}}} \prec {\norma{\mathfrak{A}_{k}^{\txt{od}}}_1\prod_{i=k}^{\fn}\norma{\fB^{\pa{i}}_{\qa{a_i}}}}\\
                \prec&~ \prod_{i=1}^{\fn}\norma{\fB^{\pa{i}}_{\qa{a_i}}}\sum_{x}\sum_{w}\p{S_{\ti}^{\txt{od}}}_{xw} \sum_{y\in\qa{a_{\fn+1}}}\absa{\pa{1-m\p{\sigma_1}m\p{\sigma_{\fn+1}}S_{\ti}}^{-1}_{xy}}\absa{\fB^{\pa{\fn+1}}_{yy}}\\
                \prec&~ \prod_{i=1}^{\fn+1}\norma{\fB^{\pa{i}}_{\qa{a_i}}}\sum_{x}\sum_{w}\p{S_{\ti}^{\txt{od}}}_{xw}\sum_{y\in\qa{a_{\fn+1}}}\pa{\delta_{xy}+\frac{1}{W^d} e^{-c|x-y|/W}}\prec \lambda^2W^d\prod_{i=1}^{\fn+1}\norm{\fB^{\pa{i}}_{\qa{a_i}}},
            \end{aligned}
        \end{equation}
        where $\mathfrak{A}_{k}^{\txt{od}}$ is defined analogously to \eqref{def_mathscrAx} with $S$ replaced by $S_{\ti}^{\txt{od}}$.
The second and third steps follow from \eqref{bound_cK_global_l_1_l_infty_norm} together with the operator representation \eqref{operator_to_matrix}; the fourth step uses \eqref{global_exp_decay}; and the last step uses \eqref{def_lambda} to get $\sum_{x\in[a]}\sum_{w}\p{S_{\ti}^{\txt{od}}}_{xw} \le\lambda^2W^d$ for any $[a]\in \Zn$.
      With \eqref{eq:second_diff1}, we obtain
        \begin{equation*}
            \begin{aligned}
                &~m\pa{\sigma_1}\sum_{k=2}^{\fn}\sum_{x}\cK_{\pa{\sigma_1,\ldots,\sigma_k}}^{\pa{k}}\pa{\mathfrak{B}_{\qa{a_1}}^{\pa{1}},\ldots,\mathfrak{B}_{\qa{a_{k-1}}}^{\pa{k-1}},S^{\pa{x}}_{\ti}}\cdot\cK_{\pa{\sigma_k,\ldots,\sigma_{\fn+1}}}^{\pa{\fn-k+2}}\pa{\mathfrak{B}_{\qa{a_k}}^{\pa{k}},\ldots,\mathfrak{B}_{\qa{a_\fn}}^{\pa{\fn}},\wt \fB_{\fn+1}{F_x}}\\
                -&~m\pa{\sigma_1}\sum_{k=2}^{\fn}\sum_{x}{\cK}_{\pa{\sigma_1,\ldots,\sigma_k}}^{\pa{k}}\pa{\mathfrak{B}_{\qa{a_1}}^{\pa{1}},\ldots,\mathfrak{B}_{\qa{a_{k-1}}}^{\pa{k-1}},\p{S_{\ti}^{\txt{d}}}^{\pa{x}}}\cdot{\cK}_{\pa{\sigma_k,\ldots,\sigma_{\fn+1}}}^{\pa{\fn-k+2}}\pa{\mathfrak{B}_{\qa{a_k}}^{\pa{k}},\ldots,\mathfrak{B}_{\qa{a_\fn}}^{\pa{\fn}},{\wt \fB_{\fn+1}}{F_x}}\\
                \prec &~{\lambda^2W^d\prod_{i=1}^{\fn+1}\norm{\fB^{\pa{i}}_{\qa{a_i}}}}.
            \end{aligned}
        \end{equation*}
        Applying the same reasoning—relying on the induction hypothesis, the bound \eqref{bound_cK_global_l_1_l_infty_norm} (and its analogue for the $\fd(\cK)$-loops), together with \eqref{global_exp_decay} and \eqref{global_exp_decay2}—a finite number of times, we conclude that
        \begin{equation}
            \begin{aligned}
                \cK_{\bsigma}^{\pa{\fn+1}}\pa{\mathfrak{B}_{\qa{a_1}}^{\pa{1}},\ldots,\mathfrak{B}_{\qa{a_{\fn+1}}}^{\pa{\fn+1}}}-\fd\pa{\cK}_{\bsigma}^{\pa{\fn+1}}\pa{\mathfrak{B}_{\qa{a_1}}^{\pa{1}},\ldots,\mathfrak{B}_{\qa{a_{\fn+1}}}^{\pa{\fn+1}}}\prec {\lambda^2W^d\prod_{i=1}^{\fn+1}\norm{\fB^{\pa{i}}_{\qa{a_i}}}}.
            \end{aligned}
        \end{equation}
        We omit further details. This completes the induction step and thus concludes the proof of \Cref{claim_difference_cK_cKd}.
    \end{proof}

    \noindent{\bf Step }\txt{(ii)}. Using \eqref{prop:ThfadC_short_pointwise}, together with an argument analogous to the one used in proving the decay property around \eqref{eq;decayglobal}, we obtain the following fast decay estimate. For any constants $\varepsilon>0$ and $D>0$, if  $\bsigma=\pa{\sigma_1,\cdots,\sigma_\fn}$ with $\sigma_1=\sigma_2=\cdots=\sigma_\fn$, then
    \begin{equation}\label{fast_decay_pure_loop}
            \max_{i,j\in \qq{\fn}}|[a_i]-[a_j]|\ge W^{\e}\ \implies \ \sup_{s\in\qa{\ti,t}}\abs{\cK_{s,\bsigma,\ba}^{\pa{\fn}}}=\OO(W^{-D}) \quad \text{for}\quad \ba=([a_1],[a_2],\ldots, [a_\fn]) .
     \end{equation}
     Next, combining \eqref{prop:ThfadC_short_pointwise} with an argument similar to the proof of \Cref{claim_global_l_1_l_infty_norm}, we obtain the bound \eqref{pure_loop_bound} for pure loops. To prove \eqref{sum_zero_cK}, we iteratively apply Ward’s identity \eqref{WI_calK}. At each application of \eqref{WI_calK}, the loop length decreases by one and we gain an additional factor  $(W^d\eta_t)^{-1}$. Repeating this reduction until all resulting loops are pure loops, and then invoking \eqref{pure_loop_bound}, along with the fast decay property \eqref{fast_decay_pure_loop}, yields the desired estimate. This completes the proof.

    \medskip
    \noindent{\bf Step }\txt{(iii)}. We again employ an induction argument on $\fn\geq 2$. The case $\fn=2$ follows directly from the estimates \eqref{prop:ThfadC} and \eqref{prop:ThfadC_improved}. Assume now that $\fn\ge 3$ and that the statements \eqref{fast_decay_K}, \eqref{eq:bcal_k}, and \eqref{eq:bcal_k_improved} hold uniformly for all $s\in\qa{\ti,t}$ and all loop lengths $2,\ldots,\fn-1$. The fast–decay property \eqref{fast_decay_K} then follows immediately from \eqref{prop:ThfadC}, \eqref{cK_Duhamel}, and the fact that $\ell_s$ is non-decreasing in $s$. It therefore remains to establish \eqref{eq:bcal_k} and \eqref{eq:bcal_k_improved}.
    We first consider the case where $\bsig$ is non-alternating, i.e., $\sigma_i=\sigma_{i+1}$ for some $i\in\qq{\fn}$. This situation is comparatively simpler, so we provide only a brief explanation and omit further details. By \eqref{prop:ThfadC_short}, there exists a constant $c>0$ such that
    \begin{equation}\nonumber
            \absa{\left[I+\pa{s-u}\Theta^{\pa{\sigma_{i},\sigma_{i+1}}}_s\right]_{[a_i] [b_i]}}\prec \delta_{\qa{a_i}\qa{b_i}}+{\lambda^2} c^{-1} e^{-c\absa{\qa{a_i}-\qa{b_i}}},\quad \txt{uniformly for }\qa{u,s}\subseteq\qa{\ti,t}.
    \end{equation}
    Inserting this bound—together with the induction hypothesis and the estimates obtained in Step (i)—into \eqref{cK_Duhamel} with $s=t$ yields \eqref{eq:bcal_k}. The improved estimate \eqref{eq:bcal_k_improved} follows in the same manner. The only additional observation needed is that one always gains an extra factor $\lambda^2/(\lambda^2+1-t)$ unless $\qa{a_j}=\qa{b_j}$ for all $j$ and $\qa{b_1}=\cdots=\qa{b_\fn}$, which forces $\qa{a_1}=\cdots=\qa{a_\fn}$.

    For the rest of the proof, we assume that $\bsig$ is alternating, i.e., $\sigma_i\neq \sigma_{i+1}$ for all $i\in\qq{\fn}$. For clarity, we first establish \eqref{eq:bcal_k} to illustrate the method, and then explain how the argument can be extended to obtain the additional factor $\lambda^2/(\lambda^2+1-t)$ in the improved bound \eqref{eq:bcal_k_improved}. Consider first the term $\scr{U}_1(t)$ in \eqref{scrU_1} (with $s$ replaced by $t$). We expand the product
            \smash{\(\prod_{i=1}^\fn [I+\pa{t-\ti}\Theta^{\pa{\sigma_{i},\sigma_{i+1}}}_t]_{[a_i] [b_i]}\)} (where we used the fact $|m(\sE)|^2=1$)
    into a sum of $2^\fn$ terms. If a term contains at least one factor $I_{\qa{a_i}\qa{b_i}}=\delta_{\qa{a_i}\qa{b_i}}$, then by combining the fast–decay estimate \eqref{fast_decay_K}, the bound \eqref{eq:bcal_k} for the $\cK$-loops at time $\ti$ obtained in Step (i), and the estimate \eqref{prop:ThfadC} for $\Theta_t$, we obtain the bound
    \[\opr{\p{W^{d}}^{-\pa{\fn-1}}\cdot \pa{1+(\ell_t^d\eta_t)^{-1}}^{\fn-1}}=\opr{\p{W^{d}\ell_t^d\eta_t}^{-\pa{\fn-1}}}.\]
    Consequently, the term $\scr{U}_1(t)$ can be written as
    \begin{equation}\label{scrU_1_first_error}
            \scr{U}_1\p{t}=\sum_{\mathbf b\in (\wt\Z_n^d)^\fn} \prod_{i=1}^\fn \qa{\pa{t-\ti}\Theta^{\pa{\sigma_{i},\sigma_{i+1}}}_t\pa{[a_i], [b_i]}} \cdot \cK_{\ti,\bsigma,\mathbf{b}}^{\pa{\fn}}+\opr{\pa{W^d\ell_t^d\eta_t}^{-\pa{\fn-1}}}.
    \end{equation}
    A similar argument applies to the term $\scr{U}_2^{(k,l)}(t)$ in \eqref{scrU_2}. Expanding the product
    \smash{\(\prod_{i=1}^\fn [I+\pa{t-u}\Theta^{\p{\sigma_{i},\sigma_{i+1}}}_t]_{[a_i] [b_i]}\)}
   again yields $2^\fn$ terms. If a term contains $1\le k\le \fn$ factors of the form $\delta_{\qa{a_i}\qa{b_i}}$, then applying the induction hypothesis (namely \eqref{fast_decay_K}--\eqref{eq:bcal_k_improved}) to the shortened $\cK$-loops \smash{$\cutL^{\qa{a}}_{k, l} \circ {\mathcal{K}}^{(\fn)}_{u, \boldsymbol{\sigma}, \mathbf{b}} $} and \smash{$\cutR^{\qa{a}}_{k, l} \circ {\mathcal{K}}^{(\fn)}_{u, \boldsymbol{\sigma}, \mathbf{b}} $}, whose lengths lie in $\{2,\ldots,\fn-1\}$, we obtain the bound
    \[\opr{\int_{\ti}^t \pa{\frac{\eta_u}{\ell_t^d\eta_t}}^{\fn-k}\frac{(\ell_u^d)^{\fn-k}}{(W^d\ell_u^d\eta_u)^\fn}\rd u}=\opr{ \frac{1}{W^d(W^d\ell_t^d\eta_t)^{\fn-1}} }.\]
    Therefore, the term $\scr{U}_2^{\pa{k,l}}\pa{t}$ can be written as
        \begin{align}
            \scr{U}_2^{\pa{k,l}}\pa{t}=&\int_{\ti}^t\rd u\sum_{\mathbf b\in (\wt\Z_n^d)^\fn} \prod_{i=1}^\fn \qa{\pa{t-u}\Theta^{\pa{\sigma_{i},\sigma_{i+1}}}_t\pa{[a_i], [b_i]}} \sum_{\qa{a}\in \wt\Z_n^d}  \pa{\cutL^{\qa{a}}_{k, l} \circ {\mathcal{K}}^{(\fn)}_{u, \boldsymbol{\sigma}, \mathbf{b}} } \cdot \pa{ \cutR^{\qa{a}}_{k, l} \circ {\mathcal{K}}^{(\fn)}_{u, \boldsymbol{\sigma}, \mathbf{b}} } \nonumber\\
            &+\opr{W^{-d}\pa{W^d\ell_t^d\eta_t}^{-\pa{\fn-1}}}.\label{scrU_2_first_error}
        \end{align}

    To estimate \eqref{scrU_1_first_error} and \eqref{scrU_2_first_error}, we use the decomposition \begin{equation}\label{Theta_expansion}
    \Theta^{\pa{\sigma_{i},\sigma_{i+1}}}_t\pa{[a_i], [b_i]}=\sum_{\xi=0}^{2}f_\xi\pa{\qa{a_i},\qa{s_i}},
    \end{equation}
    where $\qa{s_i}:=\qa{b_i}-\qa{b_1}$ and
    \begin{equation*}
        \begin{aligned}           &f_0\pa{\qa{a_i},\qa{s_i}}:=\Theta^{\pa{\sigma_{i},\sigma_{i+1}}}_t\pa{[a_i], [b_1]},\quad f_1\pa{\qa{a_i},\qa{s_i}}:=\frac{1}{2}\pa{\Theta^{\pa{\sigma_{i},\sigma_{i+1}}}_s\pa{[a_i], [b_1]+\qa{s_i}}-\Theta^{\pa{\sigma_{i},\sigma_{i+1}}}_s\pa{[a_i], \qa{b_1}-[s_i]}},\\
        &f_0\pa{\qa{a_i},\qa{s_i}}:=\frac{1}{2}\pa{\Theta^{\pa{\sigma_{i},\sigma_{i+1}}}_s\pa{[a_i], [b_1]+\qa{s_i}}+\Theta^{\pa{\sigma_{i},\sigma_{i+1}}}_s\pa{[a_i], [b_1]-\qa{s_i}}-2\Theta^{\pa{\sigma_{i},\sigma_{i+1}}}_s\pa{[a_i], [b_1]}}.
        \end{aligned}
    \end{equation*}
    Using this decomposition, we expand
    \be \label{Theta_expansion222} \prod_{i=2}^\fn \Theta^{\pa{\sigma_{i},\sigma_{i+1}}}_{t,[a_i][b_i]}=\sum_{\bxi=(\xi_2,\ldots, \xi_n)\in\{0,1,2\}^{\fn-1}} g_{\bm{\xi}}(\mathbf b),\quad \text{where}\quad g_{\bm{\xi}}(\mathbf b):=\prod_{i=2}^n f_{\xi_i}([a_i],[s_i]).\ee
   If $\|\bxi\|_1:=\sum_{i=1}^\fn \xi_i=1$, then the corresponding contribution vanishes due to the parity symmetry \eqref{partity_symmetry_cK}. If $\|\bxi\|_1\ge 2$, then using the improved bounds \eqref{prop:BD1} and \eqref{prop:BD2} for the $f_1$- and $f_2$-terms, respectively, we can estimate their contribution to $ \scr{U}_1\pa{t}$ as
    \[\pa{t-\ti}^\fn  \sum_{\mathbf b\in (\wt\Z_n^d)^\fn} \Theta^{\pa{\sigma_{1},\sigma_{2}}}_t\pa{[a_1], [b_1]} g_{\bm{\xi}}(\mathbf b)\cdot \cK_{\ti,\bsigma,\mathbf{b}}^{\pa{\fn}}\prec {\pa{W^d\ell_t^d\eta_t}^{-\pa{\fn-1}}}. \]
   The argument follows those in \cite[Lemma 3.11]{Band1D} and \cite[Lemma 3.15]{truong2025localizationlengthfinitevolumerandom}, so we omit the details. The only difference from the proofs there is that terms with $\|\bxi\|_1\ge 1$ may introduce a singular factor $\p{\lambda^2+1-t}^{-1}$.  However, whenever $f_{\xi_i}\pa{\qa{a_i},\qa{s_i}}\neq 0$, we necessarily have $\qa{s_i}\neq 0$, and thus the additional singular factor is canceled by the extra $\lambda^2$ gained from the off-diagonal part of \smash{$\cK_{\ti,\bsigma,\mathbf{b}}^{\pa{\fn}}$} via \eqref{eq:bcal_k_improved}.

   It remains to control the leading contribution corresponding to $\bxi=\mathbf 0$:
    \begin{equation}\label{scrU_1_second_error}
        \begin{aligned}
            \scr{U}_1\pa{t}&=\pa{t-\ti}^\fn\sum_{\mathbf b\in (\wt\Z_n^d)^\fn} \prod_{i=1}^\fn \Theta^{\pa{\sigma_{i},\sigma_{i+1}}}_t\pa{[a_i], [b_1]} \cdot \cK_{\ti,\bsigma,\mathbf{b}}^{\pa{\fn}}+\opr{\pa{W^d\ell_t^d\eta_t}^{-\pa{\fn-1}}}\\
            &=\sum_{\qa{b}\in \wt\Z_n^d} \prod_{i=1}^\fn\Theta^{\pa{\sigma_{i},\sigma_{i+1}}}_t\pa{[a_i], [b]} \cdot \frac{\pa{t-\ti}^{\fn}}{n^d}\sum_{\mathbf{b}\in(\Zn)^{\fn}} \cK_{\ti,\bsigma,\mathbf{b}}^{\pa{\fn}}+\opr{\pa{W^d\ell_t^d\eta_t}^{-\pa{\fn-1}}}\\
            &=:\sum_{\qa{b}\in \wt\Z_n^d} \prod_{i=1}^\fn\Theta^{\pa{\sigma_{i},\sigma_{i+1}}}_t\pa{[a_i], [b]}\cdot \cal{R}_1\pa{t}+\opr{\pa{W^d\ell_t^d\eta_t}^{-\pa{\fn-1}}},
        \end{aligned}
    \end{equation}
    where the second step follows from the translation invariance \eqref{eq:traninv} of the $\cK$-loops.
    Similarly, applying \eqref{Theta_expansion222} to the terms \smash{$\scr{U}_2^{\pa{k,l}}\pa{t}$} and performing an analogous analysis, we obtain
        \begin{align}
            \scr{U}_2^{\pa{k,l}}\pa{t}&=\int_{\ti}^t\rd u\sum_{\mathbf b\in (\wt\Z_n^d)^\fn} \prod_{i=1}^\fn \Theta^{\pa{\sigma_{i},\sigma_{i+1}}}_t\pa{[a_i], [b_1]} \cdot \pa{t-u}^\fn \sum_{\qa{a}\in \wt\Z_n^d}  \pa{\cutL^{\qa{a}}_{k, l} \circ {\mathcal{K}}^{(\fn)}_{u, \boldsymbol{\sigma}, \mathbf{b}} } \cdot \pa{ \cutR^{\qa{a}}_{k, l} \circ {\mathcal{K}}^{(\fn)}_{u, \boldsymbol{\sigma}, \mathbf{b}} } \nonumber\\
            &+\opr{W^{-d}\pa{W^d\ell_t^d\eta_t}^{-\pa{\fn-1}}} \nonumber\\
            &=\sum_{\qa{b}\in \wt\Z_n^d} \prod_{i=1}^\fn\Theta^{\pa{\sigma_{i},\sigma_{i+1}}}_t\pa{[a_i], [b]}\cdot \frac{1}{n^{d}} \int_{\ti}^t\rd u\, \pa{t-u}^{\fn} \sum_{\qa{a}\in \wt\Z_n^d}\sum_{\mathbf b\in (\wt\Z_n^d)^\fn}  \pa{\cutL^{\qa{a}}_{k, l} \circ {\mathcal{K}}^{(\fn)}_{u, \boldsymbol{\sigma}, \mathbf{b}} } \cdot  \pa{\cutR^{\qa{a}}_{k, l} \circ {\mathcal{K}}^{(\fn)}_{u, \boldsymbol{\sigma}, \mathbf{b}}}\nonumber\\
            &+\opr{W^{-d}\pa{W^d\ell_t^d\eta_t}^{-\pa{\fn-1}}} \nonumber\\
            & =: \sum_{\qa{b}\in \wt\Z_n^d} \prod_{i=1}^\fn\Theta^{\pa{\sigma_{i},\sigma_{i+1}}}_t\pa{[a_i], [b]}\cdot \cR_{2}^{\pa{k,l}}\pa{t}+\opr{W^{-d}\pa{W^d\ell_t^d\eta_t}^{-\pa{\fn-1}}}.\label{scrU_2_second_error}
        \end{align}
    Substituting the estimates \eqref{scrU_1_second_error} and \eqref{scrU_2_second_error} into \eqref{cK_Duhamel} yields
    \begin{equation}\label{first_reduction_cK_t}
            \cK_{t,\bsigma,\ba}^{\pa{\fn}}= \sum_{\qa{b}\in \wt\Z_n^d} \prod_{i=1}^\fn\Theta^{\pa{\sigma_{i},\sigma_{i+1}}}_{t, [a_i] [b]}\bigg({\cR_1\pa{t}+W^d\sum_{\substack{1\leq k<l\leq \fn:\\  2\leq l-k\leq \fn-2}}\cR_2^{\pa{k,l}}\pa{t}}\bigg)+\opr{\pa{W^d\ell_t^d\eta_t}^{-\pa{\fn-1}}}.
    \end{equation}

We now explore further cancellations in \eqref{first_reduction_cK_t}. On the one hand, by including the terms $\cR_2^{\pa{k,l}}\p{t}$ with $l-k\in\{1,\fn-1\}$, we obtain
    \begin{equation}\label{eq:R1t1}
            W^d\sum_{1\leq k<l\leq \fn}\cR_2^{\pa{k,l}}\pa{t}=W^d\sum_{\substack{1\leq k<l\leq \fn:\\  2\leq l-k\leq \fn-2}}\cR_2^{\pa{k,l}}\pa{t}+\frac{\fn}{n^d}\int_{\ti}^t\rd u\, \frac{\pa{t-u}^{\fn}}{1-u} \sum_{\mathbf b\in (\wt\Z_n^d)^\fn}   {\mathcal{K}}^{(\fn)}_{u, \boldsymbol{\sigma}, \mathbf{b}} \, .
    \end{equation}
    Here we used that, for any $k\in\qq{\fn}$ (with the convention $\cR_2^{\pa{\fn,\fn+1}}\p{t}=\cR_2^{\pa{\fn,1}}\p{t}$),
    \begin{equation*}
        \begin{aligned}
            \cR_2^{\pa{k,k+1}}\p{t}=&\frac{1}{n^d}\int_{\ti}^t\rd u\, \pa{t-u}^{\fn} \sum_{\qa{a}\in \wt\Z_n^d}\sum_{\mathbf b\in (\wt\Z_n^d)^\fn}  \pa{\cutL^{\qa{a}}_{k, k+1} \circ {\mathcal{K}}^{(\fn)}_{u, \boldsymbol{\sigma}, \mathbf{b}} } \cdot \pa{ \cutR^{\qa{a}}_{k, k+1} \circ {\mathcal{K}}^{(\fn)}_{u, \boldsymbol{\sigma}, \mathbf{b}} }\\
            =&\frac{1}{L^d}\int_{\ti}^t\rd u\, \frac{\pa{t-u}^{\fn}}{1-u} \sum_{\mathbf b\in (\wt\Z_n^d)^\fn}   {\mathcal{K}}^{(\fn)}_{u, \boldsymbol{\sigma}, \mathbf{b}},
        \end{aligned}
    \end{equation*}
    where the second identity follows from the shift invariance \eqref{shift_invariance}, together with the following consequence of the definitions \eqref{cK_and_Theta} and \eqref{eq:defwhTheta} and the fact that $t^{-1}S_t\in \fS_t$ is doubly stochastic:
\[W^d\sum_{[a]}\cK_{u,\pa{\sigma,\sigma'},\p{\qa{a},\qa{b}}}^{\pa{2}}=\frac{1}{1-u},\quad \forall \, \qa{b}\in\Zn, \ \sigma\ne\sigma'\in\{+,-\}.\]
    On the other hand, using the evolution equation \eqref{pro_dyncalK} and integrating by parts, we obtain
        \begin{align}
        W^d\sum_{1\leq k<l\leq \fn}\cR_2^{\pa{k,l}}\pa{t}
          &  =\frac{1}{n^d}\int_{\ti}^t\rd u\, \pa{t-u}^{\fn}\sum_{\mathbf b\in (\wt\Z_n^d)^\fn}W^d\sum_{1\leq k<l\leq \fn} \sum_{\qa{a}\in \wt\Z_n^d} \pa{ \cutL^{\qa{a}}_{k, l} \circ {\mathcal{K}}^{(\fn)}_{u, \boldsymbol{\sigma}, \mathbf{b}} } \cdot \pa{ \cutR^{\qa{a}}_{k, l} \circ {\mathcal{K}}^{(\fn)}_{u, \boldsymbol{\sigma}, \mathbf{b}}} \nonumber\\
          & =\frac{1}{n^d}\int_{\ti}^t\rd u\, \pa{t-u}^{\fn}\frac{\rd}{\rd u}\sum_{\mathbf b\in (\wt\Z_n^d)^\fn}\cK_{u,\bsigma,\ba}^{\pa{\fn}} \nonumber\\
          &=\frac{\fn}{n^d}\int_{\ti}^t\rd u\, \pa{t-u}^{\fn-1}\sum_{\mathbf b\in (\wt\Z_n^d)^\fn}\cK_{u,\bsigma,\ba}^{\pa{\fn}}-\frac{1}{n^d} \pa{t-\ti}^{\fn}\sum_{\mathbf b\in (\wt\Z_n^d)^\fn}\cK_{\ti,\bsigma,\ba}^{\pa{\fn}}\nonumber\\
          &  =-\cR_1\pa{t}+\frac{\fn}{n^d}\int_{\ti}^t\rd u\, \pa{t-u}^{\fn-1}\sum_{\mathbf b\in (\wt\Z_n^d)^\fn}\cK_{u,\bsigma,\ba}^{\pa{\fn}}.\label{eq:R1t2}
        \end{align}
    Combining \eqref{eq:R1t1} and \eqref{eq:R1t2} yields that
    \begin{equation}\nonumber
            \cR_1\pa{t}+W^d\sum_{\substack{1\leq k<l\leq \fn:\\  2\leq l-k\leq \fn-2}}\cR_2^{\pa{k,l}}\pa{t}=\frac{\fn}{n^d}\int_{\ti}^t\rd u\,\frac{1-t}{1-u} \pa{t-u}^{\fn-1}\sum_{\mathbf b\in (\wt\Z_n^d)^\fn}\cK_{u,\bsigma,\ba}^{\pa{\fn}}.
    \end{equation}
    Substituting this identity into \eqref{first_reduction_cK_t}, and applying the bounds \eqref{sum_zero_cK} and \eqref{prop:ThfadC}, we obtain
    \begin{equation}\label{cK_final_error}
        \begin{aligned}
            \absa{\cK_{t,\bsigma,\ba}^{\pa{\fn}}}\prec \frac{\ell_t}{\p{\ell_t^{d}\eta_t}^{\fn}} \int_{\ti}^{t}\rd u\,\frac{1-t}{1-u}W^{-\pa{\fn-1}d}+\pa{W^d\ell_t^d\eta_t}^{-\pa{\fn-1}}\prec {\pa{W^d\ell_t^d\eta_t}^{-\pa{\fn-1}}},
        \end{aligned}
    \end{equation}
    which establishes \eqref{eq:bcal_k}.

    Finally, to complete Step (iii), it remains to establish \eqref{eq:bcal_k_improved}. For this improved bound, it suffices to show that the error terms in \eqref{scrU_1_first_error}, \eqref{scrU_2_first_error}, \eqref{first_reduction_cK_t}, and \eqref{cK_final_error} can all be replaced by \[\opr{\lambda^2/\pa{\lambda^2+1-t}\cdot\pa{W^d\ell_t^d\eta_t}^{-\pa{\fn-1}}}\]
    whenever the diagonal constraint $\qa{a_1}=\cdots=\qa{a_{\fn}}$ fails. First, the corresponding improvements for the non-diagonal parts in \eqref{scrU_1_first_error} and \eqref{scrU_2_first_error} are immediate. Indeed, the derivations of \eqref{scrU_1_first_error} and \eqref{scrU_2_first_error} always produce an additional factor $\lambda^2/\pa{\lambda^2+1-t}$ when \eqref{eq:bcal_k_improved} is applied to the $\cK$-loops at time $\ti$ or to $\cK$-loops of shorter length. The only exceptional situations are: (i)  in \eqref{scrU_1_first_error}, when $\qa{a_i}=\qa{b_i}$ for all $i\in \qq{\fn}$ and $\qa{b_1}=\cdots=\qa{b_\fn}$; (ii) in \eqref{scrU_2_first_error}, when $\qa{a_i}=\qa{b_i}$ for all $i\in \qq{\fn}$ and $\qa{b_1}=\cdots=\qa{b_\fn}=\qa{a}$. In either case, these conditions force $\qa{a_1}=\cdots=\qa{a_{\fn}}$, so no improvement is needed.
    Similarly, one checks that the argument leading to \eqref{first_reduction_cK_t} also contributes the same factor $\lambda^2/\pa{\lambda^2+1-t}$ whenever the loop is non-diagonal. Finally, for \eqref{cK_final_error}, the product \smash{$\prod_{i=1}^\fn\Theta^{\pa{\sigma_{i},\sigma_{i+1}}}_t\pa{[a_i], [a]}$} gains an extra factor $\lambda^2/\pa{\lambda^2+1-t}$ by \eqref{prop:ThfadC_improved}, unless $\qa{a_1}=\cdots=\qa{a_\fn}=[a]$. Combining all of these improvements, we obtain \eqref{eq:bcal_k_improved}. This completes the induction for Step (iii), and therefore concludes the proof of \Cref{bounds_on_primitive_loops}.
\end{proof}

\subsection{Proof of \Cref{lemma_properties_of_propagators}}

The first two properties in \Cref{lemma_properties_of_propagators} follow directly from the corresponding symmetries of $S_t$.
The estimate \eqref{prop:ThfadC_short_pointwise} was established in equation (4.21) of \cite{PartII}; taking a block average immediately yields \eqref{prop:ThfadC_short}. Moreover, the improved estimate \eqref{prop:ThfadC_short_improved} follows from \eqref{prop:ThfadC_short_pointwise} by an argument analogous to \eqref{example_difference_1}.
Next, we derive \eqref{prop:ThfadC_pointwise} and \eqref{prop:ThfadC_improved} from \eqref{prop:ThfadC}.

First, the pointwise bound \eqref{prop:ThfadC_pointwise} follows directly from \eqref{prop:ThfadC} via the identity
    \begin{equation}
          \wh \Theta_t^{(+,-)}= \pa{1-S_t}^{-1}=1+S_t+S_t\pa{1-S_t}^{-1}S_t.
    \end{equation}
By the flatness condition \eqref{def_C_flat}, the entries of the last term satisfy, for $x\in[a]$ and $y\in[b]$,
\[ [S_t\pa{1-S_t}^{-1}S_t]_{xy} \le \frac{C}{W^d}\sum_{[a'],[b']:|[a']-[a]|\le C, [b']-[b]|\le C} \Theta_{t,[a'][b']}^{(+,-)}\]
for some constant $C>0$. Given \eqref{prop:ThfadC_pointwise}, for any $\qa{x}\neq \qa{0}$ and large constant $D>0$, there exist constants $c,c',C>0$ such that
\begin{align}            &~\absa{\Theta_t^{\pa{\sigma_1,\sigma_2}}\pa{\qa{0},\qa{x}}-\fd\pa{\Theta}_t^{\pa{\sigma_1,\sigma_2}}\pa{\qa{0},\qa{x}}} \nonumber\\
\leq&~ \frac{1}{W^d} \sum_{i\in\qa{0}}\sum_{v,j\in\qa{x}}\sum_{u:|u-v|\leq CW}\pa{1-S_t}^{-1}_{iu}\p{S_t^{\txt{od}}}_{uv}\pa{1-S_t^{\txt{d}}}^{-1}_{vj} \nonumber\\
\prec&~ \frac{1}{W^d} \sum_{i\in\qa{0}}\sum_{v,j\in\qa{x}}\sum_{u:|u-v|\leq CW} \pa{\delta_{iu}+\frac{ e^{-c'|i-u|/(W\ell_t)}}{W^d\ell_t^d|1-t|}}\p{S_t^{\txt{od}}}_{uv}\pa{\delta_{vj}+\frac{1}{W^d}\frac{1}{\pa{\lambda^2+1-t}}}+W^{-D} \nonumber\\
\prec&~ \frac{\lambda^2}{\lambda^2+1-t}\frac{e^{-c|\qa{x}|/\ell_t}}{|1-t|\ell_t^d}+W^{-D},
\end{align}
where we used the decomposition $S_t=S_t^{\txt{d}}+S_t^{\txt{od}}$ (defined analogously to \eqref{decomposition_S_d_od}), the local fullness condition \eqref{def_epsilon_full} and the flatness condition \eqref{def_C_flat} for $S_t^{\txt{od}}$, the block-diagonal structure of $S_t^{\txt{d}}$, the pointwise bound in \Cref{lemma_mean_field_inverse_bound_point_wise} for $\p{1-S_t^{\txt{d}}}^{-1}$, and the definition of $\lambda$ in \eqref{def_lambda}.

It remains to prove the estimates \eqref{prop:ThfadC}, \eqref{prop:BD1}, and \eqref{prop:BD2} when $\sigma_1\neq \sigma_2$. In this case, recalling \eqref{def_projection_operator}, we may write
    \begin{equation*}            \Theta_t^{\pa{\sigma_1,\sigma_2}}=\scr{P}\p{\pa{1-S_t}^{-1}},
    \end{equation*}
    and for brevity, we will abbreviate $\Theta_t\equiv \Theta_t^{\pa{\sigma_1,\sigma_2}}$ in the following proof.
    Recall that $\SRBM$ is $(2\varepsilon_0)$-locally full. By the definition $\fS_t$, there exists a parameter $s_{\txt{flow}}\in \qa{t,\tf}$ such that
    \begin{equation}\label{s_flow}
        \begin{aligned}
            S_t=t\frac{\tf}{s_{\txt{flow}}}\SRBM+\pa{t-t\frac{\tf}{s_{\txt{flow}}}}\SE.
        \end{aligned}
    \end{equation}
    Note that
    \begin{equation*}
            2\varepsilon_0t\frac{\tf}{s_{\txt{flow}}}+\pa{t-t\frac{\tf}{s_{\txt{flow}}}}\geq 2\varepsilon_0\tf+\ti-\tf=\varepsilon_0\tf\sim 1,
    \end{equation*}
    and hence $S_t$ contains a positive locally full component. More precisely, there exists a constant $c_{\txt{ker}}>0$ and a locally $c_{\ker}$-full matrix $S_{\txt{ker}}$ with non-negative entries such that
    \begin{equation}\label{eq:decomposeSt}
            S_t=S_{\txt{ker}}+c_{\txt{ker}}\SE.
    \end{equation}
    A key observation is the following lemma, which provides a random-walk representation of $\Theta_t$.
    \begin{lemma}\label{lemma_random_walk_representation}
        The propagator $\Theta_t$ has the following random-walk representation:
        \begin{equation}\label{eq:Thetatrep}
                \Theta_t=\frac{1}{c_{\ker}}\frac{\wh t K}{1-\wh t K}=\frac{1}{c_{\ker}}\sum_{k=1}^{\infty}\wh t^{k}K^{k},
        \end{equation}
        where $\wh t>0$ and the $\Zn\times\Zn$ matrix $K$ is defined by
        \begin{equation}
            \begin{aligned}
                K:=\pa{1-t+c_{\txt{ker}}}\scr{P}\qa{\pa{1-S_{\txt{ker}}}^{-1}},\quad\wh t:=\frac{c_{\txt{ker}}}{1-t+c_\txt{ker}}.
            \end{aligned}
        \end{equation}
        Clearly, we have
        \begin{equation}\label{time_equivalence}
            \begin{aligned}
                1-\wh t\sim 1-t.
            \end{aligned}
        \end{equation}
    The matrix $K$ is doubly stochastic and translation invariant on $\Zn$, and therefore represents the transition probability matrix for a random walk \smash{$\wh R_n:=\wh X_1+\ldots+\wh X_n$} on $\Zn$. It satisfies the following properties for some constant $c>0$:
        \begin{enumerate}
            \item {\bf Exponential decay}:
            \begin{equation}\label{K_exponential_decay}
                \begin{aligned}
                    K_{\qa{0}\qa{x}}\leq c^{-1}\pa{\delta_{\qa{0}\qa{x}}+\lambda^2\pa{\SRBM} e^{-c\absa{\qa{x}}}}.
                \end{aligned}
            \end{equation}
            \item {\bf Variance}:
            \begin{equation}\label{K_total_variance}
                \begin{aligned}
                    c\lambda^2\pa{\SRBM}\leq \sum_{\qa{x}\in \Zn}\absa{\qa{x}}^2K_{\qa{0}\qa{x}}\leq c^{-1}\lambda^2\pa{\SRBM}.
                \end{aligned}
            \end{equation}
            \item {\bf Isotropy}: Let $\Sigma_K$ denote the covariance matrix of the random walk induced by $K$. Then, we have
            \begin{equation}\label{K_isotropy}
                \begin{aligned}
                    c\lambda^2\pa{\SRBM}I_d\leq \Sigma_K\leq c^{-1}\lambda^2\pa{\SRBM}I_d,
                \end{aligned}
            \end{equation}
            Here, $A\leq B$ denotes that $B-A$ is positive semi-definite.
        \end{enumerate}
    \end{lemma}

    \begin{proof}
    The identity \eqref{eq:Thetatrep} follows from a direct computation using the decomposition \eqref{eq:decomposeSt}, together with the facts that
     \((\SE)^2=\SE\) and \(\scr{P}(A)_{[a][b]}=W^d(\SE A\SE)_{xy}\) for any $\ZL\times \ZL$ matrix $A$ and any $x\in[a], y\in[b]$. The translation invariance of $K$ is inherited from the block translation invariance of $S_{\ker}$.
     To verify that $K$ is doubly stochastic, note that for every $x\in\mathbb{Z}_L$,
        \begin{equation}
            \begin{aligned}
                1-\sum_{y\in \ZL}\pa{S_{\ker}}_{xy}=1-\sum_{y\in \ZL}\pa{S_t-c_{\ker}\SE}_{xy}=1-t+c_{\ker}.
            \end{aligned}
        \end{equation}
        This also gives $\norm{S_{\ker}}_{\infty\to\infty}=t-c_{\ker}\le 1-c_{\ker}$, and therefore the exponential decay estimate \eqref{K_exponential_decay} follows readily from the geometric-series expansion
        \begin{equation}\label{K_expansion}
            \begin{aligned}
                K=\pa{1-t+c_{\ker}}\scr{P}\qa{\pa{1-S_{\ker}}^{-1}}=\pa{1-t+c_{\ker}}\sum_{k=0}^{\infty}\scr{P}\pa{S_{\ker}^k}.
            \end{aligned}
        \end{equation}
        The bound \eqref{K_exponential_decay} immediately yields the upper bound in \eqref{K_total_variance}, and the latter in turn implies the upper bound in \eqref{K_isotropy} via \smash{$\Sigma_K=\E\q{\wh X_1\wh X_1^{\top}}\leq \E\norm{\wh X_1}^2$}. Moreover, the expansion \eqref{K_expansion}, combined with the isotropy assumption \eqref{S_isotropy}, gives
        \[ \Sigma_K\ge \pa{1-t+c_{\ker}}^2\Sigma_{S_{\ker}} \gtrsim \Sigma_{\SRBM} \gtrsim \lambda^2(\SRBM)I_d, \]
        In the second step, we used that $S_{\ker}$ differs from $\p{t\cdot \tf/{s_{\txt{flow}}}}\cdot \SRBM$ only on block-diagonal entries, which do not contribute to the covariance. This yields the lower bounds in \eqref{K_total_variance} and \eqref{K_isotropy}.
\end{proof}

With the random-walk representation in \Cref{lemma_random_walk_representation}, the estimates \eqref{prop:ThfadC}, \eqref{prop:BD1}, and \eqref{prop:BD2} follow from standard arguments based on the Fourier-series expansion of $\Theta_t$ together with local CLT and large-deviation estimates for the random walk governed by $K$. Since the proofs proceed in a manner analogous to those of \cite[Lemma 3.10]{truong2025localizationlengthfinitevolumerandom}, \cite[Lemma 2.14]{Band2D}, and \cite[Lemma 2.19]{DYYY25_d3}, we omit the details here.

\section{Global laws}\label{sec_global_law}

In this section, we establish the global laws for the $G$-loops at the initial time $\ti$ of the flow. These results are stated in \Cref{lemma_global_law}, where the term ``global" refers to the regime in which the imaginary part of the spectral parameter is of order one.
For notational convenience—as well as for certain technical simplifications—\Cref{lemma_global_law} is formulated in a slightly more general framework introduced in \Cref{def:general_global}. The desired global laws for the $G$-loops defined in \Cref{Def:G_loop} at time $\ti$ then follow directly from \Cref{lemma_global_law} after an appropriate choice of the spectral parameter~$z$ and the variance–profile $S$ (up to a simple rescaling).

\begin{definition}\label{def:general_global}
Let $H$ be a Hermitian random matrix with independent (up to symmetry) Gaussian entries distributed as in \eqref{eq:distr_H}, with variance profile $S$ satisfying the $(2\varepsilon_0)$-local fullness and $C_0$-flat conditions in \eqref{def_epsilon_full} and \eqref{def_C_flat}. We also assume that $S$ is doubly stochastic.\footnote{We emphasize that we do not impose all assumptions in \Cref{def_considered_model} here.} Let \smash{$G\p{z}:=\pa{H-z}^{-1}$} denote the Green's function.
Consider a sequence of deterministic diagonal matrices \smash{$\ha{\fB^{\pa{i}}}_{i=1}^{\infty}$} satisfying
    \begin{equation}\label{eq:opBi}
            \sup_{i\geq 1}\norm{\fB^{\pa{i}}}=\OO\pa{W^{-d}},
    \end{equation}
and define their $\qa{a}$-blocks $\fB^{\pa{i}}_{[a]}$ for $\qa{a}\in \wt \Z_n^d$ as in \eqref{eq:Baa}.
With a slight abuse of notation, we write
\[G\p{+}:=G\p{z},\quad G\p{-}:=G^*\p{z},\quad m\p{+}:=m\p{z},\quad m\p{-}:=\ol{m}\p{z}.\]
For $\bsigma\in\ha{-,+}^{\fn}$, $\ba\in\p{\wt \Z_n^d}^{\fn}$, and \smash{$\fB:=(\fB^{\pa{1}},\ldots, \fB^{\pa{\fn}})$}, we define the $G$-loops and $\cK$-loops by
    \begin{equation}\label{def_general_cL_and_cK}
        \begin{aligned}
            \cL_{\fB,\bsigma,\ba}^{\pa{\fn}}&:=\avga{G\p{\sigma_1}\mathfrak{B}_{\qa{a_1}}^{\pa{1}}G\p{\sigma_2}\mathfrak{B}_{\qa{a_2}}^{\pa{2}}\cdots G\p{\sigma_\fn}\mathfrak{B}_{\qa{a_\fn}}^{\pa{\fn}}},\\
            \cK_{\mathfrak{B},\bsigma,\ba}^{\pa{\fn}}&:=\sum_{x_1\in\qa{a_1},\ldots,x_\fn\in\qa{a_\fn}}\prod_{i=1}^{\fn}\fB_{\qa{a_i}}^{\pa{i}}\pa{x_i,x_i}\wh \cK_{\bsigma,\pa{x_1,\ldots,x_\fn}}^{\pa{\fn}},
        \end{aligned}
    \end{equation}
    where the $\wh \cK$-loops are defined in \Cref{Def_Ktza}.
    Similarly, the $G$- and $\cK$-chains are defined for $x\in\Z_L^d$ as
    \begin{equation}
        \begin{aligned}
            \cC_{\fB,\bsigma,\ba}^{\pa{\fn}}\pa{x}&:=\pa{G\p{\sigma_1}\mathfrak{B}_{\qa{a_1}}^{\pa{1}}G\p{\sigma_2}\mathfrak{B}_{\qa{a_2}}^{\pa{1}}\cdots\mathfrak{B}_{\qa{a_{\fn-1}}}^{\pa{\fn-1}} G\p{\sigma_\fn}}_{xx},\\
            \cK_{\mathfrak{B},\bsigma,\ba}^{\pa{\fn}}\pa{x}&:=\sum_{x_1\in\qa{a_1},\ldots,x_{\fn-1}\in\qa{a_{\fn-1}}}\prod_{i=1}^{\fn-1}\fB_{\qa{a_i}}^{\pa{i}}\pa{x_i,x_i}\wh \cK_{\bsigma,\pa{x_1,\ldots,x_{\fn-1},x}}^{\pa{\fn}},
        \end{aligned}
    \end{equation}
    for \smash{$\fB:=(\fB^{\pa{1}},\ldots, \fB^{\pa{\fn-1}})$}.
\end{definition}

\begin{lemma}\label{lemma_global_law}
In the setting of \Cref{def:general_global}, for any small constant $\fc\in(0,1)$ and fixed $\fn\in \N$, the following global laws hold for spectral parameters $z=E+\ii \eta$.
    \begin{itemize}
        \item {\bf Averaged global law}: Uniformly in $\bsigma\in\ha{-,+}^{\fn}$ and $\ba\in\p{\wt \Z_n^d}^{\fn}$,
        \begin{equation}\label{global_law_average}
              \sup_{z:\absa{E}\leq \fc^{-1}, \eta\in\qa{\fc,\fc^{-1}}}\absa{\cL_{\fB,\bsigma,\ba}^{\pa{\fn}}-\cK_{\mathfrak{B},\bsigma,\ba}^{\pa{\fn}}}\prec {W^{-\fn d}}.
        \end{equation}
        \item {\bf Isotropic global law}: Uniformly in $\bsigma\in\ha{-,+}^{\fn}$, $\ba\in\p{\wt \Z_n^d}^{\fn-1}$, and $x\in\Z_L^d$,
        \begin{equation}\label{global_law_isotropic}
 \sup_{z:\absa{E}\leq \fc^{-1},  \eta\in\qa{\fc,\fc^{-1}}}\absa{\cC_{\fB,\bsigma,\ba}^{\pa{\fn}}\pa{x}-\cK_{\mathfrak{B},\bsigma,\ba}^{\pa{\fn}}\pa{x}}\prec W^{-\pa{\fn-1/2} d} .
        \end{equation}
    \end{itemize}
\end{lemma}

We also obtain the following improved global laws in the case $\fn=2$.

    \begin{lemma}\label{improved_global_laws}
In the setting of \Cref{def:general_global}, take any deterministic diagonal matrices $\fB^{\pa{1}},\fB^{\pa{2}}$ satisfying $\max\ha{\norm{\fB^{\pa{1}}},\norm{\fB^{\pa{2}}}}=\OO\pa{W^{-d}}$. Then, the following global laws hold for spectral parameters $z=E+\ii \eta$.
        \begin{itemize}
            \item $\txt{\textbf{Expected global law}}$:
            Uniformly in $\bsigma\in\ha{+,-}^2$ and \smash{$\ba\in\p{\wt \Z_n^d}^2$},
            \begin{equation}\label{expected_global_law}
                    \sup_{z:\absa{E}\leq \fc^{-1}, \eta\in\qa{\fc,\fc^{-1}}}\absa{\E\cL_{\fB,\bsigma,\ba}^{\pa{2}}-\cK_{\fB,\bsigma,\ba}^{\pa{2}}}\prec {W^{-3d}}.
            \end{equation}
            \item $\txt{\textbf{Global law with decay}}$: Uniformly in $\bsigma\in\ha{+,-}^2$ and \smash{$\ba\in\p{\wt \Z_n^d}^2$}, for any fixed constant $D>0$,
            \begin{equation}\label{global_law_with_decay}
                   \absa{ \cL_{\fB,\bsigma,\ba}^{\pa{2}}-\cK_{\fB,\bsigma,\ba}^{\pa{2}} }\prec {W^{-2d} e^{-\absa{\qa{a_1}-\qa{a_2}}^{1/2}}+W^{-D}}.
            \end{equation}
        \end{itemize}
    \end{lemma}

\subsection{Proof of \Cref{lemma_global_law}}
We prove \Cref{lemma_global_law} by induction on $\fn$. For simplicity, we abbreviate $G\equiv G\pa{z}$ and $m\equiv m\pa{z}$ in the following proof.
First, to establish the global law for $\fn=1$, we follow the approach of \cite{He2018} and study the matrix self-consistent equation $\Pi\p{G}=0$, where
    \begin{equation}  \label{eq:PiGdef}
    \Pi\p{G}:=I+zG+\cS\q{G}G,
    \end{equation}
    where $\cS\equiv\cS_S$ is defined in \eqref{def_cS}. Assume that for some deterministic control parameter $\phi>0$, we have
    \begin{equation}\label{eq:bddPiG}
\max_{x,y\in\Z_L^d} \abs{\Pi\p{G}_{xy}}\prec \phi.
    \end{equation}
For brevity, set $\cQ:=\Pi\p{G}=HG+\cS\q{G}G$. Then, for any fixed $p\in \N$, we have
    \begin{equation*}
        \begin{aligned}
            \E\absa{\cQ_{xy}}^{2p}&=\E\pa{HG}_{xy}\ol{\cQ}_{xy}\absa{\cQ_{xy}}^{2p-2}+\pa{\cS\q{G}G}_{xy}\ol{\cQ}_{xy}\absa{\cQ_{xy}}^{2p-2}\\
            &=\sum_{w}S_{xw}\E G_{wy}\partial_{wx}\pa{\ol{\cQ}_{xy}\absa{\cQ_{xy}}^{2p-2}},
        \end{aligned}
    \end{equation*}
   where the second step uses Gaussian integration by parts. Using the identities
    \begin{equation*}
        \begin{aligned}
            &\partial_{wx}\cQ_{xy}=-\cQ_{xw}G_{xy}+\delta_{wx}G_{xy}-G_{xy}\sum_{\alpha}S_{x\alpha}G_{\alpha w}G_{x\alpha},\\
            &\partial_{wx}\ol{\cQ}_{xy}=-\ol{\cQ}_{xx}G_{wy}+\ol{G}_{wy}-\ol{G}_{xy}\sum_{\alpha}S_{x\alpha}\ol{G}_{\alpha x}\ol{G}_{w\alpha},
        \end{aligned}
    \end{equation*}
    we derive that
        \begin{align}
            \E\absa{\cQ_{xy}}^{2p}=&~p\sum_{w}S_{xw}\E G_{wy}\pa{-\ol{\cQ}_{xx}G_{wy}+\ol{G}_{wy}-\ol{G}_{xy}\sum_{\alpha}S_{x\alpha}\ol{G}_{\alpha x}G_{w\alpha}}{\absa{\cQ_{xy}}^{2p-2}}\nonumber\\
            &~+\pa{p-1}\sum_{w}S_{xw}\E G_{wy}\pa{-\cQ_{xw}G_{xy}+\delta_{wx}G_{xy}-G_{xy}\sum_{\alpha}S_{x\alpha}G_{\alpha w}G_{x\alpha}}\ol{\cQ}_{xy}^{p}\cQ_{xy}^{p-2}\nonumber\\
            \prec &~ \pa{\phi W^{-d/2}+W^{-d}} {\E\absa{\cQ_{xy}}^{2p-2}} .\label{estimate_global_entrywise_self_consistent_equtaion}
        \end{align}
In the last step we used \eqref{eq:bddPiG}, the $C_0$-flatness condition \eqref{def_C_flat}, the Cauchy–Schwarz inequality, and the bounds
    \begin{equation}
        \begin{aligned}
        \sum_{w}\absa{G_{xw}}^2=\pa{GG^*}_{xx}\lesssim 1,\quad \sum_{w}\absa{G_{wx}}^2=\pa{G^*G}_{xx}\lesssim 1,
        \end{aligned}
    \end{equation}
 which follow from $\norm{G}\le \eta^{-1}\lesssim{1}$.
From \eqref{estimate_global_entrywise_self_consistent_equtaion}, Hölder’s and Young’s inequalities yield the self-improving estimate
    \begin{equation}\max_{x,y\in\Z_L^d} \abs{\Pi\p{G}_{xy}}\prec \phi\quad\Longrightarrow\quad \max_{x,y\in\Z_L^d} \abs{\Pi\p{G}_{xy}}\prec \phi^{1/2} W^{-d/4}+W^{-d/2}.
    \end{equation}
    Starting from $\phi=1$ and iterating this bound $\OO\pa{1}$ times, we conclude that    \begin{equation}\label{result_global_entrywise_self_consistent_equation}\max_{x,y\in\Z_L^d} \abs{\Pi\p{G}_{xy}}\prec W^{-d/2}.
    \end{equation}

To derive the global law for $\fn=1$ from the bound on $\Pi(G)$, we use the identity    \begin{equation}\label{solving_matrix_self_consistent_equation}
    G=-\pa{z+\cS\qa{G}}^{-1}+\pa{z+\cS\qa{G}}^{-1}\Pi\pa{G}.
    \end{equation}
Subtracting the diagonal part of the scalar self-consistent equation $m=-(z+m)^{-1}$ from the diagonal entries of \eqref{solving_matrix_self_consistent_equation}, we obtain
    \begin{equation}\label{eq:vx}
    v_x=\frac{m^2\sum_{y}S_{xy}v_y}{1-m\sum_{y}S_{xy}v_y}+\opr{W^{-d/2}},
    \end{equation}
    where we set $v_x:=\pa{G-m}_{xx}$ and used the bound \eqref{result_global_entrywise_self_consistent_equation} for $\Pi\p{G}_{xx}$.
    For $\eta\geq 2$, we have $\absa{m}\leq 1/2$ and $\|G\|\le \eta^{-1}\le 1/2$, which implies $\max_x \absa{v_x}\leq 1$. Thus, from \eqref{eq:vx}, we deduce that for $\eta\ge 2$,
    \begin{equation*}
            \max_{x}\absa{v_x}\leq \frac{1}{2}\max_{x}\absa{v_x}+\opr{W^{-d/2}}\ \implies \ \max_{x}\absa{\pa{G-m}_{xx}}\prec {W^{-d/2}}.
    \end{equation*}
    For $x\neq y$, combining \eqref{result_global_entrywise_self_consistent_equation} with \eqref{solving_matrix_self_consistent_equation}, we obtain
    \begin{equation*}            G_{xy}=\pa{z+\cS\qa{G}}_{xx}^{-1}\Pi\pa{G}_{xy}\prec {W^{-d/2}}.
    \end{equation*}
     Putting together the diagonal and off-diagonal estimates, we obtain the entrywise local law
    \[\norm{G-m}_{\max}\prec {W^{-d/2}}, \quad \text{for}\ \  \eta\geq 2.\]
    On the other hand, for $\eta\geq \fc$, we have $\norm{\pa{1-m^2S}^{-1}}_{\infty\to\infty}\lesssim 1$. Therefore, applying \Cref{from_2_loop_to_1_chain} below with control parameter $\Psi^2=W^{-d}$, and using a standard continuity argument (see e.g., \cite[Section 5]{Semicircle} and \cite[Section 5.1]{Band1D}), we conclude the global laws for $\fn=1$.

The remainder of the proof is by induction. Assume that the estimates \eqref{global_law_average} and \eqref{global_law_isotropic} hold for loop lengths $1,\ldots,\fn$. We will show that they also hold for loop length $\fn+1$. We only present the proof of \eqref{global_law_average}; the proof of \eqref{global_law_isotropic} is analogous. We adopt the following notation from \cite{cipolloni-erdos2021}.

\begin{definition}\label{def:underline}
Let $f$ and $g$ be matrix-valued functions. Define
\be\label{eq:gHf}
\underline{g(H)Hf(H)} := g(H)Hf(H)-\widetilde{\E}g(H)\widetilde{H}(\partial_{\widetilde{H}}f)(H) - \widetilde{\E}(\partial_{\widetilde{H}}g)(H)\widetilde{H}f(H),
\ee
where $\widetilde{H}$ is an independent copy of $H$, $\widetilde{\E}$ denotes the partial expectation with respect to $\wt H$, and $(\partial_{\widetilde{H}}f)(H)$ denotes the directional derivative of $f$ at the point $H$ in the direction \smash{$\wt H$}, i.e.,
\be\nonumber
[(\partial_{\widetilde{H}}f)(H)]_{xy} = (\widetilde{H}\cdot \nabla f(H))_{xy} := \sum_{\alpha,\beta\in \cI}\widetilde{H}_{\alpha\beta}\frac{\partial f(H)_{xy}}{\partial H_{\alpha\beta}}.
\ee
Note that the subtracted terms are exactly the second-order contributions arising from Gaussian integration by parts. In particular, one has $\E\underline{g(H)Hf(H)} = 0$.
\end{definition}

With $G(z)=(H-z)^{-1}$ and the self-consistent equation $m=-(z+m)^{-1}$, we can verify the identity
\be\label{eq:idG-m}G-m=-m\pa{H+m}G.\ee
Applying this identity to $G\p{\sigma_1}$ inside \smash{$\cL_{\fB,\bsigma,\ba}^{\pa{\fn+1}}$} yields
    \begin{equation*}
        \begin{aligned}
\avga{G\p{\sigma_1}\mathfrak{B}_{\qa{a_1}}^{\pa{1}}G\p{\sigma_2}\mathfrak{B}_{\qa{a_2}}^{\pa{2}}\cdots G\p{\sigma_{\fn+1}}\mathfrak{B}_{\qa{a_{\fn+1}}}^{\pa{\fn+1}}}&=m\p{\sigma_1} \avga{\mathfrak{B}_{\qa{a_1}}^{\pa{1}}G\p{\sigma_2}\mathfrak{B}_{\qa{a_2}}^{\pa{2}}\cdots G\p{\sigma_{\fn+1}}\mathfrak{B}_{\qa{a_{\fn+1}}}^{\pa{\fn+1}}}\\
&-m\p{\sigma_1}\avga{\pa{H+m\p{\sigma_1}}G\p{\sigma_1}\mathfrak{B}_{\qa{a_1}}^{\pa{1}}G\p{\sigma_2}\mathfrak{B}_{\qa{a_2}}^{\pa{2}}\cdots G\p{\sigma_{\fn+1}}\mathfrak{B}_{\qa{a_{\fn+1}}}^{\pa{\fn+1}}}.
        \end{aligned}
    \end{equation*}
Using the underline notation in \Cref{def:underline} and inserting an arbitrary deterministic matrix $\wt\fB_{\fn+1}$ in place of \smash{$\mathfrak{B}_{\qa{a_{\fn+1}}}^{\pa{\fn+1}}$}, we obtain the expansion
    \begin{align}
& \avga{G\p{\sigma_1}\mathfrak{B}_{\qa{a_1}}^{\pa{1}}G\p{\sigma_2}\mathfrak{B}_{\qa{a_2}}^{\pa{2}}\cdots  G\p{\sigma_{\fn+1}}\wt \fB_{\fn+1}} =m\p{\sigma_1}\avga{\mathfrak{B}_{\qa{a_1}}^{\pa{1}}G\p{\sigma_2}\mathfrak{B}_{\qa{a_2}}^{\pa{2}}\cdots G\p{\sigma_{\fn+1}}\wt \fB_{\fn+1}}\nonumber\\
&+m\p{\sigma_1}\avga{\cS\qa{G\p{\sigma_1}-m\p{\sigma_1}}{G\p{\sigma_1}\mathfrak{B}_{\qa{a_1}}^{\pa{1}}G\p{\sigma_2}\mathfrak{B}_{\qa{a_2}}^{\pa{2}}\cdots G\p{\sigma_{\fn+1}}\wt \fB_{\fn+1}}}\nonumber\\
&+m\p{\sigma_1}\sum_{k=2}^{\fn}\avga{\cS\q{G\p{\sigma_1}\mathfrak{B}_{\qa{a_1}}^{\pa{1}}\cdots G\p{\sigma_k}}{G\p{\sigma_k}\mathfrak{B}_{[a_k]}^{(k)}\cdots G\p{\sigma_{\fn+1}}\wt \fB_{\fn+1}}}\nonumber\\
&+m\p{\sigma_1}\avga{{G\p{\sigma_1}\mathfrak{B}_{\qa{a_1}}^{\pa{1}}G\p{\sigma_2}\mathfrak{B}_{\qa{a_2}}^{\pa{2}}\cdots G\p{\sigma_{\fn+1}}}\cS\q{\pa{G\p{\sigma_{\fn+1}}-m\p{\sigma_{\fn+1}}}\wt \fB_{\fn+1}}}\nonumber\\
&+m\p{\sigma_{\fn+1}}m\p{\sigma_1}\avga{G\p{\sigma_1}\mathfrak{B}_{\qa{a_1}}^{\pa{1}}G\p{\sigma_2}\mathfrak{B}_{\qa{a_2}}^{\pa{2}}\cdots G\p{\sigma_{\fn+1}}\cS\q{\wt \fB_{\fn+1}}}\nonumber\\
&-m\p{\sigma_1}\avga{\ul{HG\p{\sigma_1}\mathfrak{B}_{\qa{a_1}}^{\pa{1}}G\p{\sigma_2}\mathfrak{B}_{\qa{a_2}}^{\pa{2}}\cdots G\p{\sigma_{\fn+1}}\wt \fB_{\fn+1}}}.\label{cumulant_expansion_global_recursive}
\end{align}
Now move the fifth term on the RHS of
\eqref{cumulant_expansion_global_recursive} to the left and set $\wt \fB_{\fn+1}=\pa{I-m\p{\sigma_1}m\p{\sigma_{\fn+1}}\cS\qa{\cdot}}^{-1}\fB_{\qa{a_{\fn+1}}}^{\pa{\fn+1}}$. This yields that
\begin{align}
&\avga{G\p{\sigma_1}\mathfrak{B}_{\qa{a_1}}^{\pa{1}}G\p{\sigma_2}\mathfrak{B}_{\qa{a_2}}^{\pa{2}}\cdots G\p{\sigma_{\fn+1}}\mathfrak{B}_{\qa{a_{\fn+1}}}^{\pa{\fn+1}}} =m\p{\sigma_1}\avga{\mathfrak{B}_{\qa{a_1}}^{\pa{1}}G\p{\sigma_2}\mathfrak{B}_{\qa{a_2}}^{\pa{2}}\cdots G\p{\sigma_{\fn+1}}\wt \fB_{\fn+1}}\nonumber\\
&+m\p{\sigma_1}\avga{\cS\qa{G\p{\sigma_1}-m\p{\sigma_1}}{G\p{\sigma_1}\mathfrak{B}_{\qa{a_1}}^{\pa{1}}G\p{\sigma_2}\mathfrak{B}_{\qa{a_2}}^{\pa{2}}\cdots G\p{\sigma_{\fn+1}}\wt \fB_{\fn+1}}}\nonumber\\
&+m\p{\sigma_1}\sum_{k=2}^{\fn}\avga{\cS\q{G\p{\sigma_1}\mathfrak{B}_{\qa{a_1}}^{\pa{1}}\cdots G\p{\sigma_k}}{G\p{\sigma_k}\mathfrak{B}_{\qa{a_k}}^{\pa{k}}\cdots G\p{\sigma_{\fn+1}}\wt \fB_{\fn+1}}}\nonumber\\
&+m\p{\sigma_1}\avga{{G\p{\sigma_1}\mathfrak{B}_{\qa{a_1}}^{\pa{1}}G\p{\sigma_2}\mathfrak{B}_{\qa{a_2}}^{\pa{2}}\cdots G\p{\sigma_{\fn+1}}}\cS\q{\pa{G\p{\sigma_{\fn+1}}-m\p{\sigma_{\fn+1}}}\wt \fB_{\fn+1}}}\nonumber\\
&-m\p{\sigma_1}\avga{\ul{HG\p{\sigma_1}\mathfrak{B}_{\qa{a_1}}^{\pa{1}}G\p{\sigma_2}\mathfrak{B}_{\qa{a_2}}^{\pa{2}}\cdots G\p{\sigma_{\fn+1}}\wt \fB_{\fn+1}}}.\label{recursive_relation_global_loop}
\end{align}
We now record a convenient decay property of the matrix $\wt \fB_{\fn+1}$, which will be used repeatedly.
\begin{claim}\label{claim:decayfB}
Let $\fB$ be a deterministic diagonal matrix supported in the block $[a]$, i.e., $\fB=I_{\qa{a}}\fB I_{\qa{a}}$, with $\norm{\fB}=\OO\pa{1}$. Define \smash{$\wt \fB:=\pa{I-m\p{\sigma}m\p{\sigma'}\cS\qa{\cdot}}^{-1}\fB$} for $\sigma,\sigma'\in\ha{+,-}$. Then, there exists a constant $c >0$ such that
    \begin{equation}\label{decay_of_Theta_global}
            \norm{I_{\qa{x}}\wt \fB I_{\qa{x}}}\lesssim e^{-c\absa{\qa{x}-\qa{a}}}.
    \end{equation}
    \end{claim}
\begin{proof}
    Expand $\p{I-m\p{\sigma}m\p{\sigma'}\cS\qa{\cdot}}^{-1}$ as a geometric series. Since $1-|m|\gtrsim1$ for $ z=E+\ii\eta$ with $\eta\in\q{\fc,\fc^{-1}}$, the series converges and each power of $\cS$ spreads mass only over nearby blocks by the flatness condition. The claimed exponential blockwise decay follows.
\end{proof}

Using \eqref{decay_of_Theta_global} together with the assumption \eqref{eq:opBi} and the averaged global law \eqref{global_law_average} for $\fn=1$, we estimate the second term on the RHS of \eqref{recursive_relation_global_loop} as
    \begin{align}
    		&\avga{\cS\qa{G\p{\sigma_1}-m\pa{\sigma_1}}{G\p{\sigma_1}\mathfrak{B}_{\qa{a_1}}^{\pa{1}}G\p{\sigma_2}\mathfrak{B}_{\qa{a_2}}^{\pa{2}}\cdots G\p{\sigma_{\fn+1}}\wt \fB_{\fn+1}}}\nonumber\\
    		=&~ \sum_{\qa{a}}\sum_{x\in\qa{a}}\avga{\pa{G\p{\sigma_1}-m\p{\sigma_1}}S^{\pa{x}}}\pa{G\p{\sigma_1}\mathfrak{B}_{\qa{a_1}}^{\pa{1}}G\p{\sigma_2}\mathfrak{B}_{\qa{a_2}}^{\pa{2}}\cdots G\p{\sigma_{\fn+1}}I_{\qa{a}}\wt \fB_{\fn+1}I_{\qa{a}}}_{xx}\nonumber\\
    		\prec&~ \sum_{\qa{a}}\sum_{x\in\qa{a}} W^{-d}\cdot W^{-\pa{\fn+1}d}\cdot e^{-c\absa{\qa{a}-\qa{a_{\fn+1}}}}\lesssim W^{-\pa{\fn+1}d},\label{global_bound_error_term_1}
    \end{align}
    where recall that $S^{(x)}$ is defined in \eqref{eq:Sxij}. An identical argument bounds the fourth term on the RHS of \eqref{recursive_relation_global_loop} by $\opr{W^{-\pa{\fn+1}d}}$.
    Next, we control the underline term
    \begin{equation*}    		\cQ:=\avga{\ul{HG\p{\sigma_1}\mathfrak{B}_{\qa{a_1}}^{\pa{1}}G\p{\sigma_2}\mathfrak{B}_{\qa{a_2}}^{\pa{2}}\cdots G\p{\sigma_{\fn+1}}\wt \fB_{\fn+1}}}.
    \end{equation*}
    To bound this term, we estimate its $(2p)$-th moment for any fixed (large) $p\in \N$. Using Gaussian integration by parts, we get
    	\begin{align}
    		\E\absa{\cQ}^{2p}&=\E\avga{\ul{HG\p{\sigma_1}\mathfrak{B}_{\qa{a_1}}^{\pa{1}}G\p{\sigma_2}\mathfrak{B}_{\qa{a_2}}^{\pa{2}}\cdots G\p{\sigma_{\fn+1}}\wt \fB_{\fn+1}}}\ol{\cQ}\absa{\cQ}^{2p-2}\nonumber\\
    		&=\sum_{x,y}S_{xy}\pa{G\p{\sigma_1}\mathfrak{B}_{\qa{a_1}}^{\pa{1}}G\p{\sigma_2}\mathfrak{B}_{\qa{a_2}}^{\pa{2}}\cdots G\p{\sigma_{\fn+1}}\wt \fB_{\fn+1}}_{yx}\partial_{yx}\pa{\ol{\cQ}\absa{\cQ}^{2p-2}}\nonumber\\
    		&=p\sum_{x,y}S_{xy}\pa{G\p{\sigma_1}\mathfrak{B}_{\qa{a_1}}^{\pa{1}}G\p{\sigma_2}\mathfrak{B}_{\qa{a_2}}^{\pa{2}}\cdots G\p{\sigma_{\fn+1}}\wt \fB_{\fn+1}}_{yx}\pa{\partial_{yx}\ol{\cQ}}\cdot{\absa{\cQ}^{2p-2}}\nonumber\\
    		&+\pa{p-1}\sum_{x,y}S_{xy}\pa{G\p{\sigma_1}\mathfrak{B}_{\qa{a_1}}^{\pa{1}}G\p{\sigma_2}\mathfrak{B}_{\qa{a_2}}^{\pa{2}}\cdots G\p{\sigma_{\fn+1}}\wt \fB_{\fn+1}}_{yx}\pa{\partial_{yx}\cQ}\cdot\ol{\cQ}^{p}{\cQ^{p-2}}=: \txt{I}+\txt{II},\label{cumulant_expansion_global_loop}
    	\end{align}
    where the derivative of $\p{G\p{\sigma_1}\mathfrak{B}_{\qa{a_1}}^{\pa{1}}G\p{\sigma_2}\mathfrak{B}_{\qa{a_2}}^{\pa{2}}\cdots G\p{\sigma_{\fn+1}}\wt \fB_{\fn+1}}_{yx}$ produces exactly the cancellation encoded in the underline notation \eqref{eq:gHf}.

    For brevity, we bound only the term $\txt{II}$ in \eqref{cumulant_expansion_global_loop}, as the term $I$ can be treated by a similar argument. By definition, and using the identity $HG\pa{z}=I+zG\pa{z}$, we can write $\cQ$ as
   \begin{align}
    \cQ=& \avga{\pa{I+z^{\sigma_1}G\p{\sigma_1}}\mathfrak{B}_{\qa{a_1}}^{\pa{1}}G\p{\sigma_2}\mathfrak{B}_{\qa{a_2}}^{\pa{2}}\cdots G\p{\sigma_{\fn+1}}\wt \fB_{\fn+1}}\nonumber\\
    &-\sum_{i,j\in \ZL}S_{ij}\partial_{ji}\pa{G\p{\sigma_1}\mathfrak{B}_{\qa{a_1}}^{\pa{1}}G\p{\sigma_2}\mathfrak{B}_{\qa{a_2}}^{\pa{2}}\cdots G\p{\sigma_{\fn+1}}\wt \fB_{\fn+1}}_{ji},\label{eq:another_Qform}
    \end{align}
where $z^\sigma$ denotes $z$ when $\sigma=+$ and $\bar z$ when $\sigma=-$. Using \eqref{eq:another_Qform}, each contribution generated from $\txt{II}$ can be written in one of the following two forms:
    	\begin{align*}
            &\txt{(i):}\quad \frac{1}{W^{\pa{2\fn+1}d}}\sum_{x,y}S_{xy}\p{A\wt \fB_{\fn+1}}_{yx}A'_{xy},\\
    		&\txt{(ii):}\quad\frac{1}{W^{2\fn d}}\sum_{x,y}\sum_{i,j}S_{xy}S_{ij}\qa{\p{A\wt \fB_{\fn+1}}_{yx}\pa{A_1}_{jy}\pa{A_2}_{xj}\p{A_3\wt \fB_{\fn+1}}_{ii}+\p{A'\wt \fB_{\fn+1}}_{yx}\p{A_1'}_{jj}\p{A_2'}_{iy}\p{A_3'\wt \fB_{\fn+1}}_{xi}},
    	\end{align*}
    where $A,A'$ and $A_i,A_i'$ for $i\in\{1,2\}$ denote (random) matrices with operator norms of order $\OO\pa{1}$.
    For case (i), using \eqref{decay_of_Theta_global}, we can bound it as follows (recall \eqref{eq:Sxij}):
    \begin{align}
        \frac{1}{W^{\pa{2\fn+1}d}}\sum_{x,y}S_{xy}\p{A\wt \fB_{\fn+1}}_{yx}A'_{xy} &= \frac{1}{W^{\pa{2\fn+1}d}}\sum_{[a]}\sum_{x\in[a]} \p{A'S^{(x)}A\wt \fB_{\fn+1}I_{[a]}}_{xx} \nonumber\\
        &\lesssim \frac{1}{W^{(2\fn +1)d}} \sum_{\qa{a}}\sum_{x\in\qa{a}} W^{-2d} e^{-c\absa{\qa{a}-\qa{a_{\fn+1}}}}\lesssim W^{-\pa{2\fn+2}d}.\label{global_bound_error_term_20}
    \end{align}
    For case (ii), we rewrite the expression as
    \begin{equation*}
    	\begin{aligned}
    		&~\frac{1}{W^{2\fn d}}\sum_{x,y}\sum_{i,j}S_{xy}S_{ij}\qa{\p{A\wt \fB_{\fn+1}}_{yx}\pa{A_1}_{jy}\pa{A_2}_{xj}\p{A_3\wt \fB_{\fn+1}}_{ii}+\p{A'\wt \fB_{\fn+1}}_{yx}\p{A_1'}_{jj}\p{A_2'}_{iy}\p{A_3'\wt \fB_{\fn+1}}_{xi}}\\
    		=&~\frac{1}{W^{2\fn d}}\sum_{x}\sum_{i}\p{A_2S^{\pa{i}}A_1 S^{\pa{x}}A\wt \fB_{\fn+1}}_{xx}\p{A_3\wt \fB_{\fn+1}}_{ii}+\frac{1}{W^{2\fn d}}\sum_{x}\sum_{j}\p{A_3'\wt \fB_{\fn+1}S^{\pa{j}}A_2'S^{\pa{x}}A'\wt \fB_{\fn+1}}_{xx}\pa{A_1'}_{jj}.
    	\end{aligned}
    \end{equation*}
    Without loss of generality, we bound the second term; the first can be handled similarly but more easily. In this case, using \eqref{decay_of_Theta_global}, we obtain
         \begin{align}
            &~\frac{1}{W^{2\fn d}}\sum_{x}\sum_{j}\p{A_3'\wt \fB_{\fn+1}S^{\pa{j}}A_2'S^{\pa{x}}A'\wt \fB_{\fn+1}}_{xx}\pa{A_1'}_{jj}\nonumber\\
            =&~\frac{1}{W^{2\fn d}}\sum_{\qa{a},\qa{b}}\sum_{x\in\qa{a}}\sum_{j\in\qa{b}}\p{A_3'\wt \fB_{\fn+1}I_{\qa{b}}S^{\pa{j}}A_2'S^{\pa{x}}A'\wt \fB_{\fn+1}I_{\qa{a}}}_{xx}\pa{A_1'}_{jj}\nonumber\\
            \lesssim&~ \frac{1}{W^{2\fn d}} \sum_{\qa{a},\qa{b}}\sum_{x\in\qa{a}}\sum_{j\in\qa{b}}W^{-4d} e^{-c\absa{\qa{b}-\qa{a_{\fn+1}}}-c\absa{\qa{a}-\qa{a_{\fn+1}}}}\lesssim W^{-\pa{2\fn+2}d}.\label{global_bound_error_term_2}
        \end{align}
    Combining \eqref{global_bound_error_term_20} and \eqref{global_bound_error_term_2} yields
           \( \E\absa{\cQ}^{2p}\prec {W^{-\pa{2\fn+2}d}}\cdot {\E\absa{\cQ}^{2p-2}} .\)
    By Hölder’s inequality, it gives $\E\absa{\cQ}^{2p} \prec ({W^{-\pa{2\fn+2}d}})^p$, and hence, by Markov's inequality,
    \be\label{global_bound_error_term_cQ}
    \cQ\prec {W^{-\pa{\fn+1}d}}.\ee

Subsequently, it remains to estimate the leading terms (i.e., the first and third terms) on the RHS of \eqref{recursive_relation_global_loop}. For this purpose, we again employ the notation of generalized $\cK$-loops introduced in \eqref{def_generalized_cK}:
    \begin{equation*}            \cK_{\fB,\bsigma,\ba}^{\pa{k}}\equiv \cK_{\bsigma}^{\pa{k}}\pa{\fB_{\qa{a_1}}^{\pa{1}},\ldots,\fB_{\qa{a_k}}^{\pa{k}}},\quad \cK_{\fB,\bsigma,\ba}^{\pa{k}}\pa{x}\equiv \cK_{\bsigma}^{\pa{k}}\pa{\fB_{\qa{a_1}}^{\pa{1}},\ldots,\fB_{\qa{a_{k-1}}}^{\pa{k-1}},F_{x}},
    \end{equation*}
    where recall that $F_x$ denotes the matrix with \smash{$\pa{F_x}_{ij}=\mathbf{1}_{i=j=x}$}.
    By the induction hypothesis, the first term on the RHS of \eqref{recursive_relation_global_loop} can be written as
    \begin{align}
&~m(\sigma_1)\avga{\mathfrak{B}_{\qa{a_1}}^{\pa{1}}G\p{\sigma_2}\mathfrak{B}_{\qa{a_2}}^{\pa{2}}\cdots G\p{\sigma_{\fn+1}}\wt \fB_{\fn+1}} \nonumber\\
=&~m(\sigma_1)\cK_{\pa{\sigma_2,\ldots,\sigma_{\fn+1}}}^{\pa{\fn}}\pa{\mathfrak{B}_{\qa{a_2}}^{\pa{2}},\ldots,\mathfrak{B}_{\qa{a_{\fn}}}^{\pa{\fn}},\wt \fB_{\fn+1}\mathfrak{B}_{\qa{a_1}}^{\pa{1}}}+\opr{W^{-\pa{\fn+1}d}}. \label{eq:first_term}
    \end{align}
    For the third term on the RHS of \eqref{recursive_relation_global_loop}, we rewrite it as
   \begin{align}
    &~m\p{\sigma_1}\sum_{k=2}^{\fn}\avga{\cS\q{G\p{\sigma_1}\mathfrak{B}_{\qa{a_1}}^{\pa{1}}\cdots G\p{\sigma_k}}{G\p{\sigma_k}\mathfrak{B}_{\qa{a_k}}^{\pa{k}}\cdots G\p{\sigma_{\fn+1}}\wt \fB_{\fn+1}}}\nonumber\\
    =&~m\p{\sigma_1}\sum_{k=2}^{\fn}\sum_{x}\avga{G\p{\sigma_1}\mathfrak{B}_{\qa{a_1}}^{\pa{1}}\cdots G\p{\sigma_k}S^{\pa{x}}}\pa{{G\p{\sigma_k}\mathfrak{B}_{\qa{a_k}}^{\pa{k}}\cdots G\p{\sigma_{\fn+1}}\wt \fB_{\fn+1}}}_{xx}.\label{eq:second_term0}
    \end{align}
    Applying the induction hypothesis to the $G$-loop of length $k$, together with an argument analogous to that used in \eqref{global_bound_error_term_2}, we obtain
    \begin{align}
        \eqref{eq:second_term0}&=m\p{\sigma_1}\sum_{k=2}^{\fn}\sum_{x}\cK_{\pa{\sigma_1,\ldots,\sigma_k}}^{\pa{k}}\pa{\mathfrak{B}_{\qa{a_1}}^{\pa{1}},\ldots,\mathfrak{B}_{\qa{a_{k-1}}}^{\pa{k-1}},S^{\pa{x}}}\pa{{G\p{\sigma_k}\mathfrak{B}_{\qa{a_k}}^{\pa{k}}\cdots G\p{\sigma_{\fn+1}}\wt \fB_{\fn+1}}}_{xx}+\opr{W^{-\pa{\fn+1}d}}\nonumber\\
        &=m\p{\sigma_1}\sum_{k=2}^{\fn}\sum_{x}\cK_{\pa{\sigma_1,\ldots,\sigma_k}}^{\pa{k}}\pa{\mathfrak{B}_{\qa{a_1}}^{\pa{1}},\ldots,\mathfrak{B}_{\qa{a_{k-1}}}^{\pa{k-1}},S^{\pa{x}}}\cdot\cK_{\pa{\sigma_k,\ldots,\sigma_{\fn+1}}}^{\pa{\fn-k+2}}\pa{\mathfrak{B}_{\qa{a_k}}^{\pa{k}},\ldots,\mathfrak{B}_{\qa{a_\fn}}^{\pa{\fn}},\wt \fB_{\fn+1}F_x} \nonumber\\
        &\quad +\opr{W^{-\pa{\fn+1}d}}.\label{eq:second_term}
    \end{align}
    Here, in the second step, we view
    \begin{equation*}
        \begin{aligned}
            m\pa{\sigma_1}\sum_{k=2}^{\fn}\sum_{x}\cK_{\pa{\sigma_1,\ldots,\sigma_k}}^{\pa{k}}\pa{\mathfrak{B}_{\qa{a_1}}^{\pa{1}},\ldots,\mathfrak{B}_{\qa{a_{k-1}}}^{\pa{k-1}},S^{\pa{x}}}\pa{{G\p{\sigma_k}\mathfrak{B}_{\qa{a_k}}^{\pa{k}}\cdots G\p{\sigma_{\fn+1}}\wt \fB_{\fn+1}}}_{xx}
        \end{aligned}
    \end{equation*}
    as a $\cK$-loop rather than a $\cK$-chain, and apply the averaged global law \eqref{global_law_average}---together with the $\cK$-loop bound \eqref{bound_cK_global_l_1_l_infty_norm}\footnote{Although the setting of \Cref{claim_global_l_1_l_infty_norm} differs from ours, its proof carries over verbatim.}---from the induction hypothesis.
Combining \eqref{global_bound_error_term_1}, \eqref{global_bound_error_term_cQ}, \eqref{eq:first_term}, and \eqref{eq:second_term}, we obtain
\begin{align*}
&   \eqref{recursive_relation_global_loop}= \opr{W^{-\pa{\fn+1}d}}+ m\p{\sigma_1}\cK^{(\fn)}_{\pa{\sigma_2,\ldots,\sigma_{\fn+1}}}\pa{\mathfrak{B}_{\qa{a_2}}^{\pa{2}},\ldots,\mathfrak{B}_{\qa{a_{\fn}}}^{\pa{\fn}},\wt \fB_{\fn+1}\mathfrak{B}_{\qa{a_1}}^{\pa{1}}}\\
&+m\p{\sigma_1}\sum_{k=2}^{\fn}\sum_{x}\cK_{\pa{\sigma_1,\ldots,\sigma_k}}^{\pa{k}}\pa{\mathfrak{B}_{\qa{a_1}}^{\pa{1}},\ldots,\mathfrak{B}_{\qa{a_{k-1}}}^{\pa{k-1}},S^{\pa{x}}} \cK_{\pa{\sigma_k,\ldots,\sigma_{\fn+1}}}^{\pa{\fn-k+2}}\pa{\mathfrak{B}_{\qa{a_k}}^{\pa{k}},\ldots,\mathfrak{B}_{\qa{a_\fn}}^{\pa{\fn}},\wt \fB_{\fn+1}F_x}.
\end{align*}
Using the recursive relation \eqref{general_cK_recursion} (which clearly remains valid when $S_{\ti}$ is replaced by $S$), the last two terms on the RHS combine to give
\smash{$\cK^{\pa{\fn+1}}_{\fB,\bsigma,\ba}=\cK_{\bsigma}^{\pa{\fn+1}}\p{\mathfrak{B}_{\qa{a_1}}^{\pa{1}},\ldots,\mathfrak{B}_{\qa{a_{\fn+1}}}^{\pa{\fn+1}}}$}.
This establishes the averaged global law for loops of length $\fn+1$ and completes the induction.

\subsection{Proof of \Cref{improved_global_laws}}

For the proof of the expected global law \eqref{expected_global_law} for the 2-$G$-loops, we first establish an improved averaged global law for $\fn=1$ after taking expectation. Let $B$ be any deterministic matrix. Using the identity \eqref{eq:idG-m} and applying Gaussian integration by parts, we obtain
    \begin{align}
            \E\avga{\pa{G-m}B}&=-m\E\avga{\pa{H+m}GB}=m\E\sum_{i,j}S_{ij}\pa{G-m}_{jj}\pa{GB}_{ii}\nonumber\\
            &=m^2\E\avga{\pa{G-m}\cS\qa{B}}+m\E\sum_{i,j}S_{ij}\pa{G-m}_{jj}\qa{\pa{G-m}B}_{ii}.\label{first_expansion_of_light_weights}
        \end{align}
    Given any $\qa{x}\in\wt \Z_n^d$ and a deterministic matrix $\fB$ with $\norm{\fB}=\OO\pa{W^{-d}}$, we choose $B=\pa{I-m^2\cS\qa{\cdot}}^{-1}\fB_{\qa{x}}$ (recall \eqref{eq:Baa}) in equation \eqref{first_expansion_of_light_weights} and get
        \begin{align}
            \E\avga{\pa{G-m}\fB_{\qa{x}}}&=m\E\sum_{i,j}S_{ij}\pa{G-m}_{jj}\qa{\pa{G-m}B}_{ii} =-m^2\E\sum_{i,j}S_{ij}\qa{\pa{H+m}G}_{jj}\qa{\pa{G-m}B}_{ii}\nonumber\\
            &=m^2\E\sum_{i,j,k}S_{ij}S_{jk}\pa{G-m}_{kk}G_{jj}\qa{\pa{G-m}B}_{ii}+m^2\E\sum_{i,j,k}S_{ij}S_{jk}G_{kj}G_{ik}\pa{GB}_{ji}\nonumber\\
            &=m^2\E\sum_{j}\avg{\pa{G-m}S^{\pa{j}}}G_{jj}\avg{\pa{G-m}BS^{\pa{j}}}+m^2\E\sum_{j}\p{GBS^{\pa{j}}GS^{\pa{j}}G}_{jj},\label{second_expansion_of_light_weights}
        \end{align}
where the second step uses \eqref{eq:idG-m} again, and the third step applies Gaussian integration by parts once more.
Clearly, $B$ satisfies the same type of decay property as in \eqref{decay_of_Theta_global}, by exactly the same argument as in \Cref{claim:decayfB}. Therefore, by arguments analogous to those used in \eqref{global_bound_error_term_1} and \eqref{global_bound_error_term_2}, we deduce from \eqref{second_expansion_of_light_weights} that
    \begin{equation}\label{expect_single_resolvent_global_law}
        \begin{aligned}
            \E\avga{\pa{G-m}\fB_{\qa{x}}}\prec {W^{-2d}}.
        \end{aligned}
    \end{equation}

With \eqref{expect_single_resolvent_global_law} in hand, we now turn to the proof of \eqref{expected_global_law}. Again, using \eqref{eq:idG-m} and Gaussian integration by parts—following arguments analogous to those used in \eqref{cumulant_expansion_global_recursive} and \eqref{first_expansion_of_light_weights}—we obtain
 \begin{align} \E\avga{G\p{\sigma_1}\fB^{\pa{1}}_{\qa{a_1}}G\p{\sigma_2}\fB^{\pa{2}}_{\qa{a_2}}}&=m\p{\sigma_1}\E\avga{\fB^{\pa{1}}_{\qa{a_1}}G\p{\sigma_2}\wt \fB_2} +m\p{\sigma_1}\E\sum_{i,j}S_{ij}\pa{G\p{\sigma_1}-m\p{\sigma_1}}_{jj}\pa{G\p{\sigma_1}\fB^{\pa{1}}_{\qa{a_1}}G\p{\sigma_2}\wt \fB_2}_{ii} \nonumber\\
&+m\p{\sigma_1}\E\sum_{i,j}S_{ij}\pa{G\p{\sigma_1}\fB^{\pa{1}}_{\qa{a_1}}G\p{\sigma_2}}_{jj}\pa{\p{G\p{\sigma_2}-m\p{\sigma_2}}\wt \fB_2}_{ii},\label{first_expansion_two_resolvent_global_law}
\end{align}
where $\wt \fB_2:=\pa{I-m\p{\sigma_1}m\p{\sigma_2}\cS\qa{\cdot}}^{-1}\fB^{\pa{2}}_{\qa{a_2}}$. For the leading term, we have
\begin{align}
m\p{\sigma_1}\E\avga{\fB^{\pa{1}}_{\qa{a_1}}G\p{\sigma_2}\wt \fB_2}&=m\p{\sigma_1}m\p{\sigma_2}\avga{\fB^{\pa{1}}_{\qa{a_1}}\wt \fB_2}+\opr{W^{-3d}}\nonumber\\    &=\cK_{\bsigma}^{\pa{2}}\pa{\fB^{\pa{1}}_{\qa{a_1}},\fB^{\pa{2}}_{\qa{a_2}}}+\opr{W^{-3d}},\label{eq;expected2first}
\end{align}
where in the first step we used the decay property of $\wt \fB_2$ from \eqref{decay_of_Theta_global} (with $[a]=[a_2]$), together with the expected averaged global law \eqref{expect_single_resolvent_global_law}. The second step follows directly from the form of the 2-$\cK$-loop in \eqref{eq:2Kloop}.
It remains to bound the second and third terms in \eqref{first_expansion_two_resolvent_global_law}. For the second term on the RHS, applying \eqref{eq:idG-m} to $G(\sigma_1)-m(\sigma_1)$ and using Gaussian integration by parts yields the following expansion:
\begin{align}
&\quad\, \E\sum_{i,j}S_{ij}\pa{G\p{\sigma_1}-m\p{\sigma_1}}_{jj}\pa{G\p{\sigma_1}\fB^{\pa{1}}_{\qa{a_1}}G\p{\sigma_2}\wt \fB_2}_{ii}\nonumber\\
&=m\p{\sigma_1}\E\sum_{j}\avga{\pa{G\p{\sigma_1}-m\p{\sigma_1}}S^{\pa{j}}}G\p{\sigma_1}_{jj}\avga{G\pa{\sigma_1}\fB^{\pa{1}}_{\qa{a_1}}G\p{\sigma_2}\wt \fB_2S^{\pa{j}}}\nonumber\\
&+m\p{\sigma_1}\E\sum_{j}\pa{G\p{\sigma_1}\fB^{\pa{1}}_{\qa{a_1}}G\p{\sigma_2}\wt \fB_2S^{\pa{j}}G\p{\sigma_1}S^{\pa{j}}G\p{\sigma_1}}_{jj}\nonumber\\
&+m\p{\sigma_1}\E\sum_{j}\pa{G\p{\sigma_2}\wt \fB_2S^{\pa{j}}G\p{\sigma_1}\fB^{\pa{1}}_{\qa{a_1}}G\p{\sigma_2}S^{\pa{j}}G\p{\sigma_1}}_{jj}.\label{eq;expected2second}
\end{align}
First, we apply the decay property of $\wt \fB_2$ from \eqref{decay_of_Theta_global}, along with an argument analogous to \eqref{global_bound_error_term_2}, to bound the second and third terms on the RHS by $\opr{W^{-3d}}$. For the first term, we use the averaged global law \eqref{global_law_average} with $\fn=2$ to replace the 2-loop with its deterministic limit, obtaining
    \begin{align}
        &~\E\sum_{j}\avga{\pa{G\p{\sigma_1}-m\p{\sigma_1}}S^{\pa{j}}}G\p{\sigma_1}_{jj}\avga{G\pa{\sigma_1}\fB^{\pa{1}}_{\qa{a_1}}G\p{\sigma_2}\wt \fB_2S^{\pa{j}}}\nonumber\\
        \prec &~ m(\sigma_1) \E\sum_{[a]}\sum_{j\in[a]} \avga{\pa{G\p{\sigma_1}-m\p{\sigma_1}}S^{\pa{j}}} \cK^{(2)}_{\bsigma}\pa{\fB^{\pa{1}}_{\qa{a_1}},\wt \fB_2S^{\pa{j}}}+ \sum_{[a]}\sum_{j\in[a]} \opr{W^{-d}}\cdot W^{-3d}e^{-c|[a]-[a_2]|} \nonumber\\
        &~ +\E\sum_{[a]}\sum_{j\in[a]}\qa{G\p{\sigma_1}-m\p{\sigma_1}}_{jj}\avga{\pa{G\p{\sigma_1}-m\p{\sigma_1}}S^{\pa{j}}} \cK^{(2)}_{\bsigma}\pa{\fB^{\pa{1}}_{\qa{a_1}},\wt \fB_2S^{\pa{j}}} \nonumber\\
        =&~\E\sum_{[a]}\sum_{j\in[a]}\qa{G\p{\sigma_1}-m\p{\sigma_1}}_{jj}\avga{\pa{G\p{\sigma_1}-m\p{\sigma_1}}S^{\pa{j}}} \cK^{(2)}_{\bsigma}\pa{\fB^{\pa{1}}_{\qa{a_1}},\wt \fB_2S^{\pa{j}}} + \opr{W^{-3d}}.\label{eq:EGG-m}
    \end{align}
    In the first step above, we also used that
    \be\label{eq:B2Sj}
    \|\wt \fB_2S^{\pa{j}}\|\lesssim W^{-2d}e^{-c|[a]-[a_2]|}
    \ee
    from the decay property \eqref{decay_of_Theta_global}, together with the averaged global law \eqref{global_law_average} applied to $\avga{(G(\sigma_1)-m(\sigma_1))S^{(j)}}$.
    In the second step, we applied \eqref{expect_single_resolvent_global_law}, the $\cK$-loop bound \eqref{bound_cK_global_l_1_l_infty_norm}, and \eqref{eq:B2Sj}, which together imply
    \be\label{eq:B2-loop}
    \cK^{(2)}_{\bsigma}\pa{\fB^{\pa{1}}_{\qa{a_1}},\wt \fB_2S^{\pa{j}}} \prec W^{-2d}e^{-c|[a]-[a_2]|}.\ee
     For the first term on the RHS of \eqref{eq:EGG-m}, we again apply \eqref{eq:idG-m} to $(G(\sigma_1)-m(\sigma_1))_{jj}$ and use Gaussian integration by parts to get that
    \begin{align}
        &~\E\sum_{[a]}\sum_{j\in[a]}\qa{G\p{\sigma_1}-m\p{\sigma_1}}_{jj}\avga{\pa{G\p{\sigma_1}-m\p{\sigma_1}}S^{\pa{j}}} \cK^{(2)}_{\bsigma}\pa{\fB^{\pa{1}}_{\qa{a_1}},\wt \fB_2S^{\pa{j}}} \nonumber\\
        =&~ m(\sigma_1)\E\sum_{[a]}\sum_{j\in[a]}G(\sigma_1)_{jj} \avga{\pa{G\p{\sigma_1}-m\p{\sigma_1}}S^{\pa{j}}}^2 \cK^{(2)}_{\bsigma}\pa{\fB^{\pa{1}}_{\qa{a_1}},\wt \fB_2S^{\pa{j}}}\nonumber\\
        &~+ m(\sigma_1)\E\sum_{[a]}\sum_{j\in[a]}\pa{G(\sigma_1)S^{\pa{j}} G\p{\sigma_1} S^{\pa{j}} G(\sigma_1)}_{jj} \cK^{(2)}_{\bsigma}\pa{\fB^{\pa{1}}_{\qa{a_1}},\wt \fB_2S^{\pa{j}}}\prec W^{-3d},\label{eq:EGG-m2}
    \end{align}
    where in the second step, we used the averaged global law \eqref{global_law_average} for $\avg{\pa{G\p{\sigma_1}-m\p{\sigma_1}}S^{\pa{j}}}$, the $2$-$\cK$-loop bound \eqref{eq:B2-loop}, and the fact $\|G\|\le \eta^{-1}\lesssim 1$. Altogether, we deduce that
\begin{align*}
            &\quad\, \E\sum_{i,j}S_{ij}\pa{G\p{\sigma_1}-m\p{\sigma_1}}_{jj}\pa{G\p{\sigma_1}\fB^{\pa{1}}_{\qa{a_1}}G\p{\sigma_2}\wt \fB_2}_{ii} \prec W^{-3d}.
    \end{align*}
    The last term on the RHS of \eqref{first_expansion_two_resolvent_global_law} can be bounded in exactly the same way. Combining these bounds with \eqref{eq;expected2first} and \eqref{first_expansion_two_resolvent_global_law} completes the proof of \eqref{expected_global_law}.

    Finally, we prove the global law with decay \eqref{global_law_with_decay}. For this purpose, we introduce the following notation: for any $\ell \in \qq{0,2n}$ and a large constant $D>0$, define
    \begin{equation}
        \begin{aligned}
            \cT_{D}\pa{\ell}:=W^{-2d}\exp\pa{-\ell^{1/2}}+W^{-D}.
        \end{aligned}
    \end{equation}
    Let $\cJ_D^*\geq 1$ be a deterministic control parameter such that
    \begin{equation}\label{decay_global_law_assumed_control}
        \begin{aligned}
            \max_{\bsigma\in\ha{+,-}^2}\max_{\ba\in\p{\wt \Z_n^d}^2}\absa{\cL_{\fB,\bsigma,\ba}^{\pa{2}}-\cK_{\fB,\bsigma,\ba}^{\pa{2}}}\Big/\cT_{D}\pa{\absa{\qa{a_1}-\qa{a_2}}}\prec \cJ_D^*
        \end{aligned}
    \end{equation}
    holds uniformly for all deterministic matrices $\fB^{(1)},\fB^{(2)}$ satisfying \smash{$\max\ha{\norm{\fB^{\pa{1}}},\norm{\fB^{\pa{2}}}}\leq W^{-d}$} and belonging to a fixed set $\mathbf S$ of diagonal matrices with cardinality at most $N^C$ for some constant $C>0$.\footnote{The set $\mathbf S$ contains all deterministic diagonal matrices relevant to the argument below; we do not specify it explicitly.} Our goal is to prove that
    \begin{equation}\label{goal_global_law_with_decay}
        \begin{aligned}
            \max_{\bsigma\in\ha{+,-}^2}\max_{\ba\in\p{\wt \Z_n^d}^2}\absa{\cL_{\fB,\bsigma,\ba}^{\pa{2}}-\cK_{\fB,\bsigma,\ba}^{\pa{2}}}\Big/\cT_{D}\pa{\absa{\qa{a_1}-\qa{a_2}}}\prec 1.
        \end{aligned}
    \end{equation}
We establish \eqref{goal_global_law_with_decay} via an iteratively self-improving bound. First, consider the case \smash{$|\qa{a_1}-\qa{a_2}|\le (\log W)^{3/2}$}. By the averaged global law \eqref{global_law_average}, we already have
    \begin{equation}\label{eq:l2-k2}
            \absa{\cL_{\fB,\bsigma,\ba}^{\pa{2}}-\cK_{\fB,\bsigma,\ba}^{\pa{2}}}\prec W^{-2d}\prec \cT_D\pa{\absa{\qa{a_1}-\qa{a_2}}},
    \end{equation}
    since $\exp(C(\log W)^{3/4})\prec 1$. It remains to treat the case where $\abs{\qa{a_1}-\qa{a_2}}>(\log W)^{3/2}$.
    In this case, if \smash{$\absa{\qa{a_1}-\qa{a_2}}>\pa{\log W}^{5/4}$}, then there exists a constant $c>0$ such that for any large constant $D>0$,
    \begin{equation*}
        \begin{aligned}            \cK_{\fB,\bsigma,\ba}^{\pa{2}}&=\sum_{x\in\qa{a_1},y\in\qa{a_2}}m\p{\sigma_1}m\p{\sigma_2}\pa{1-m\p{\sigma_1}m\p{\sigma_2}S}^{-1}_{x_1x_2}\fB^{\pa{1}}_{x_1x_1}\fB^{\pa{2}}_{x_2x_2}\\
&=m\p{\sigma_1}m\p{\sigma_2}\bv\p{\fB_{\qa{a_1}}^{\pa{1}}}^{\top}\pa{1-m\p{\sigma_1}m\p{\sigma_2}S}^{-1}\bv\p{\fB_{\qa{a_2}}^{\pa{2}}}=\OO\pa{W^{-d} e^{-c\absa{\qa{a_1}-\qa{a_2}}}}=\OO\pa{W^{-D}},
        \end{aligned}
    \end{equation*}
    where the first equality follows from the representation of the $2$-$\cK$-loop in \eqref{eq:2Kloop}. In the second step, \smash{$\bv(\fB^{(i)}_{\qa{a_i}})$} ($i\in\{1,2\}$) denotes the row vector formed by the diagonal entries of \smash{$\fB^{(i)}_{\qa{a_i}}$}, and the last step follows from the decay property \eqref{decay_of_Theta_global}. Together with \eqref{decay_global_law_assumed_control}, this yields the $2$-loop bound
    \begin{equation}\label{eq;2Loopbound}
            \cL_{\fB,\bsigma,\ba}^{\pa{2}}\prec \cJ_D^*\cdot \cT_D\pa{\absa{\qa{a_1}-\qa{a_2}}},\quad \text{for}\quad  \absa{\qa{a_1}-\qa{a_2}}>\pa{\log W}^{5/4}.
    \end{equation}
    Using \eqref{entrywise_estimate_1_from_2_loop_to_1_chain} and
\eqref{entrywise_estimate_2_from_2_loop_to_1_chain} below, this further implies that for $\absa{\qa{a_1}-\qa{a_2}}\gtrsim \pa{\log W}^{3/2}$,    \begin{equation}\label{iterative_deday_single_resolvent_entrywise_global_law}
        \begin{aligned}
            \max_{i\in\qa{a_1},j\in\qa{a_2}}\abs{\pa{G-m}_{ij}}^2\prec \cJ_D^*\cdot \cT_D\pa{\absa{\qa{a_1}-\qa{a_2}}},
        \end{aligned}
    \end{equation}
    where we also used that $\cT_D\pa{\ell-C}\prec \cT_D\pa{\ell}$ for any constant $C>0$ and $\ell\geq \log W$. With an argument analogous to those in
\eqref{cumulant_expansion_global_recursive} and
\eqref{first_expansion_two_resolvent_global_law}, we obtain
    \begin{align}    \avga{G\p{\sigma_1}\fB^{\pa{1}}_{\qa{a_1}}G\p{\sigma_2}\fB^{\pa{2}}_{\qa{a_2}}}        &=m\p{\sigma_1}\avga{\fB^{\pa{1}}_{\qa{a_1}}G\p{\sigma_2}\wt \fB_{2}}+m\p{\sigma_1}\sum_{i,j}S_{ij}\pa{G\p{\sigma_1}-m\p{\sigma_1}}_{jj}\pa{G\p{\sigma_1}\fB^{\pa{1}}_{\qa{a_1}}G\p{\sigma_2}\wt \fB_2}_{ii}\nonumber\\
&+m\p{\sigma_1}\sum_{i,j}S_{ij}\pa{G\p{\sigma_1}\fB^{\pa{1}}_{\qa{a_1}}G\p{\sigma_2}}_{jj}\pa{\pa{G\p{\sigma_2}-m\p{\sigma_2}}\wt \fB_2}_{ii}\nonumber\\
&-m\p{\sigma_1}\avga{\ul{HG\p{\sigma_1}\fB^{\pa{1}}_{\qa{a_1}}G\p{\sigma_2}\wt \fB_2}},\label{expansion_for_getting_decay}
\end{align}
where $\wt \fB_2$ is defined below \eqref{first_expansion_two_resolvent_global_law} and satisfies the decay property \eqref{decay_of_Theta_global}. Using this decay property, together with the averaged global law
\eqref{global_law_average} and an argument analogous to
\eqref{eq;expected2first}, we obtain
    \begin{align}
m\p{\sigma_1}\avga{\fB^{\pa{1}}_{\qa{a_1}}G\p{\sigma_2}\wt \fB_2}-\cK_{\bsigma}^{\pa{2}}\pa{\fB^{\pa{1}}_{\qa{a_1}},\fB^{\pa{2}}_{\qa{a_2}}}&\prec \|\wt \fB_{2}\fB^{\pa{1}}_{\qa{a_1}}\|\prec  W^{-2d} e^{-c\absa{\qa{a_1}-\qa{a_2}}}.\label{expansion_for_getting_decay_leading_term}
\end{align}
It therefore remains to bound the second through fourth terms on the RHS of \eqref{expansion_for_getting_decay}.

For the second term on the RHS of \eqref{expansion_for_getting_decay}, under the condition $\abs{\qa{a_1}-\qa{a_2}}>(\log W)^{3/2}$, we bound it as follows, for any large constant $D>0$:
        \begin{align}
            &~\sum_{i,j}S_{ij}\pa{G\p{\sigma_1}-m\p{\sigma_1}}_{jj}\pa{G\p{\sigma_1}\fB^{\pa{1}}_{\qa{a_1}}G\p{\sigma_2}\wt \fB_2}_{ii} \nonumber\\
            =&~\sum_{\qa{x}:|[x]-[a_2]|\le (\log W)^{5/4}}\sum_{\qa{y}:\absa{\qa{x}-\qa{y}}\leq 2C_0}\sum_{j\in \qa{y}}\pa{G\p{\sigma_1}-m\p{\sigma_1}}_{jj}\avga{G\p{\sigma_1}\fB^{\pa{1}}_{\qa{a_1}}G\p{\sigma_2}\wt \fB_2I_{\qa{x}}S^{\pa{j}}} \nonumber\\
            \prec&~ \sum_{\qa{x}:|[x]-[a_2]|\le (\log W)^{5/4}}W^{-d/2}\cdot \qa{\cJ_D^*\cdot \cT_D\pa{\absa{\qa{a_1}-\qa{x}}}}\cdot e^{-c\absa{\qa{x}-\qa{a_2}}}+W^{-D-10d} \nonumber\\
            \prec &~ {W^{-d/2}}\cdot \qa{\cJ_D^*\cdot \cT_D\pa{\absa{\qa{a_1}-\qa{a_2}}}}+W^{-D-10d}\lesssim W^{-d/2}\cdot \qa{\cJ_D^*\cdot \cT_D\pa{\absa{\qa{a_1}-\qa{a_2}}}}.\label{cutoff_argument_example}
        \end{align}
In the first step, we use the $C_0$-flatness condition \eqref{def_C_flat}. In the second step, we apply the entrywise global law \eqref{entrywise_estimate_2_from_2_loop_to_1_chain} below to $(G(\sigma_1)-m(\sigma_1))_{jj}$, together with the decay property \eqref{decay_of_Theta_global} for \smash{$\wt\fB_2$} and the $2$-loop bound \eqref{eq;2Loopbound}. In the third step, we use the fact that \be\label{eq:slowTdecay}\cT_D\p{\ell-\pa{\log W}^{5/4}}\prec \cT_D\pa{\ell} \quad \text{for}\quad \ell\gtrsim \pa{\log W}^{3/2}.
\ee
The third term on the RHS of \eqref{expansion_for_getting_decay} can be bounded in exactly the same manner.
    Finally, it remains to control the underlined term
    \[\cQ:=\avga{\ul{HG\p{\sigma_1}\fB^{\pa{1}}_{\qa{a_1}}G\p{\sigma_2}\wt \fB_2}}.\]
    Using the identity $HG\p{z}=I+zG\p{z}$, we rewrite $\cQ$ as
    \begin{equation*}
            \cQ=\avga{\fB^{\pa{1}}_{\qa{a_1}}G\p{\sigma_2}\wt \fB_2}+z^{\sigma_1}\avga{G\p{\sigma_1}\fB^{\pa{1}}_{\qa{a_1}}G\p{\sigma_2}\wt \fB_2}-\sum_{x,y}S_{xy}\partial_{yx}\pa{G\p{\sigma_1}\fB^{\pa{1}}_{\qa{a_1}}G\p{\sigma_2}\wt \fB_2}_{yx}.
    \end{equation*}
    We now bound the high moments of $\cQ$ using Gaussian integration by parts, following an argument analogous to that below \eqref{cumulant_expansion_global_loop}. This yields that for any $p\in\N$,
    \begin{equation}     \label{eq:EQ2p}       \E\absa{\cQ}^{2p}=\E\sum_{i,j}S_{ij}\pa{G\p{\sigma_1}\fB^{\pa{1}}_{\qa{a_1}}G\p{\sigma_2}\wt \fB_2}_{ji}\pa{p\cdot \pa{\partial_{ji}\ol\cQ}\cdot\absa{\cQ}^{2p-2}+\pa{p-1}\cdot \pa{\partial_{ji}\cQ}\cdot\ol{\cQ}^{p}\cQ^{p-2}}.
    \end{equation}

    For simplicity of presentation, we bound only the most complicated contribution in \eqref{eq:EQ2p} as a representative example, namely
    \begin{equation*}
        \begin{aligned}
        \cal E:= \pa{p-1}\sum_{i,j}\sum_{x,y}S_{ij}S_{xy}\pa{G\p{\sigma_1}\fB^{\pa{1}}_{\qa{a_1}}G\p{\sigma_2}\wt \fB_2}_{ji}\partial_{ji}\partial_{yx}\pa{G\p{\sigma_1}\fB^{\pa{1}}_{\qa{a_1}}G\p{\sigma_2}\wt \fB_2}_{yx}\cdot\ol{\cQ}^{p}\cQ^{p-2}.
        \end{aligned}
    \end{equation*}
    All remaining terms in \eqref{eq:EQ2p} can be treated in an entirely analogous manner. Fix any large constant $D>0$. We estimate $\cal E$ as follows:
   \begin{align*}
    \cal E \lesssim &\sum_{\qa{a},\qa{b}}\sum_{i,j\in\ppp{a}}\sum_{x,y\in\ppp{b}}\sum_{\alpha,\beta\in\qa{a_1}}S_{ij}S_{xy}\fB^{\pa{1}}_{\alpha\alpha}\fB^{\pa{1}}_{\beta\beta}\p{\wt \fB_2}_{ii}\p{\wt \fB_2}_{xx}  G\p{\sigma_1}_{j\alpha}G\p{\sigma_2}_{\alpha i}\partial_{ji}\partial_{yx}\pa{G\p{\sigma_1}_{y\beta}G\p{\sigma_2}_{\beta x}}\cdot\ol{\cQ}^{p}\cQ^{p-2}\\
    \prec &\sum_{\qa{a},\qa{b}}^{\star}\sum_{i,j\in\ppp{a}}\sum_{x,y\in\ppp{b}}\sum_{\alpha,\beta\in\qa{a_1}} \frac{e^{-c|[a]-[a_2]|-c|[b]-[a_2]|}}{W^{6d}} G\p{\sigma_1}_{j\alpha}G\p{\sigma_2}_{\alpha i}\partial_{ji}\partial_{yx}\pa{G\p{\sigma_1}_{y\beta}G\p{\sigma_2}_{\beta x}}\cdot\ol{\cQ}^{p}\cQ^{p-2} + W^{-2pD},
    \end{align*}
    where we have used the decay property \eqref{decay_of_Theta_global} for \smash{$\wt\fB_2$}. Here we abbreviate
    \[\ppp{a}:=\bigcup_{\absa{\qa{a'}-\qa{a}}\leq 2C_0}\qa{a'}, \quad \ppp{b}:=\bigcup_{\absa{\qa{b'}-\qa{b}}\leq 2C_0}\qa{b'},\quad  \sum^\star_{[a],[b]}:=\sum_{|[a]-[a_2]|\vee|[b]-[a_2]| \le (\log W)^{5/4}}.\]
    For notational clarity, we refer to $i,j$ as the $\qa{a}$–indices, $x,y$ as the $\qa{b}$–indices, and $\alpha,\beta$ as the $\qa{a_1}$–indices. Owing to the exponential decay factor, we also regard $i,j,x,y$ as $\qa{a_2}$–indices. We call a resolvent entry $G(\sigma_*)_{\#_1\#_2}$ a decaying $G$–factor if $\#_1,\#_2$ consist of one $\qa{a_1}$–index and one $\qa{a_2}$–index, and an interacting $G$–factor if $\#_1,\#_2$ consist of one $\qa{a}$–index and one $\qa{b}$–index. With this terminology, expanding the derivatives produces a linear combination of terms, each containing four decaying $G$–factors and at least one interacting $G$–factor. The decaying $G$-factors are bounded using \eqref{iterative_deday_single_resolvent_entrywise_global_law} and \eqref{eq:slowTdecay},
    \begin{equation*}
    |G\p{\sigma_*}_{\#_1\#_2}|^2\prec \cJ_D^*\cdot \cT_D\pa{\absa{\qa{a_1}-\qa{a_2}}},
    \end{equation*}
    while the interacting $G$-factors are estimated by
    \begin{equation*}            G\p{\sigma_*}_{\#_1\#_2}=m\p{\sigma_*}\delta_{\#_1\#_2}+\OO_{\prec}\p{W^{-d/2}}
    \end{equation*}
    using \eqref{entrywise_estimate_2_from_2_loop_to_1_chain} below.
    Combining these bounds yields
    \begin{equation*}
        \begin{aligned}
            \absa{\cal E}\prec&~ \frac{\pa{\cJ_D^*}^2\qa{ \cT_D\pa{\absa{\qa{a_1}-\qa{a_2}}}}^2}{W^{6d}}\sum_{\qa{a},\qa{b}}^\star \sum_{i,j\in\ppp{a}}\sum_{x,y\in\ppp{b}}\sum_{\alpha,\beta\in\qa{a_1}}\sum_{\#_1\in\ha{i,j}}\sum_{\#_2\in\ha{x,y}}\pa{\delta_{\#_1\#_2}+W^{-d/2}} \\
            &~\times e^{-c\absa{\qa{a}-\qa{a_2}}-c\absa{\qa{b}-\qa{a_2}}}\cdot\absa{\cQ}^{2p-2}+W^{-2pD} \\
            \lesssim&~ \frac{\pa{\cJ_D^*}^2\qa{ \cT_D\pa{\absa{\qa{a_1}-\qa{a_2}}}}^2}{W^{d/2}} \sum_{\qa{a},\qa{b}}^\star e^{-c\absa{\qa{a}-\qa{a_2}}-c\absa{\qa{b}-\qa{a_2}}}\cdot\absa{\cQ}^{2p-2}\lesssim \frac{\pa{\cJ_D^*}^2}{W^{d/2}}\qa{\cT_D\pa{\absa{\qa{a_1}-\qa{a_2}}}}^2\cdot\absa{\cQ}^{2p-2}.
        \end{aligned}
    \end{equation*}
    All other terms arising from \eqref{eq:EQ2p} satisfy the same bound. Consequently, we obtain
    \begin{equation}     \label{eq:EQ2p2}       \E\absa{\cQ}^{2p}\prec W^{-d/2} \cdot \pa{\cJ_D^*}^2 \qa{\cT_D\pa{\absa{\qa{a_1}-\qa{a_2}}}}^2\cdot\E\absa{\cQ}^{2p-2}.
    \end{equation}
     Applying Hölder’s and Markov’s inequalities to \eqref{eq:EQ2p2}, we obtain
    \begin{equation}     \label{eq:EQ2p3}
    \cQ\prec {W^{-d/4}\cdot\cJ_D^*\cT_D\pa{\absa{\qa{a_1}-\qa{a_2}}}}.\end{equation}

    Combining \eqref{eq:EQ2p3} with \eqref{expansion_for_getting_decay_leading_term} and \eqref{cutoff_argument_example}, we conclude under the assumption \eqref{decay_global_law_assumed_control} that
    \begin{equation*}
        \begin{aligned}
            \absa{\cL_{\fB,\bsigma,\ba}^{\pa{2}}-\cK_{\fB,\bsigma,\ba}^{\pa{2}}}\Big/\cT_{D}\pa{\absa{\qa{a_1}-\qa{a_2}}}\prec 1+W^{-d/4}\cJ_D^* \quad \text{for}\quad |[a_1]-[a_2]|>(\log W)^{3/2}.
        \end{aligned}
    \end{equation*}
    Recall that the same estimate also holds for $|[a_1]-[a_2]|\le (\log W)^{3/2}$ by \eqref{eq:l2-k2}. Starting from the trivial bound $\cJ_D^*=W^{D}$ and iterating the above self-improving estimate $\OO(1)$ times, we arrive at
    \begin{equation}
        \begin{aligned}
            \max_{\bsigma\in\ha{+,-}^2}\max_{\ba\in\p{\wt \Z_n^d}^2}\absa{\cL_{\fB,\bsigma,\ba}^{\pa{2}}-\cK_{\fB,\bsigma,\ba}^{\pa{2}}}\Big/\cT_{D}\pa{\absa{\qa{a_1}-\qa{a_2}}}\prec 1.
        \end{aligned}
    \end{equation}
    This establishes \eqref{global_law_with_decay} and thus completes the proof of \Cref{improved_global_laws}.

\subsection{Resolvent entry estimates}

Our analysis above relies critically on the following lemma, which bounds resolvent entries via estimates on 2-$G$-loops.

\begin{lemma}\label{from_2_loop_to_1_chain}
In the setting of \Cref{def:general_global}, fix small constants $c,\kappa>0$ and a spectral parameter $z=E+\ii\eta$ with $|E|\le 2-\kappa$.
Define the event
    \begin{equation}
        \begin{aligned}
            \Omega\pa{c,z}:=\ha{\norma{G\p{z}-m\p{z}}_{\max}\leq W^{-c}},
        \end{aligned}
    \end{equation}
and abbreviate $\cL_{\bsig,\ba}^{\pa{2}}\equiv \cL_{\fB,\bsig,\ba}^{\pa{2}}$ for $\fB=(I,I)$. Suppose that the following bound on $2$-loops at $z$ holds:
    \begin{equation}
        \begin{aligned}
        \max_{\bsig\in\{(+,-),(-,+)\}}\cL_{\bsig,\pa{\qa{a},\qa{b}}}^{\pa{2}}(z)\prec\Psi^2\p{\qa{a},\qa{b}},
        \end{aligned}
    \end{equation}
    where $\Psi\p{\qa{a},\qa{b}}$ are deterministic control parameters satisfying $W^{-d}\le \Psi^2:=\max_{\qa{a},\qa{b}}\Psi^2\pa{\qa{a},\qa{b}}\leq W^{-2c}$. Then, for any \smash{$\qa{a},\qa{b}\in\wt \Z_n^d$} with $\qa{a}\neq \qa{b}$, the following entrywise resolvent estimates hold uniformly in $z$:
    \begin{align}        \mathbf{1}_{\Omega\pa{c,z}}\max_{i\in\qa{a},j\in\qa{b}}\abs{\pa{G\p{z}-m\p{z}}_{ij}}^2\prec \sum_{\absa{\qa{a'}-\qa{a}}\leq C_0} \sum_{\absa{\qa{b'}-\qa{b}}\leq C_0} \Psi^2\pa{\q{a'},\q{b'}}+\frac{1}{W^d}\mathbf{1}_{\absa{\qa{a}-\qa{b}}\leq C_0}\label{entrywise_estimate_1_from_2_loop_to_1_chain},
    \end{align}
    \begin{equation}\label{entrywise_estimate_2_from_2_loop_to_1_chain}
        \begin{aligned}
            \mathbf{1}_{\Omega\pa{c,z}}\max_{i,j\in\qa{a}}\abs{\pa{G\p{z}-m\p{z}}_{ij}}^2\prec \Psi^2.
        \end{aligned}
    \end{equation}
    Consequently, if $\norm{G\pa{z}-m\pa{z}}_{\max}\prec W^{-c}$ for some constant $c>0$, then the bounds \eqref{entrywise_estimate_1_from_2_loop_to_1_chain} and \eqref{entrywise_estimate_2_from_2_loop_to_1_chain} hold without the indicator function. Moreover, for any deterministic diagonal matrix $B$ with $\norm{B}=\OO\pa{1}$, the following averaged resolvent estimate holds:
    \begin{equation}\label{average_estimate_from_2_loop_to_1_chain}
            \frac{1}{W^d}\max_{[a]}\absa{\avga{\pa{G\p{z}-m\p{z}}I_{\qa{a}}BI_{\qa{a}}}}\prec \Psi^2.
    \end{equation}

\end{lemma}

\begin{proof}
The proofs of the estimates \eqref{entrywise_estimate_1_from_2_loop_to_1_chain} and \eqref{entrywise_estimate_2_from_2_loop_to_1_chain} are identical to those of (4.2) and (4.3) in \cite{Band1D}, and are therefore omitted.
It remains to prove the averaged estimate \eqref{average_estimate_from_2_loop_to_1_chain}, under the assumption that $\norm{G\pa{z}-m\pa{z}}_{\max}\prec W^{-c}$.
We first claim that the $T$-variables, defined by $T_{xy}=\pa{GE_{\qa{x}}G^*}_{yy}$, satisfy the bound
    \begin{equation}\label{estimate_from_2_loop_to_2_chain}
        \begin{aligned}
            \max_{x,y\in \Zn} T_{xy}\prec \Psi^2.
        \end{aligned}
    \end{equation}
This can be proved by essentially the same argument as that following equation (6.47) in \cite{truong2025localizationlengthfinitevolumerandom}, where the $T$-variables are bounded in terms of $2$-$G$-loops. We therefore omit the details.
Next, we estimate the matrix self-consistent equation in an averaged sense using \eqref{estimate_from_2_loop_to_2_chain}. Recall that $\Pi(G)$ is defined in \eqref{eq:PiGdef}. Adopting an argument similar to the proof of \cite[Proposition 3.2]{He2018}, we can show that for any deterministic diagonal matrix $B$ with $\|B\|=\OO(1)$,

\begin{equation}\label{estimate_matrix_self_consistent_equation_average}
            \frac{1}{W^d}\max_{\qa{a}\in\wt \Z_n^d}\absa{\avga{\Pi\p{G}I_{\qa{a}}BI_{\qa{a}}}}=\opr{\Psi^2}.
    \end{equation}
More precisely, in the proof of \cite[Proposition 3.2]{He2018}, Ward’s identity is used at several points. In our setting, these terms can instead be bounded using Cauchy–Schwarz together with Gaussian integration by parts and the estimate \eqref{estimate_from_2_loop_to_2_chain}. In fact, our argument is considerably simpler than that of \cite{He2018}, since we only need to control terms arising from Gaussian integration by parts, whereas \cite{He2018} also treats additional terms coming from higher-order cumulant expansions. For this reason, we omit the details.

We are now ready to complete the proof of \eqref{average_estimate_from_2_loop_to_1_chain}. Suppose that the estimate
    \begin{equation}\label{eq:aver_theta}
            \max_{x\in\Z_L^d}\absa{\avg{\pa{G-m}S^{\pa{x}}}}\prec \theta
    \end{equation}
    holds for some control parameter $\theta>0$.
    Using the identity
    \begin{equation}\label{eq:G-mPiG}
            G-m=-m\Pi\pa{G}+m\cS\qa{G-m}G,
    \end{equation}
    together with the averaged estimate \eqref{estimate_matrix_self_consistent_equation_average}, we obtain
        \begin{align}
            \avg{\pa{G-m}S^{\pa{x}}}&=m\avg{\cS\qa{G-m}GS^{\pa{x}}}+\opr{\Psi^2}=m\sum_{y}\avg{\pa{G-m}S^{\pa{y}}}\p{GS^{\pa{x}}}_{yy}+\opr{\Psi^2} \nonumber\\
            &=m^2\sum_{y}\avg{\pa{G-m}S^{\pa{y}}}S_{xy}+\opr{\theta\Psi+\Psi^2}.\label{eq:aver_theta2}
        \end{align}
    In the last step, we used the assumption \eqref{eq:aver_theta} together with the entrywise resolvent estimate \eqref{entrywise_estimate_2_from_2_loop_to_1_chain}, applied to \smash{$\p{(G-m)S^{\pa{x}}}_{yy}=S_{xy}(G_{yy}-m)$}.
    Solving the linear equation \eqref{eq:aver_theta2} by applying the inverse of the matrix $1-m^2S$, we obtain
    \begin{equation}\label{eq;selftheta}
            \max_{x\in\Z_L^d}\absa{\avg{\pa{G-m}S^{\pa{x}}}}\prec \theta\Psi+\Psi^2.
    \end{equation}
    Here, we also use the fact that $\norm{\pa{1-m^2S}^{-1}}_{\infty\to\infty}=\OO\pa{1}$ for $z=E+\ii\eta$ with $|E|\le 2-\kappa$. (This bound follows from \eqref{prop:ThfadC_short} by choosing the spectral parameters appropriately so that $z_{\tf}=z$.)
    Iterating the self-improving estimate \eqref{eq;selftheta} for $\OO(1)$ steps, starting from the initial value $\theta=\Psi$, yields
    \begin{equation}\label{eq:averPsi2}
            \max_{x\in\Z_L^d}\absa{\avg{\pa{G\pa{z}-m\pa{z}}S^{\pa{x}}}}\prec \Psi^2.
    \end{equation}
    Finally, we treat the general case $B\neq I$. Without loss of generality, we may assume that $B=I_{\qa{a}}BI_{\qa{a}}$. Using \eqref{eq:G-mPiG} once more, we obtain
    \begin{equation}
        \begin{aligned}
            \avga{\pa{G-m}B}=m\avga{\cS\qa{G-m}GB}+\opr{\Psi^2}=\opr{\Psi^2},
        \end{aligned}
    \end{equation}
    where the first step follows from the averaged estimate \eqref{estimate_matrix_self_consistent_equation_average}, and the second step uses \eqref{eq:averPsi2}. This completes the proof of \eqref{average_estimate_from_2_loop_to_1_chain}.
\end{proof}

\section{Analysis of the loop hierarchy}\label{Sec:Stoflo}

We now return to the setting described at the beginning of \Cref{sec_preliminaries}, so that all definitions and results established there apply to the arguments below. More precisely, we consider a RBM model $H$ as in \Cref{def_considered_model}, with variance profile $\SRBM$. We assume that the conditions \eqref{def_epsilon_full} and \eqref{def_C_flat} hold with constants $\varepsilon_S=2\varepsilon_0$ and $C_S=C_0$, respectively. We further fix small constants $\kappa,\mathfrak{c}>0$ and a target spectral parameter $z=E+\ii\eta$ satisfying $\abs{E}\le 2-\kappa$ and $\eta\in[W^{\mathfrak{c}}\eta_*,\mathfrak{c}^{-1}]$.
Finally, we choose the deterministic flow according to \Cref{lem:paraselect}. We now state the main results of this section and then derive Theorems \ref{thm_locallaw} and \ref{thm_diffu} from them.

\begin{theorem}[$G$-loop estimates]\label{ML:GLoop}
For any fixed integer $\fn\geq 2$, consider $t\in\qa{\ti,\tf}$, $S_t\in t\fS_t$ (recall \eqref{def_variance_flow}), and the associated $G$-loops and $\cK$-loops defined in \Cref{Def:G_loop} and \Cref{Def_Ktza}. Then, for each $S_t\in t\fS_t$, the following estimates hold uniformly for $t\in [\ti,\tf]$ (recall $\ell_t$ defined in \eqref{eq:ellt}):
\be\label{Eq:L-KGt}
 \max_{\boldsymbol{\sigma}, \ba}\left|{\cL}^{(\fn)}_{t, \boldsymbol{\sigma}, \ba}-{\cal K}^{(\fn)}_{t, \boldsymbol{\sigma}, \ba}\right|\prec (W^d\ell_t^d\eta_t)^{-\fn} .
 \ee
Together with \eqref{eq:bcal_k}, this implies
\be
\max_{\boldsymbol{\sigma}, \ba}\left|{\cal L}^{(\fn)}_{t, \boldsymbol{\sigma}, \ba} \right|\prec (W^d\ell_t^d\eta_t)^{-\fn+1}. \label{Eq:L-KGt2}
\ee
\end{theorem}

\begin{theorem}[$2$-$G$-loop estimates]\label{ML:GLoop_expec}
In the setting of \Cref{ML:GLoop}, fix any $S_t\in t\fS_t$. The expectation of a $2$-$G$-loop satisfies the improved bound, uniformly for $t\in [\ti,\tf]$,
   \begin{equation}\label{Eq:Gtlp_exp}
 \max_{\boldsymbol{\sigma},\ba}\left|\mathbb E{\cal L}^{(2)}_{t, \boldsymbol{\sigma}, \ba}-{\cal K}^{(2)}_{t, \boldsymbol{\sigma}, \ba}\right|\prec (W^d\ell_t^d\eta_t)^{-3}
 .
\end{equation}
Moreover, for any $\boldsymbol{\sigma}=(+,-)$ and $ \ba=([a_1], [a_2])$, we have the following pointwise estimate, uniformly for $t\in [\ti,\tf]$: for any large constant $D>0$,
\begin{equation}\label{Eq:Gdecay}
 \left| {\cal L}^{(2)}_{t, \boldsymbol{\sigma}, \ba}-{\cal K}^{(2)}_{t, \boldsymbol{\sigma}, \ba}\right|\prec (W^d\ell_t^d\eta_t)^{-2}\exp \left(-\left|\frac{ [a_1]-[a_2] }{\ell_t}\right|^{1/2}\right)+W^{-D}.
\end{equation}
\end{theorem}

\begin{theorem}[Local law for $G_t$]\label{ML:GtLocal}
In the setting of \Cref{ML:GLoop}, fix any $S_t\in t\fS_t$. Then the following local laws hold uniformly for $t\in[\ti,\tf]$:
\begin{align}\label{Gt_bound}
 \|G_{t}-m\|_{\max} &\prec (W^d\ell_t^d\eta_t)^{-1/2},\\
 \label{Eq:Gt_1_lp_exp}
\max_{\qa{a}}\absa{\E\avga{\p{G_t-m}E_{\qa{a}}}}&\prec \pa{W^d\ell_t^d\eta_t}^{-2}.
\end{align}
\end{theorem}

The proofs of Theorems \ref{ML:GLoop}, \ref{ML:GLoop_expec}, and \ref{ML:GtLocal} are based on a detailed analysis of the $G$-loops along the flow. Since all of these results at the global time scale—namely when $1-t\sim 1$—have already been established in \Cref{sec_global_law} (see \Cref{lemma_global_law,improved_global_laws,from_2_loop_to_1_chain}), it suffices to prove the following theorem.

\begin{theorem}\label{lem:main_ind}

In the setting of \Cref{ML:GLoop}, suppose that the estimates \eqref{Eq:L-KGt}, \eqref{Eq:Gtlp_exp}, \eqref{Eq:Gdecay}, and \eqref{Gt_bound} hold at some fixed time $s\in [\ti,\tf]$ and for any $S_s\in s\fS_s$. More precisely, assume that the following statements hold.
\begin{itemize}
\item[(a)] {\bf $G$-loop estimate}: For each fixed integer $\fn\ge 2$,
\be\label{Eq:L-KGt+IND}
 \max_{\boldsymbol{\sigma}, \ba}\left|{\cal L}^{(\fn)}_{s, \boldsymbol{\sigma}, \ba}-{\cal K}^{(\fn)}_{s, \boldsymbol{\sigma}, \ba}\right|\prec (W^d\ell_s^d\eta_s)^{-\fn}.
\ee
\item[(b)] {\bf 2-$G$-loop estimate}:
For $\boldsymbol{\sigma}\in\{(+,-),(-,+)\}$ and $ \ba=([a_1],[a_2]),$ and for any large constant $D>0$,
\be\label{Eq:Gdecay+IND}
\left| {\cal L}^{(2)}_{s, \boldsymbol{\sigma}, \ba}-{\cal K}^{(2)}_{s, \boldsymbol{\sigma}, \ba}\right|\prec (W^d\ell_s^d\eta_s)^{-2}\exp \left(- \left|\frac{ [a_1]-[a_2] }{\ell_s}\right|^{1/2}\right)+W^{-D} .
\ee
\item[(c)] {\bf Local law}: The following local laws hold:
\begin{align} \label{Gt_bound+IND}
 \|G_{s}-m\|_{\max} &\prec (W^d\ell_s^d\eta_s)^{-1/2},\\
 \label{Eq:Gt_1_lp_exp+IND}
  \max_{\qa{a}}\absa{\E\avga{\p{G_s-m}E_{\qa{a}}}}&\prec \pa{W^d\ell_s^d\eta_s}^{-2}.
\end{align}

\item[(d)] {\bf Expected $2$-$G$-loop estimate}:
 \be \label{Eq:Gtlp_exp+IND}
 \max_{\boldsymbol{\sigma}, \ba}\left|\mathbb E{\cal L}^{(2)}_{s, \boldsymbol{\sigma}, \ba}-{\cal K}^{(2)}_{s, \boldsymbol{\sigma}, \ba}\right|\prec (W\ell_s\eta_s)^{-3}.
\ee
\end{itemize}
Then, for any $t\in [s,\tf]$ satisfying
\begin{equation}\label{con_st_ind}
(W^d\ell_t^d\eta_t)^{-\frac{1}{100}} \le  \frac{1-t}{1-s} \le \frac{1}{100},
\end{equation}
and for any $S_t\in t\fS_t$, the estimates \eqref{Eq:L-KGt}, \eqref{Eq:Gtlp_exp}, \eqref{Eq:Gdecay}, and \eqref{Gt_bound} also hold at time $t$.

\end{theorem}

By \Cref{lemma_variance_flow}, we may choose $S_{s}\in s\fS_s$ such that $S_t\pa{s,S_s}=S_t$ and $S_u\pa{s,S_s}\in u\fS_u$ for all $u\in\qa{s,t}$ along the flow defined in \eqref{def_S_t}. We then define the stochastic flow $H_u$ as in \eqref{def_stochastic_flow}, with initial time $t_0=s$ and initial variance profile $S_{s}$.
Throughout the remainder of this section, we consider the $G$-loops and $\cK$-loops associated with $H_u$ and its variance profile along this flow. With this setup, the proof of \Cref{lem:main_ind} is divided into the following six steps.

\medskip
\noindent
\textbf{Step 1} (A priori $G$-loop bound): We first show that $\fn$-$G$-loops satisfy the a priori bound
 \begin{equation}\label{lRB1}
   {\cal L}^{(\fn)}_{u,\boldsymbol{\sigma}, \ba}\prec (\ell_u^d/\ell_s^d)^{(\fn-1)}\cdot
   (W^d\ell_u^d\eta_u)^{-\fn+1},\quad  \forall s\le u\le t.
\end{equation}
Furthermore, the following weak local law holds:
\begin{equation}\label{Gtmwc}
    \|G_u-m\|_{\max}\prec  (W^d\ell_u^d\eta_u)^{-1/4},\quad \forall s\le u\le t .
\end{equation}

\medskip
\noindent
\textbf{Step 2} (Sharp local law and a priori $2$-$G$-loop estimate):
The following sharp local law holds:
\begin{equation}\label{Gt_bound_flow}
     \|G_u-M\|_{\max}\prec  (W^d \ell_u^d \eta_u)^{-1/2},\quad \forall s\le u\le t.
\end{equation}
In particular, the local law \eqref{Gt_bound} holds at time $t$.
In addition, for $\boldsymbol{\sigma}\in\{(+,-),(-,+)\}$, $ \ba=([a_1],[a_2]),$ and any large constant $D>0$, we have
\begin{equation}\label{Eq:Gdecay_w}
\left| {\cal L}^{(2)}_{u, \boldsymbol{\sigma}, \ba}-{\cal K}^{(2)}_{u, \boldsymbol{\sigma}, \ba}\right| \prec \left(\eta_s/\eta_u\right)^4\cdot (W^d\ell_u^d\eta_u)^{-2}\exp \left(- \left|\frac{ [a_1]-[a_2] }{\ell_u}\right|^{1/2}\right)+W^{-D} , \quad \forall s\le u \le t.
\end{equation}

   \medskip
 \noindent
\textbf{Step 3}  (Sharp $G$-loop bound): For each fixed $\fn\ge 2$, the following sharp bound on $\fn$-$G$-loops holds:
\begin{equation}\label{Eq:LGxb}
\max_{\boldsymbol{\sigma}, \ba}\left| {\cal L}^{(\fn)}_{u, \boldsymbol{\sigma}, \ba} \right|
\prec
  (W^d\ell_u^d\eta_u)^{-\fn+1} ,\quad \forall s\le u\le t .
\end{equation}

\medskip
 \noindent
\textbf{Step 4}  (Sharp $(\cL-\cK)$-loop limit): For each fixed $\fn\ge 2$, the following sharp estimate on $ \p{{\cal L}-{\cal K}}$-loops holds:
\begin{equation}\label{Eq:L-KGt-flow}
 \max_{\boldsymbol{\sigma}, \ba}\left|{\cal L}^{(\fn)}_{u, \boldsymbol{\sigma}, \ba}-{\cal K}^{(\fn)}_{u, \boldsymbol{\sigma}, \ba}\right|\prec (W^d\ell_u^d\eta_u)^{-\fn},\quad \forall s\le u\le t .
\end{equation}
Hence, the $G$-loop estimate \eqref{Eq:L-KGt} holds at time $t$.

\medskip
 \noindent
\textbf{Step 5}  (Sharp 2-$G$-loop estimate): For $\boldsymbol{\sigma}\in\{(+,-),(-,+)\}$ and $ \ba=([a_1],[a_2]),$ the following estimate holds for any large constant $D>0$:
\begin{equation}\label{Eq:Gdecay_flow}
\left| {\cal L}^{(2)}_{u, \boldsymbol{\sigma}, \ba}-{\cal K}^{(2)}_{u, \boldsymbol{\sigma}, \ba}\right| \prec   (W^d\ell_u^d\eta_u)^{-2}\exp \left(- \left|\frac{ [a_1]-[a_2] }{\ell_u}\right|^{1/2}\right)+W^{-D}.
\end{equation}
Consequently, the estimate \eqref{Eq:Gdecay} holds at time $t$.

\medskip
\noindent
\textbf{Step 6} (Expected 2-$G$-loop estimate):
For $\boldsymbol{\sigma}\in\{(+,-),(-,+)\}$, we have
\begin{equation}\label{Eq:Gtlp_exp_flow}
 \max_{\boldsymbol{\sigma}, \ba}\left|\mathbb E{\cal L}^{(2)}_{u, \boldsymbol{\sigma}, \ba}-{\cal K}^{(2)}_{u, \boldsymbol{\sigma}, \ba}\right|\prec (W^d\ell_u^d\eta_u)^{-3},\quad
 \forall s \le u \le t .
\end{equation}
In particular, the estimate \eqref{Eq:Gtlp_exp} holds at time $t$.

\medskip

We remark that all estimates obtained in the above steps hold uniformly in $u\in[s,t]$ (recall \Cref{stoch_domination}), as ensured by a standard $N^{-C}$-net argument. For simplicity of presentation, we will not emphasize this uniformity in the subsequent proofs.
Before proceeding to the proof of \Cref{lem:main_ind}, we first use these results to complete the proofs of \Cref{thm_locallaw} and \Cref{thm_diffu}.

\begin{proof}[Proofs of \Cref{thm_locallaw,thm_diffu}]
By the results established in \Cref{sec_global_law} (namely \Cref{lemma_global_law,improved_global_laws,from_2_loop_to_1_chain}), the estimates \eqref{Eq:L-KGt}, \eqref{Eq:Gtlp_exp}, \eqref{Eq:Gdecay}, and \eqref{Gt_bound} hold at time $\ti$. Applying \Cref{lem:main_ind}, we can then propagate these estimates along the flow up to time $\tf$.
By the definition \eqref{def_variance_flow}, we have \smash{$\SRBM\in\fS_{\tf}$}. Moreover,  by \eqref{equal_in_distribution_t_0}, we have \smash{$G\pa{z}\overset{\txt{d}}{=}\sqrt{\tf}G_{\tf}$}. Therefore, the estimates \eqref{Eq:L-KGt}, \eqref{Eq:Gtlp_exp}, and \eqref{Gt_bound} at time $\tf$ immediately yield the quantum diffusion estimates \eqref{eq:diffu1}--\eqref{eq:diffuExp2}, as well as the local laws \eqref{locallaw} and \eqref{locallaw_aver}, for any fixed $z\in \ha{z=E+\ii\eta: \, \absa{E}\leq 2-\kappa,W^{\fc}\eta_*\leq \eta\leq \fc^{-1}}$.
To extend these estimates uniformly to all such $z$, we invoke a standard $N^{-C}$-net argument, whose details we omit.
\end{proof}

The remainder of this section is devoted to the proof of \Cref{lem:main_ind}. The proof follows the strategy developed in \cite{Band1D,Band2D,truong2025localizationlengthfinitevolumerandom}. In fact, in view of the results already established, most of the arguments in these works extend to our setting without any modification. We therefore focus only on the key differences, omitting similar details.

\subsection{Step 1: A priori $G$-loop bound}

The proofs of \eqref{lRB1} and \eqref{Gtmwc} rely on \Cref{from_2_loop_to_1_chain} (after an appropriate rescaling of the random matrix) together with the following continuity estimate for $G$-loops.

\begin{lemma}[Continuity estimate for $G$-loops]\label{lem_ConArg}
Let \( \ti\le t_1 \leq t_2 \leq \tf\). Suppose that, at time \(t_1\), the following bound holds for each $S_{t_1}\in t_1\fS_{t_1}$ and fixed \(\fn \in \N\):
\begin{equation}\label{55}
    \max_{\boldsymbol{\sigma}, \ba} \absa{{\cal L}^{(\fn)}_{t_1, \boldsymbol{\sigma}, \ba}} \prec \left( W^d \ell_{t_1}^d \eta_{t_1} \right)^{-\fn+1}.
\end{equation}
Then, on the event $\Omega:= \left\{\|G_{t_2}\|_{\max} \leq C\right\}$ for some constant $C>0$, the following estimate holds at time $t_2$ for each $S_{t_2}\in t_2\fS_{t_2}$ and fixed \(\fn \in \N\):
\begin{equation}\label{res_lo_bo_eta}
    {\bf 1}_\Omega \cdot \absa{\max_{\boldsymbol{\sigma}, \ba} {\cal L}_{t_2, \boldsymbol{\sigma}, \ba}^{(\fn)}} \prec \left( W^d \ell_{t_1}^d \eta_{t_2} \right)^{-\fn+1} = \left({\ell_{t_2}^d}/{\ell_{t_1}^d}\right)^{\fn-1}\left( W^d \ell_{t_2}^d \eta_{t_2} \right)^{-\fn+1}.
\end{equation}
\end{lemma}
\begin{proof}
The proof follows the argument of Lemma 5.1 in \cite{Band1D}. The only difference is that, in our setting, $S_t$ is no longer proportional to $t$, so the scaling relation \smash{$H_{t_2}\overset{\txt{d}}{=}\sqrt{t_2/t_1}\cdot H_{t_1}$} used in \cite{Band1D} does not apply directly. This issue can be resolved by observing that, for any $S_{t_2}\in t_2\fS_{t_2}$, we have $t_1/t_2\cdot S_{t_2}\in t_1\fS_{t_2}\subseteq t_1\fS_{t_1}$.
Thus, given an RBM $H_{t_2}$ with variance profile $S_{t_2}\in t_2\fS_{t_2}$, we may choose $H_{t_1}$ to be an RBM with variance profile $t_1/t_2\cdot S_{t_2}\in t_1\fS_{t_1}$. Applying the bound \eqref{55} at time $t_1$ and then repeating the argument of Lemma 5.1 in \cite{Band1D} yields the desired estimate.
\end{proof}

With \Cref{from_2_loop_to_1_chain} and \Cref{lem_ConArg} at hand, the bound \eqref{lRB1} follows directly from \eqref{Eq:L-KGt+IND}, \eqref{eq:bcal_k}, and \eqref{55}. Moreover, the weak local law \eqref{Gtmwc} is an immediate consequence of \eqref{lRB1} together with \eqref{entrywise_estimate_2_from_2_loop_to_1_chain}, by the same argument as in Section 5.1 of \cite{Band1D}. We therefore omit the details.

\subsection{Steps 2--5: Proofs of \eqref{Gt_bound_flow}--\eqref{Eq:Gdecay_flow}}

The proofs of Steps 2–5 rely on an analysis of the loop hierarchy along the flow, using approximation by the primitive loops introduced in \Cref{Def_Ktza}.
Fix any $\fn\in \N$. Combining the loop hierarchy \eqref{eq:mainStoflow} with the evolution equation \eqref{pro_dyncalK} for the primitive loops, we obtain
\begin{align}\label{eq_L-K-1}
       \dd(\mathcal{L} - \mathcal{K})^{(\fn)}_{t, \boldsymbol{\sigma}, \ba}
    =&~ W^d \sum_{1 \leq k < l \leq \fn} \sum_{[a]}
   (\mathcal{L} - \mathcal{K})^{(\fn+k-l+1)}_{t, \cutL^{[a]}_{k, l}\left(\boldsymbol{\sigma},\, \ba\right)}
   \mathcal{K}^{(l-k+1)}_{t,\cutR^{[a]}_{k,l}\left(\boldsymbol{\sigma} ,\ba\right)}\, \dd t \nonumber\\
    +&~ W^d \sum_{1 \leq k < l \leq \fn} \sum_{[a]} \mathcal{K}^{(\fn+k-l+1)}_{t, \cutL^{[a]}_{k, \,l}\left(\boldsymbol{\sigma},\ba\right)}
   \left(\cL-\cK\right)^{(l-k+1)}_{t,\cutR^{[a]}_{k, l}\left(\boldsymbol{\sigma},\ba\right)} \, \dd t \nonumber\\
   +&~ \mathcal{E}^{(\fn)}_{t, \boldsymbol{\sigma}, \ba}\dd t +
    \dd\mathcal{B}^{(\fn)}_{t, \boldsymbol{\sigma}, \ba}
    +\mathcal{W}^{(\fn)}_{t, \boldsymbol{\sigma}, \ba}
    \dd t,
\end{align}
where $\mathcal{E}^{(\fn)}_{t, \boldsymbol{\sigma}, \ba}$ is defined by
\begin{equation}\label{def_ELKLK}
\mathcal{E}^{(\fn)}_{t, \boldsymbol{\sigma}, \ba} :=
    W^d \sum_{1 \leq k < l \leq \fn} \sum_{[a]}
   (\mathcal{L} - \mathcal{K})^{(\fn+k-l+1)}_{t, \cutL^{[a]}_{k, l}\left(\boldsymbol{\sigma},\, \ba\right)}
   (\cL-\mathcal{K})^{(l-k+1)}_{t,\cutR^{[a]}_{k,l}\left(\boldsymbol{\sigma} ,\ba\right)}\,  .
\end{equation}
We rearrange the first two terms on the RHS of \eqref{eq_L-K-1} according to the length of the $\cK$-loops. This yields
$$\sum_{\lenk=2}^\fn \left[\OK^{(\lenk)} (\mathcal{L} - \mathcal{K})\right]^{(\fn)}_{t, \boldsymbol{\sigma}, \ba}\dd t,$$
where $\OK^{(\lenk)}$ is a linear operator defined as
\begin{align}\label{DefKsimLK}
\left[\OK^{(\lenk)} (\mathcal{L} - \mathcal{K})\right]_{t, \boldsymbol{\sigma}, \ba}^{(\fn)}:=  &~W^d \sum_{1\le k < l \leq \fn : l-k=\lenk-1} \sum_{[a]}
   (\mathcal{L} - \mathcal{K})^{(\fn-\lenk+2)}_{t, \cutL^{[a]}_{k, l}\left(\boldsymbol{\sigma},\, \ba\right)}
   \mathcal{K}^{(\lenk)}_{t,\cutR^{[a]}_{k,l}\left(\boldsymbol{\sigma} ,\ba\right)} \nonumber\\
    +&~  W^d \sum_{1 \leq k < l \leq \fn:l-k=\fn-\lenk+1} \sum_{[a]} \mathcal{K}^{(\lenk)}_{t, \cutL^{[a]}_{k, l}\left(\boldsymbol{\sigma},\ba\right)}
   \left(\cL-\cK\right)^{(\fn-\lenk+2)}_{t,\cutR^{[a]}_{k, l}\left(\boldsymbol{\sigma},\ba\right)}  .
\end{align}
Extracting the leading term corresponding to $\lenk =2$, and observing that the operator $\OK^{\pa{2}}$ coincides with the operator \smash{$\vartheta_{t,\bsigma}^{\pa{\fn}}$} defined in \eqref{def:op_thn}, we may rewrite \eqref{eq_L-K-1} as
\begin{align}\label{eq_L-Keee}
    \dd(\mathcal{L} - \mathcal{K})^{(\fn)}_{t, \boldsymbol{\sigma}, \ba} = &~\left[\vartheta_{t,\bsigma}^{\pa{\fn}}\circ (\mathcal{L} - \mathcal{K})\right]^{(\fn)}_{\ba} \, \dd t+\sum_{\lenk=3}^\fn \left[\OK^{(\lenk)} (\mathcal{L} - \mathcal{K})\right]^{(\fn)}_{t, \boldsymbol{\sigma}, \ba}\, \dd t \nonumber\\
    +&~ \mathcal{E}^{(\fn)}_{t, \boldsymbol{\sigma}, \ba}\dd t +
    \dd\mathcal{B}^{(\fn)}_{t, \boldsymbol{\sigma}, \ba}
    +
    \mathcal{W}^{(\fn)}_{t, \boldsymbol{\sigma}, \ba}\dd t.
\end{align}
Applying Duhamel’s principle to \eqref{eq_L-Keee}, we obtain for any $s\le t$:
\begin{align}\label{int_K-LcalE}
    (\mathcal{L} - \mathcal{K})^{(\fn)}_{t, \boldsymbol{\sigma}, \ba} & =
    \left(\mathcal{U}^{(\fn)}_{s, t, \boldsymbol{\sigma}} \circ (\mathcal{L} - \mathcal{K})^{(\fn)}_{s, \boldsymbol{\sigma}}\right)_{\ba} + \sum_{l_\mathcal{K} =3}^\fn \int_{s}^t \left(\mathcal{U}^{(\fn)}_{u, t, \boldsymbol{\sigma}} \circ \Big[\OK^{(\lenk)} (\mathcal{L} - \mathcal{K})\Big]^{(\fn)}_{u, \boldsymbol{\sigma}}\right)_{\ba} \dd u \nonumber \\
    &+ \int_{s}^t \left(\mathcal{U}^{(\fn)}_{u, t, \boldsymbol{\sigma}} \circ \mathcal{E}^{(\fn)}_{u, \boldsymbol{\sigma}}\right)_{\ba} \dd u  + \int_{s}^t \left(\mathcal{U}^{(\fn)}_{u, t, \boldsymbol{\sigma}} \circ \cW^{(\fn)}_{u, \boldsymbol{\sigma}}\right)_{\ba} \dd u + \int_{s}^t \left(\mathcal{U}^{(\fn)}_{u, t, \boldsymbol{\sigma}} \circ \dd \cB^{(\fn)}_{u, \boldsymbol{\sigma}}\right)_{\ba} ,
\end{align}
where the evolution kernel $\mathcal{U}^{(\fn)}_{s, t, \boldsymbol{\sigma}}$ is defined in \eqref{def_Ustz}.
Furthermore, let $T$ be a stopping time with respect to the matrix Brownian motion $\{H_t\}$, and set $\tau:= T\wedge t$. Then the stopped version of \eqref{int_K-LcalE} reads
\begin{align}\label{int_K-L_ST}
(\mathcal{L} - \mathcal{K})^{(\fn)}_{\tau, \boldsymbol{\sigma}, \ba} & =
    \left(\mathcal{U}^{(\fn)}_{s, \tau, \boldsymbol{\sigma}} \circ (\mathcal{L} - \mathcal{K})^{(\fn)}_{s, \boldsymbol{\sigma}}\right)_{\ba} + \sum_{l_\mathcal{K} =3}^\fn \int_{s}^\tau \left(\mathcal{U}^{(\fn)}_{u, \tau, \boldsymbol{\sigma}} \circ \Big[\OK^{(\lenk)} (\mathcal{L} - \mathcal{K})\Big]^{(\fn)}_{u, \boldsymbol{\sigma}}\right)_{\ba} \dd u \nonumber \\
    &+ \int_{s}^\tau \left(\mathcal{U}^{(\fn)}_{u, \tau, \boldsymbol{\sigma}} \circ \mathcal{E}^{(\fn)}_{u, \boldsymbol{\sigma}}\right)_{\ba} \dd u  + \int_{s}^\tau \left(\mathcal{U}^{(\fn)}_{u, \tau, \boldsymbol{\sigma}} \circ \cW^{(\fn)}_{u, \boldsymbol{\sigma}}\right)_{\ba} \dd u + \int_{s}^\tau \left(\mathcal{U}^{(\fn)}_{u, \tau, \boldsymbol{\sigma}} \circ \dd \cB^{(\fn)}_{u, \boldsymbol{\sigma}}\right)_{\ba} .
\end{align}

With the preparations above—namely, the deterministic estimates from Sections \ref{subsec:propagators} and \ref{sec_analysis_of_primitive_loops}, Lemma \ref{from_2_loop_to_1_chain}, estimates \eqref{lRB1} and \eqref{Gtmwc}, together with the integral equations \eqref{int_K-LcalE} and \eqref{int_K-L_ST}—we are now in a position to carry out the arguments for Steps 2–5 exactly as in \cite[Section 5]{Band1D} or \cite[Section 7]{truong2025localizationlengthfinitevolumerandom}.
Indeed, under the block reduction framework introduced in \eqref{def_S_t}, the setting in those references coincides verbatim with the present one. Consequently, the same arguments apply without modification and yield the estimates \eqref{Gt_bound_flow}–\eqref{Eq:Gdecay_flow}. We therefore omit the details for brevity.

\subsection{Step 6: Expected 2-$G$-loop estimate}

At this stage, we have at our disposal the initial estimate \eqref{Eq:Gtlp_exp+IND} at time $s$, the sharp local law \eqref{Gt_bound_flow}, and the sharp $G$-loop estimates \eqref{Eq:LGxb}, \eqref{Eq:L-KGt-flow}, and \eqref{Eq:Gdecay_flow}. We now apply these inputs to the expectation of \eqref{int_K-LcalE} in order to establish \eqref{Eq:Gtlp_exp_flow}.
We have already proved the sharp averaged local law
\begin{align}\label{eq:res_ELK_n=1}
  \max_{[a]}  \left|\langle (G_u-m) E_{[a]}\rangle\right|
  \prec (W^d\ell_u^d\eta_u)^{-1}  ,
  \end{align}
by applying \eqref{Eq:L-KGt-flow} with $\fn=1$. Our first goal is to improve this bound at the level of expectation:
\begin{align}\label{res_ELK_n=1}
  \max_{[a]}  \left|\mathbb E\langle (G_u-m) E_{[a]}\rangle\right|
  \prec (W^d\ell_u^d\eta_u)^{-2}  .
  \end{align}

We note that the argument based on Gaussian integration by parts used in the proof of Lemma A.4 in \cite{truong2025localizationlengthfinitevolumerandom} is not available in our setting, due to the absence of a block structure in the variance profile. We therefore employ a dynamical approach to prove \eqref{res_ELK_n=1}.
Specifically, taking expectations on both sides of \eqref{eq_L-K-1} with $\fn=1$, we obtain that for any \smash{$\qa{a}\in\Zn$} and $u\in\qa{s,t}$,
\begin{equation}\label{evolution_ELK_fn_1}
    \begin{aligned}
        \frac{\rd}{\rd u}\E\avga{\pa{G_u-m}E_{\qa{a}}}&=W^d\sum_{\qa
        x}\E\avga{\pa{G_u-m}E_{\qa{x}}}\cL_{u,\pa{+,+},\pa{\qa{x},\qa{a}}}^{\pa{2}}\\
        &=W^d\sum_{\qa
        x}\E\avga{\pa{G_u-m}E_{\qa{x}}}\cK_{u,\pa{+,+},\pa{\qa{x},\qa{a}}}^{\pa{2}}+\cR_{\qa{a}}\p{u},
    \end{aligned}
\end{equation}
where the remainder term is defined by
\begin{equation}\nonumber
        \cR_{\qa{a}}\p{u}:=W^d\sum_{\qa
        x}\E\avga{\pa{G_u-m}E_{\qa{x}}}\pa{\cL-\cK}_{u,\pa{+,+},\pa{\qa{x},\qa{a}}}^{\pa{2}}.
\end{equation}
Applying Duhamel’s principle to \eqref{evolution_ELK_fn_1}, we obtain
\begin{equation}\label{integration_ELK_fn_1}
    \begin{aligned}
        \E\avga{\pa{G_u-m}E_{\qa{a}}}=&~\sum_{\qa{x}}\pa{I+\pa{u-s}m^2\Theta_u}_{\qa{a}\qa{x}}\E\avga{\pa{G_s-m}E_{\qa{x}}}\\
        &~+\int_s^{u}\sum_{[x]}\pa{I+\pa{u-v}m^2\Theta_u}_{\qa{a}\qa{x}}\cR_{\qa{x}}\p{v}\,\rd v.
    \end{aligned}
\end{equation}
Using \eqref{Eq:L-KGt-flow} and \eqref{eq:res_ELK_n=1}, we estimate the remainder term as
\begin{equation*}
    \begin{aligned}
        \absa{\cR_{\qa{x}}\p{v}}\prec W^d \ell_v^d\cdot\frac{1}{W^d\ell_v^d\eta_v}\cdot\frac{1}{\pa{W^d\ell_v^d\eta_v}^2}\lesssim \frac{1}{\pa{W^d\ell_u^d\eta_u}^2}\cdot\frac{1}{1-v},
    \end{aligned}
\end{equation*}
where we used $\ell_v^d\eta_v\gtrsim \ell_u^d\eta_u$, and the factor $\ell_v^d$ arises from the exponential decay of the $(\cL-\cK)$-loops by \eqref{Eq:Gdecay+IND}. Substituting this bound into the second term on the RHS of \eqref{integration_ELK_fn_1} yields
\begin{equation*}
    \begin{aligned}
        \int_s^{u}\sum_{[x]}\pa{I+\pa{u-v}m^2\Theta_u}_{\qa{a}\qa{x}}\cR_{\qa{x}}\p{v}\rd v\prec \int_s^{u}\frac{1}{\pa{W^d\ell_u^d\eta_u}^2}\cdot\frac{1}{1-v}\,\rd v\prec \frac{1}{\pa{W^d\ell_u^d\eta_u}^2},
    \end{aligned}
\end{equation*}
where we also used \eqref{prop:ThfadC_short} in the first step. On the other hand, by the assumption \eqref{Eq:Gt_1_lp_exp+IND}, the first term on the RHS of \eqref{integration_ELK_fn_1} can be bounded as
\begin{equation*}
    \begin{aligned}
        \sum_{\qa{x}}\pa{I+\pa{u-s}m^2\Theta_u}_{\qa{a}\qa{x}}\E\avga{\pa{G_s-m}E_{\qa{x}}}\prec \frac{1}{\pa{W^d\ell_s^d\eta_s}^2}\lesssim \frac{1}{\pa{W^d\ell_u^d\eta_u}^2}.
    \end{aligned}
\end{equation*}
Combining these bounds proves \eqref{res_ELK_n=1}.

With \eqref{res_ELK_n=1} established, we can now complete the proof of \eqref{Eq:Gtlp_exp_flow} by exactly the same argument as in the proof of equation (2.80) in \cite{Band1D} or equation (5.18) in \cite{truong2025localizationlengthfinitevolumerandom}. This completes Step 6 in the proof of \Cref{lem:main_ind}.

\appendix

\section{Proof of some auxiliary lemmas}\label{additional_proofs}

In this appendix, we collect two auxiliary results used in the main proofs.

\begin{lemma}[Evolution of primitive loops]\label{lemma_evolution_of_primitive_loops}
For the primitive loops \smash{$\wh \cK_{t,\bsigma,\bx}^{\pa{\fn}}$} defined above \eqref{entrywise_pro_dyncalK}, the following evolution equation holds:
    \begin{align}
       \frac{\dd}{\dd t}\,{\wh\cK}^{(\fn)}_{t, \boldsymbol{\sigma}, \bx}
       =
        \sum_{1\le k < l \le \fn} \sum_{a, b\in \ZL} \pa{ \cutL^{\pa{a}}_{k, l} \circ \wh{\mathcal{K}}^{(\fn)}_{t, \boldsymbol{\sigma}, \bx}}  \cdot \pa{\SE}_{ab}\cdot  \pa{\cutR^{\pa{b}}_{k, l} \circ \wh{\mathcal{K}}^{(\fn)}_{t, \boldsymbol{\sigma}, \bx} } .
    \end{align}
\end{lemma}
\begin{proof}
 We prove the claim by induction on the loop length $\fn$. For $\fn=1$, the statement is immediate from the invariance property \eqref{m_invariant}.
 For $\fn=2$, by the recursive definition \eqref{wh_cK_recursive_relation}, we have
    \begin{equation}\label{expression_2_wh_cK_loop}
        \begin{aligned}            \wh\cK_{t,\pa{\sigma_1,\sigma_2},\pa{x_1,x_2}}^{\pa{2}}=m\p{\sigma_1}m\p{\sigma_2}\pa{1-m\p{\sigma_1}m\p{\sigma_2}S_t}^{-1}_{x_1x_2},
        \end{aligned}
    \end{equation}
which implies \eqref{entrywise_pro_dyncalK} upon differentiation with respect to $t$ and using $\partial_t S_t=\SE$.  Now assume that \eqref{entrywise_pro_dyncalK} holds for all loop lengths $1,2,\ldots,\fn$, and consider a $\cK$-loop of length $\fn+1$.
By \eqref{wh_cK_recursive_relation}, for $\bsigma\in\ha{+,-}^{\fn+1}$, \smash{$\bx\in\p{\Zn}^{\fn+1}$}, and $x\in\ZL$, we have
    \begin{equation}\label{transformed_recursive_relation_wh_cK}
        \begin{aligned}
            \sum_{x_{\fn+1}}\wh \cK_{t,\bsigma,\bx}^{\pa{\fn+1}}\pa{1-m\p{\sigma_1}m\p{\sigma_{\fn+1}}S_t}_{x_{\fn+1}x}=m\p{\sigma_1}\wh \cK_{t,\pa{\sigma_2,\ldots,\sigma_{\fn+1}},\pa{x_2,\ldots,x_\fn,x_1}}^{\pa{\fn}}\delta_{x_{1}x}\\
            +m\p{\sigma_1}\sum_{k=2}^{\fn}\sum_{y}\wh \cK_{t,\pa{\sigma_k,\ldots,\sigma_{\fn+1}},\pa{x_k,\ldots,x_\fn,x}}^{\pa{\fn-k+2}}\pa{S_t}_{xy} \wh \cK_{t,\pa{\sigma_1,\ldots,\sigma_k},\pa{x_1,\ldots,x_{k-1},y}}^{\pa{k}}.
        \end{aligned}
    \end{equation}
Differentiating both sides with respect to $t$, we obtain
   \begin{align*}
            &\sum_{x_{\fn+1}}\partial_t\wh \cK_{t,\bsigma,\bx}^{\pa{\fn+1}}\pa{1-m\p{\sigma_1}m\p{\sigma_{\fn+1}}S_t}_{x_{\fn+1}x}-m\p{\sigma_1}m\p{\sigma_{\fn+1}}\sum_{x_{\fn+1}}\wh \cK_{t,\bsigma,\bx}^{\pa{\fn+1}}\pa{\SE}_{x_{\fn+1}x}\\
            =&~m\p{\sigma_1}\partial_t\wh \cK_{t,\pa{\sigma_2,\ldots,\sigma_{\fn+1}},\pa{x_2,\ldots,x_\fn,x_1}}^{\pa{\fn}}\delta_{x_{1}x}\\
            +&~m\p{\sigma_1}\sum_{k=2}^{\fn}\sum_{y}\wh \cK_{t,\pa{\sigma_k,\ldots,\sigma_{\fn+1}},\pa{x_k,\ldots,x_\fn,x}}^{\pa{\fn-k+2}} \pa{\SE}_{xy} \wh \cK_{t,\pa{\sigma_1,\ldots,\sigma_k},\pa{x_1,\ldots,x_{k-1},y}}^{\pa{k}} \\
            +&~m\p{\sigma_1}\sum_{k=2}^{\fn}\sum_{y} \wh \cK_{t,\pa{\sigma_k,\ldots,\sigma_{\fn+1}},\pa{x_k,\ldots,x_\fn,x}}^{\pa{\fn-k+2}} \pa{S_t}_{xy} \partial_t\wh \cK_{t,\pa{\sigma_1,\ldots,\sigma_k},\pa{x_1,\ldots,x_{k-1},y}}^{\pa{k}} \\
            +&~m\p{\sigma_1}\sum_{k=2}^{\fn}\sum_{y}\partial_t\wh \cK_{t,\pa{\sigma_k,\ldots,\sigma_{\fn+1}},\pa{x_k,\ldots,x_\fn,x}}^{\pa{\fn-k+2}} \pa{S_t}_{xy} \wh \cK_{t,\pa{\sigma_1,\ldots,\sigma_k},\pa{x_1,\ldots,x_{k-1},y}}^{\pa{k}} .
        \end{align*}
    Multiplying both sides on the right by $\p{1-m\p{\sigma_1}m\p{\sigma_{\fn+1}}S_t}^{-1}$, and applying \eqref{expression_2_wh_cK_loop} together with the induction hypothesis to each occurrence of $\partial_t \wh\cK_t$ on the RHS, yields that
    \begin{align*}
        &~\partial_t\wh \cK_{t,\bsigma,\bx}^{\pa{\fn+1}}-\sum_{x,y}\wh \cK_{t,\bsigma,\pa{x_1,\ldots,x_{\fn},y}}^{\pa{\fn+1}}\pa{\SE}_{yx}\wh\cK_{t,\pa{\sigma_{\fn+1},\sigma_1},\pa{x,x_{\fn+1}}}^{\pa{2}}\\
            =&~ m\p{\sigma_1}\sum_{2 \leq r < \ell \leq \fn+1} \sum_{c, d}
          \wh \cK^{(\ell -r + 1)}_{t,\pa{\sigma_r,\ldots,\sigma_\ell},\pa{x_r,\ldots,x_{\ell-1},d}}
        \pa{\SE}_{cd} \wh\cK^{(\fn+r-\ell+1)}_{t,\pa{\sigma_2,\ldots,\sigma_r,\sigma_\ell,\ldots,\sigma_{\fn+1}},\pa{x_2,\ldots,x_{r-1},c,x_\ell,\ldots,x_{\fn},x_1}}\\
        &~\qquad \times \pa{1-m\p{\sigma_1}m\p{\sigma_{\fn+1}}S_t}^{-1}_{x_1x_{\fn+1}}\\
        +&~m\p{\sigma_1}\sum_{k=2}^{\fn}\sum_{c,d}\wh \cK_{t,\pa{\sigma_1,\ldots,\sigma_k},\pa{x_1,\ldots,x_{k-1},d}}^{\pa{k}} \pa{\SE}_{cd}\wh \cK_{t,\pa{\sigma_k,\ldots,\sigma_{\fn+1}},\pa{x_k,\ldots,x_\fn,c}}^{\pa{\fn-k+2}} \pa{1-m\p{\sigma_1}m\p{\sigma_{\fn+1}}S_t}^{-1}_{cx_{\fn+1}}\\
        +&~m\p{\sigma_1}\sum_{k=2}^{\fn}\sum_{1 \leq r < \ell \leq k} \sum_{c, d}\sum_{x,y}
          \wh\cK^{(k+r-\ell+1)}_{t,\pa{\sigma_1,\ldots,\sigma_r,\sigma_\ell,\ldots,\sigma_k},\pa{x_1,\ldots,x_{r-1},c,x_\ell,\ldots,x_{k-1},y}}
        \pa{\SE}_{cd} \wh \cK^{(\ell-r+1)}_{t,\pa{\sigma_r,\ldots,\sigma_\ell},\pa{x_r,\ldots,x_{\ell-1},d}}\\
        &~\qquad\times\pa{S_t}_{xy}\cdot\wh \cK^{(\fn-k+2)}_{t,\pa{\sigma_k,\ldots,\sigma_{\fn+1}},\pa{x_k,\ldots,x_\fn,x}} \pa{1-m\p{\sigma_1}m\p{\sigma_{\fn+1}}S_t}^{-1}_{xx_{\fn+1}}\\
        +&~ m\p{\sigma_1}\sum_{k=2}^{\fn}\sum_{k \leq r < \ell \leq \fn+1} \sum_{c, d}\sum_{x,y}\wh \cK_{t,\pa{\sigma_1,\ldots,\sigma_k},\pa{x_1,\ldots,x_{k-1},y}}^{\pa{k}}\pa{S_t}_{xy}\wh\cK^{(\fn-\ell+r-k+3)}_{t,\pa{\sigma_k,\ldots,\sigma_r,\sigma_\ell,\ldots,\sigma_{\fn+1}},\pa{x_k,\ldots,x_{r-1},c,x_\ell,\ldots,x_{\fn},x}}\\
        &~\qquad\times
        \pa{\SE}_{cd} \wh \cK^{(\ell-r+1)}_{t,\pa{\sigma_r,\ldots,\sigma_\ell},\pa{x_r,\ldots,x_{\ell-1},d}}\pa{1-m\p{\sigma_1}m\p{\sigma_{\fn+1}}S_t}^{-1}_{xx_{\fn+1}}
         .
        \end{align*}
        We now apply the recursive relation \eqref{wh_cK_recursive_relation} and obtain
        \begin{align*}
       & \partial_t\wh \cK_{t,\bsigma,\bx}^{\pa{\fn+1}}=\sum_{y,w}\wh \cK_{t,\bsigma,\pa{x_1,\ldots,x_{\fn},y}}^{\pa{\fn+1}}\pa{\SE}_{yw}\wh\cK_{t,\pa{\sigma_{\fn+1},\sigma_1},\pa{w,x_{\fn+1}}}^{\pa{2}}\\
        &+  \sum_{2 \leq r < \ell \leq \fn+1} \sum_{c, d}
          \wh\cK^{(\fn+r-\ell+2)}_{t,\pa{\sigma_1,\ldots,\sigma_r,\sigma_\ell,\ldots,\sigma_{\fn+1}},\pa{x_1,\ldots,x_{r-1},c,x_\ell,\ldots,x_{\fn+1}}}
        \pa{\SE}_{cd} \wh \cK^{(\ell -r + 1)}_{t,\pa{\sigma_r,\ldots,\sigma_\ell},\pa{x_r,\ldots,x_{\ell-1},d}} \\
        &+ \sum_{1\le r < \ell \le \fn}\sum_{c, d} \cK_{t,\pa{\sigma_1,\sigma_k,\ldots,\sigma_{\fn+1}},\pa{c,x_k,\ldots,x_\fn,x_{\fn+1}}}^{\pa{\fn-k+3}} \pa{\SE}_{cd} \wh \cK_{t,\pa{\sigma_1,\ldots,\sigma_k},\pa{x_1,\ldots,x_{k-1},d}}^{\pa{k}}  .
        \end{align*}
        Together with the definitions of the operators $\cutL$ and $\cutR$, this identity can be rewritten in the compact form
        \begin{equation*}
            \begin{aligned}
              \partial_t{\wh\cK}^{(\fn+1)}_{t, \boldsymbol{\sigma}, \bx}
       =
         \sum_{1\le k < l \le \fn+1} \sum_{a, b\in \ZL}  \pa{\cutL^{\pa{c}}_{k, l} \circ \wh{\mathcal{K}}^{(\fn)}_{t, \boldsymbol{\sigma}, \bx} } \cdot \pa{\SE}_{cd}\cdot  \pa{\cutR^{\pa{d}}_{k, l} \circ \wh{\mathcal{K}}^{(\fn)}_{t, \boldsymbol{\sigma}, \bx}}.
            \end{aligned}
        \end{equation*}
        This completes the induction step and hence concludes the proof of \eqref{entrywise_pro_dyncalK}.
\end{proof}

\begin{lemma}\label{lemma_mean_field_inverse_bound_point_wise}
Consider a symmetric $W^d\times W^d$ matrix $A$ with nonnegative entries. Suppose that for some constant $c\in\pa{0,1}$,
    \begin{equation*}
            c W^{-d}\leq A_{ij}\leq c^{-1}W^{-d},\quad  1\le i,j\le W^d.
    \end{equation*}
 For each $1\le i\le W^d$, denote $r_i:=\sum_{j}A_{ij}$, and assume that $\max_i r_i< 1$ and $a:=W^{-d}\sum_{i}r_i<1$. Then, $1-A$ is invertible, and there exists a constant $C=C(c)>0$ such that
    \begin{equation}\label{mean_field_inverse_bound_point_wise}
        \begin{aligned}
         \pa{1-A}^{-1}_{ij}\leq C\pa{\delta_{ij}+ \frac{1}{W^d}\frac{1}{1-a}},\quad \forall \  1\le i,j\le W^d.
        \end{aligned}
    \end{equation}
\end{lemma}
\begin{proof}
By the Perron–Frobenius theorem and the assumption $\max_i r_i< 1$, the spectral radius of $A$ is strictly smaller than 1. Hence, $1-A$ is invertible. To prove \eqref{mean_field_inverse_bound_point_wise}, we use the identity
    \begin{equation}\label{A_average_to_pointwise}
       \pa{1-A}^{-1}_{ij}=\qa{1+A+A\pa{1-A}^{-1}A}_{ij}\leq \delta_{ij}+c^{-1}\frac{1}{W^d}+c^{-2}\frac{1}{W^{2d}}\sum_{i,j}\pa{1-A}^{-1}_{ij}.
    \end{equation}
    Therefore, it suffices to bound $\mathbf{1}^{\top}\pa{1-A}^{-1}\mathbf{1}$, where $\mathbf{1}$ denotes the vector with all entries equal to 1. Set $\bu:=\pa{1-A}^{-1}\mathbf{1}$. Using the identity $(1-A)^{-1}=1+A(1-A)^{-1}$, we compute
    \begin{equation*}
        \begin{aligned}
            \sum_{i}u_i= W^d+\sum_{i,j}A_{ij}u_j =W^d+\sum_{j}r_ju_j.
        \end{aligned}
    \end{equation*}
    Rearranging gives
    \begin{equation}\label{bu_bound_min}
        W^d= \sum_{j}\pa{1-r_j}u_j \ge  W^d\pa{1-a}\min_ju_j.
    \end{equation}
    On the other hand, for any $1\le i\le W^d$, we have
    \begin{equation}\label{bu_min_to_average}
            \sum_j u_j\leq c^{-1}W^d\sum_{j}A_{ij} u_j = c^{-1}W^d (A\pa{1-A}^{-1}\mathbf{1})_i =c^{-1}W^{d}\pa{u_{i}-1}\leq c^{-1}W^d u_i.
    \end{equation}
    Taking the minimum over $i$ and using \eqref{bu_bound_min}, we obtain
    \begin{equation*}
        \begin{aligned}
         \sum_{j} u_j\leq c^{-1}W^{d}\min_i u_i\leq c^{-1}{W^d}/\p{1-a}.
        \end{aligned}
    \end{equation*}
     Plugging this estimate back into \eqref{A_average_to_pointwise}, we conclude that for some constant $C=C(c)>0$,
    \begin{equation*}
        \begin{aligned}
        \pa{1-A}^{-1}_{ij}\leq \delta_{ij}+c^{-1}\frac{1}{W^d}+c^{-2}\frac{1}{W^{2d}}\sum_{i}u_i \le   \delta_{ij}+\frac{1}{W^d}\frac{2c^{-3}}{1-a},
        \end{aligned}
    \end{equation*}
which proves \eqref{mean_field_inverse_bound_point_wise}.
\end{proof}

\end{document}